\documentclass{TPmod}
\DeclareMathAlphabet{\mathpzc}{OT1}{pzc}{m}{it}
\usepackage{url}
\usepackage{setspace}
\usepackage{scrextend}
\usepackage{mathrsfs}
\usepackage{tocloft}
\usepackage{empheq}
\usepackage{pigpen}
\usepackage{tikz}

\newcommand{\po}{\ar@{}[dr]|{\text{\pigpenfont R}}}
\newcommand{\pb}{\ar@{}[dr]|{\text{\pigpenfont J}}}

\usetikzlibrary{decorations.markings,arrows}
\usepackage{xcolor}

\mathchardef\mhyphen="2D 

\makeatletter
\DeclareRobustCommand{\hvec}[1]{{\mathpalette\hvec@{#1}}}
\newcommand{\hvec@}[2]{
  \vbox{\offinterlineskip
    \ialign{
      \hfil##\hfil\cr
      $\m@th#1{}_{\rightharpoonup}$\kern-\scriptspace\cr
      $\m@th#1#2$\cr
    }
  }
}
\makeatother

\usepackage{graphicx,calc}
\newlength\myheight
\newlength\mydepth
\settototalheight\myheight{Xygp}
\usepackage{amsthm}
\usepackage{amssymb}
\usepackage[capitalize]{cleveref}
\usepackage{mathtools}
\usepackage{xy}
\usepackage{tikz}
\usepackage{tikz-cd}

\usepackage{cancel}

\usepackage[normalem]{ulem}
\usetikzlibrary{shapes.geometric,arrows}
\usetikzlibrary{calc, arrows,decorations.pathmorphing,backgrounds,fit,positioning,shapes.symbols,chains,snakes}
\usetikzlibrary{decorations.pathreplacing,decorations.markings}
\usetikzlibrary{decorations.pathreplacing,calligraphy}

\usetikzlibrary{arrows.meta}

\DeclareMathOperator{\Aut}{Aut}

\newcommand{\Tot}{\mathrm{Tot}}

\newcommand{\Sym}{\mathrm{Sym}}
\newcommand{\Cptd}{\operatorname{Cpt}\partial}
\newcommand{\Cpt}{\operatorname{Cpt}}

\newcommand{\triv}{\mathrm{triv}}
\newcommand{\st}{\mathrm{st}}

\newcommand{\hocolim}{\mathrm{hocolim}}
\newcommand{\colim}{\mathrm{colim}}
\newcommand{\hofib}{\mathrm{hofib}}

\usepackage{xcolor}

\DeclareMathOperator{\Sing}{\operatorname{Sing}}

\DeclareMathOperator{\loc}{loc}

\newcommand{\pt}{\operatorname{pt}}

\def\bB{\mathbb{B}}
\def\bC{\mathbb{C}}
\def\bD{\mathbb{D}}
\def\bE{\mathbb{E}}
\def\bF{\mathbb{F}}

\def\bP{\mathbb{P}}

\def\bR{\mathbb{R}}
\def\bS{\mathbb{S}}

\def\bU{\mathbb{U}}
\def\bV{\mathbb{V}}

\def\bZ{\mathbb{Z}}

\def\BB{\prescript{}2B}
\def\ee{\prescript{}2e}
\def\oBB{\prescript{}2{\overline B}}
\def\XX{\prescript{}2\Xi}

\DeclareMathOperator{\Thom}{\mathsf{Thom}}
\DeclareMathOperator{\Base}{\mathsf{Base}}
\DeclareMathOperator{\sSet}{\mathsf{sSet}}
\DeclareMathOperator{\Vect}{\mathsf{Vect}}

\newcommand{\eps}{\varepsilon}

\newcommand{\cA}{\mathcal{A}}
\newcommand{\cB}{\mathcal{B}}
\newcommand{\cC}{\mathcal{C}}
\newcommand{\cD}{\mathcal{D}}
\newcommand{\cE}{\mathcal{E}}
\newcommand{\cF}{\mathcal{F}}
\newcommand{\cG}{\mathcal{G}}
\newcommand{\cH}{\mathcal{H}}
\newcommand{\cI}{\mathcal{I}}

\newcommand{\cL}{\mathcal{L}}
\newcommand{\cM}{\mathcal{M}}
\newcommand{\cN}{\mathcal{N}}
\newcommand{\cO}{\mathcal{O}}
\newcommand{\cP}{\mathcal{P}}
\newcommand{\cQ}{\mathcal{Q}}
\newcommand{\cR}{\mathcal{R}}
\newcommand{\cS}{\mathcal{S}}

\newcommand{\cV}{\mathcal{V}}
\newcommand{\cW}{\mathcal{W}}
\newcommand{\cX}{\mathcal{X}}
\newcommand{\cY}{\mathcal{Y}}

\def\bZ{\mathbb{Z}}

\tikzset{
    ncbar angle/.initial=90,
    ncbar/.style={
        to path=(\tikztostart)
        -- ($(\tikztostart)!#1!\pgfkeysvalueof{/tikz/ncbar angle}:(\tikztotarget)$)
        -- ($(\tikztotarget)!($(\tikztostart)!#1!\pgfkeysvalueof{/tikz/ncbar angle}:(\tikztotarget)$)!\pgfkeysvalueof{/tikz/ncbar angle}:(\tikztostart)$)
        -- (\tikztotarget)
    },
    ncbar/.default=0.5cm,
}

\tikzset{square left brace/.style={ncbar=0.5cm}}
\tikzset{square right brace/.style={ncbar=-0.5cm}}

\tikzset{round left paren/.style={ncbar=0.5cm,out=120,in=-120}}
\tikzset{round right paren/.style={ncbar=0.5cm,out=60,in=-60}}

\renewcommand{\cong}{\simeq}

\numberwithin{equation}{section} 
\renewcommand{\theequation}{\thesection.\arabic{equation}}

\begin{document}

\title{Open-closed maps and spectral local systems}

\author[]{Noah Porcelli, Ivan Smith}
\thanks{The Clynelish distillery.}
\date{September 2024. Unauthorised readers will be pestered by pufflings.}

\address{Noah Porcelli, Max Planck Institute for Mathematics, Vivatsgasse 7, 53111 Bonn, Germany}
\address{Ivan Smith, Centre for Mathematical Sciences, University of Cambridge, Wilberforce Road, Cambridge CB3 0WB, U.K.}


\begin{abstract} {\sc Abstract:}
Let $X$ be a graded Liouville domain. Fix a pair of infinite loop spaces $\Psi = (\Theta \to \Phi)$ living over $(BO \to BU)$. This determines a spectral Fukaya category $\scrF(X;\Psi)$ whenever $TX$ lifts to $\Phi$, containing closed exact Lagrangians $L$ for which $TL$ lifts compatibly to $\Theta$; and by Bott periodicity and index theory, a Thom spectrum $R$ with bordism theory $R_*$.

This paper has two main goals: we incorporate rank one spectral local systems $\xi: L \to BGL_1(R)$ into the spectral category; and we prove that the bordism class $[(L,\xi)]$ defined by the open-closed map differs from the class $[L]$ by a multiplicative two-torsion element in $R^0(L)^{\times}$ determined by an action of the stable homotopy class of the Hopf map $\eta \in \pi_1^{st}$ on $\xi$.   Methods include a twisting construction associating flow categories to spectral local systems, and a model for the open-closed map incorporating Schlichtkrull's construction of the trace map $BGL_1(R) \subseteq K(R) \to R$.

The companion paper \cite{PS4} shows that (for Lagrangians which themselves admit spectral lifts) one can lift quasi-isomorphisms from $\bZ$ to $\Psi$ at the cost of introducing rank one local systems. Together with the open-closed computation given here, this gives an essentially complete picture of the bordism-theoretic consequences of quasi-isomorphism in the classical exact Fukaya category.
\end{abstract}

\maketitle
\thispagestyle{empty}
\setlength{\cftbeforesubsecskip}{-2pt}
\tableofcontents

\section{Introduction}

  \subsection{Context}

Throughout symplectic topology and Floer theory, one studies Lagrangian submanifolds equipped with local systems, usually of rank one. Allowing local systems is essential from many viewpoints: rank one $\bZ$-local systems (which are closely related to $Spin$ structures) allow one to incorporate changes in orientation signs; rank one $\bC^*$-local systems underpin a deep relation to cluster theory; Novikov-unitary local systems are at the heart of the `SYZ viewpoint' on mirror symmetry, and the construction of mirror spaces.  

A well-behaved (`unobstructed') Lagrangian submanifold has a fundamental class, in general living in quantum homology,   defined via the open-closed map as a count of perturbed holomorphic discs with boundary on $L$.  An important fact is that in the exact -- which implies unobstructed -- case, the fundamental class  of a Lagrangian equipped with a local system is the same as its classical fundamental class with no local system (this changes, in general, in the presence of holomorphic discs).

This paper has two goals:
\begin{enumerate}
    \item We incorporate rank one spectral local systems into exact spectral Floer theory; more precisely, we show how to `twist' the Floer flow category associated to $L$ and $K$ when either carries a spectral local system.  
    \item We compute the corresponding open-closed bordism-valued fundamental class. Strikingly this \emph{does} now depend on the choice of local system, with an essentially universal dependence governed by the stable homotopy class of the Hopf map $\eta \in \pi_1^{st}$. It follows that the dependence of the bordism class on the local system is concentrated at the prime $2$.
\end{enumerate}

The appearance of $\eta$ is not just a curiosity. Numerous applications of this result to classical symplectic topology questions are given in \cite{PS4}; in particular, we show that the normal invariant of a nearby Lagrangian factors through $\eta$ (indeed, we obtain analogous constraints which apply even outside the case of cotangent bundles).

   Our restriction to rank one spectral local systems is motivated by both the ubiquity of rank one local systems in classical Floer theory, and the main theorem of the sequel \cite{PS4}, which shows that if $L$ and $K$ themselves admit lifts to the spectral Fukaya category, then a quasi-isomorphism of $L$ and $K$ over $\bZ$ lifts to a quasi-isomorphism of $(L,\xi)$ and $K$ over $\Psi$, for some rank one spectral local system $\xi$. 
 Together with the computation of the open-closed map in the presence of the local system given here, this gives essentially complete control of the consequences in  bordism of knowing that two Lagrangians are quasi-isomorphic in the integral Fukaya category. 
 
 \subsection{Results}

A precise statement of our results is most naturally given in the framework of the spectral Fukaya category.  Recall that 
Abouzaid and Blumberg \cite{AB} have constructed a stable infinity-category of framed flow categories, which is equivalent to the category of spectra, with a view to creating `Floer homotopy theory without homotopy theory'. In this worldview, bordism takes centre stage: morphism groups in the spectral Fukaya category are bordism classes of flow modules.  More classical viewpoints on Floer homotopy theory in the vein of work of Cohen, Jones and Segal live on the `other side' of the Pontrjagin-Thom isomorphism. 

 In \cite{PS}, building on work of \cite{Large,FO3:smoothness}, we constructed such a `bordism-flavoured'  spectral Fukaya category for a stably framed Liouville manifold, with objects being compatibly stably framed compact exact Lagrangian submanifolds. The sequel \cite{PS2} constructed an analogous spectral Fukaya category for any tangential structure $\Psi = (\Theta \to \Phi)$ living over $BO \to BU$, thus allowing us to access information contained in more general bordism theories. Here the symplectic manifold $X$ is assumed to have a lift of the classifying map for $TX$ to $\Phi$, and objects of the category are compact exact Lagrangians equipped with lifts of their tangent bundles to $\Theta$. (Both \cite{Large} and \cite{ADP} also construct versions of exact spectral Fukaya categories, but both pass via Cohen-Jones-Segal from bordism to spectra and homotopy types.)

To formulate the main result of this paper requires revisiting the construction of the spectral Fukaya category $\scrF(X;\Psi)$ in the special case where the tangential pair $(\Theta \to \Phi)$ lives over $(BO \to BU)$ as a diagram of infinite loop spaces. The appearance of $\eta$ reflects a trace map from algebraic $K$-theory of a ring spectrum to the original ring spectrum which only exists in the commutative setting.  Indeed, even before computing the associated fundamental class, our way of incorporating spectral local systems into the Fukaya category itself makes essential use of commutativity, and doesn't seem to have any obvious analogue in any of the other current approaches to Floer homotopy theory.

Thus, the version of the Fukaya category we work with in this paper uses \emph{commutative} tangential pairs, which amounts to asking that 
\[
\xymatrix{
\Theta \ar[r] \ar[d] & \Phi \ar[d] \\
BO \ar[r] & BU
}
\]
be a diagram of infinite loop spaces\footnote{In fact, we impose grading and orientation requirements that mean that our tangential structures more naturally live over $BSpin \to BS_{\pm}U$, with the latter being the fibre over $\pm 1$ of the  determinant map; indeed, we work with a slightly more general notion of oriented tangential pair, see  Definition \ref{def: comm tang pair}.}.  Let $F$ be the homotopy fibre of the map $\Theta \to \Phi$. Using real Bott periodicity, we may define a commutative ring spectrum $R_\Psi$, as the Thom spectrum of an index bundle defined by the composition $\Omega F \to \Omega\widetilde{U/O} \xrightarrow{\mathrm{Bott}} BO$, see Section \ref{sec:abstract discs} for a detailed construction. Its base, which we write as $E_\Psi$, is homotopy equivalent to $\Omega F$.

\begin{ex}
    We write $fr$ for the tangential pair $\Psi=(\pt, \pt)$, corresponding to the case of framed symplectic manifold and Lagrangians. In this case, $R_\Psi$ is equivalent to the sphere spectrum. If $\Psi=(\pt, BS_\pm U)$, $R_\Psi$ is the complex cobordism ring spectrum $MU$.
\end{ex}

Similarly to the case of an ordinary ring, $R_\Psi$ has a group of units $GL_1(R_\Psi)$, and we may model $R_\Psi$-local systems over a base $B$ as maps $B \to BGL_1(R_\Psi)$. We call these \emph{$\Psi$-local systems}. Through the Pontrjagin-Thom isomorphism, the Thom ring spectrum $R = R_{\Psi}$ defines an associated bordism theory $\Omega_*^{E_{\Psi}}(\cdot)$. 

To a pair $\Psi=(\Theta, \Phi)$ as above and a Liouville domain $X$ with a lift $\phi: X \to \Phi$ of the classifying map of $TX$, we associate a spectral Fukaya category $\scrF(X; \Psi)$. This has objects pairs $(L, \xi_L)$, where $L$ is a closed exact Lagrangian with a compatible lift of $TL$ to $\Theta$, and $\xi_L$ is a $\Psi$-local system. The morphism groups are linear over the ring $\Omega_*^{E_{\Psi}}$. Our construction here incorporates the infinite loop structure of $\Psi$ in a way that was absent from \cite{PS2}.

\begin{rmk}
    In the ordinary ($\bZ$-linear) Fukaya category, objects are closed exact Lagrangians $L$, equipped with a (relative) Spin structure $\mathfrak{s}$ and a $\bZ$-local system $\xi$. However, the local system is in some sense redundant: any $(L, \mathfrak{s}, \xi)$ is isomorphic to $L$ equipped with a different Spin structure and no local system. This is not the case in our spectral Fukaya category; the incorporation of a local system can't be absorbed by a change in the tangential structure on the underlying Lagrangian brane.
\end{rmk}

Let $\phi$ denote the $\Phi$-orientation on $X$. There is a (open-closed or `$\cO\cC$') bordism theory $\Omega_*^{E_{\Psi},\cO\cC}(X,\phi)$, which is a module over $R^{-*}(X)$  by cap-product, and which is a variant of $\Omega_*^{E_{\Psi}}(X)$; see Section \ref{sec:OC definitions} for the precise definition, and Example \ref{ex:OC bordism maps here} for relations to more familiar theories. This is 
the `correct' place for the open-closed map to land in this setting: for a $\Theta$-oriented Lagrangian $L$, or pair $(L,\xi_L)$ of such a Lagrangian with a spectral local system, the open-closed map defines a class $[(L,\xi_L)] \in \Omega_d^{E_{\Psi},\cO\cC}(X,\phi)$. The definition of this class in principle uses moduli spaces of holomorphic curves, and is not \emph{a priori} manifestly topological. Our main result below gives a topological description of how this class depends on $\xi_L$.

\begin{ex}
    If $\Psi=fr$, $\Omega^{E_\Psi, \cO\cC}_*(X, \phi)$ maps to the framed bordism groups of $X$. If $\Psi = (\pt, BS_\pm U)$, $\Omega^{E_\Psi,\cO\cC}_*(X,\phi)$ maps to the complex cobordism groups of $X$.
\end{ex}

The fact that $R$ is commutative means that the stable homotopy class $\eta \in \pi_1^{st}$ defines  a distinguished subgroup of the group of units $R^0(L)^\times$, comprising classes that factor through a map $\eta^*: BGL_1(R) \to GL_1(R)$ induced by $\eta$. Here we are using that, similarly to the case of ordinary rings, there is an inclusion $GL_1(R) \to R$. Slightly abusively we call this subgroup the `image of $\eta$'; it is composed of multiplicatively two-torsion elements (i.e. they square to 1).

\begin{thm}[{Theorem \ref{thm: tech OC main}}] \label{thm:main2}
    Let $\Psi$ be a commutative tangential pair, and $X$ a Liouville domain with $\Phi$-orientation $\phi$. Let $L \subseteq X$ be a closed exact Lagrangian equipped with a compatible $\Theta$-orientation, and $\xi_L$ a $\Psi$-local system on $L$ (thus $L$ and $(L, \xi_L)$ are objects of the spectral Fukaya category $\scrF(X; \Psi)$). Write $[L], [(L, \xi_L)] \in \Omega_*^{E_\Psi;\cO\cC}(X,\phi)$ for their open-closed fundamental classes.

    Then we have: $[(L, \xi_L)] = [L] \cap [\eta\xi_L]$, where $[\eta\xi_L] \in R^0(L)$ lies in the image of $\eta$.
\end{thm}
The class $[\eta\xi_L]$ is determined by $\xi_L$ explicitly, cf. the full statement in Theorem \ref{thm: tech OC main}.
\begin{cor}\label{2}
With the notation of Theorem \ref{thm:main2}: 
    \begin{enumerate}
        \item The restriction of the class $[\eta\xi_L]$ to any point in $L$ represents the class $+1 \in R^0(\pt)$.

        \item The difference $[(L, \xi_L)]-[L]=0$ after inverting 2, i.e. it is $2^l$-torsion for some $l$. 
        \item If $L$ is a homotopy sphere, $[(L, \xi_L)]-[L]$ is 2-torsion.
    \end{enumerate}
\end{cor}
See Section \ref{proof 2} for the proof of Corollary \ref{2}.

\begin{rmk}\label{rmk: eoirkhgourhg}
    The class $[L] \cap [\eta\xi]$ is in general not equal to $[L]$. Concretely, if $R_\Psi \cong \bS$ and $L=S^n$, $[L, BGL_1(R_\Psi)] \cong \pi_{n-1}\bS$, and if we choose a class $\xi \in \pi_{n-1}\bS$ which is not annihilated by multiplication by $\eta$, we obtain such an example (such classes exist when $(n-1) = 1,2,3,7,8,9,\dots$).

    This is in contrast to the case of (exact) Fukaya categories over $\bZ$, where the image of a Lagrangian brane under the open-closed map is independent of the local system. In light of Theorem \ref{thm:main2}, morally this is because $\eta$ acts trivially on ordinary rings.
\end{rmk}

Theorem \ref{thm:main2} is obtained using work of Schlichtkrull \cite{Schlichtkrull}. Schlichtkrull proved that for a commutative ring spectrum $R$, there is a universal identity relating the trace map on algebraic $K$-theory of $R$, when restricted to rank-one local systems, to the class $\eta$. This further reflects the expected behaviour of the open-closed map as a kind of `trace'. Commutativity is required here even in stating this identity.

\begin{rmk}
    Philosophically $\eta$ should also arise in \cite{AAGCKarxiv} (which, however, is not directly about either open-closed maps or local systems). Theorem \ref{thm:main2} essentially confirms, in our setting, the speculation in that direction given in the Introduction of \cite{AAGCKarxiv}; the result is accessible here because we keep track of more commutativity than  \emph{op. cit.}.  Amongst the applications of the main results given in \cite{PS4} (which also use commutativity in an essential way), those related to normal invariants of nearby Lagrangians in cotangent bundles can be seen as mild refinements of the results of \cite{AAGCKarxiv}, where the refinement arises because $\eta$ is visible rather than in the background.
\end{rmk}

\subsection{Methods of proof}

\subsubsection*{Construction of the Fukaya category}
    To incorporate commutativity of the tangential pair into the Fukaya category, we prove that all relevant moduli spaces $\cM$ admit coherent $\Psi$-orientations, namely a lift of their stable tangent bundle to $E_\Psi$. In \cite{PS2}, we constructed related tangential structures on such moduli spaces, but these were not a priori compatible with the commutativity of the tangential pair: heuristically, the point of view in \cite{PS2} was that breaking of moduli spaces corresponds to concatenation of loops in $\Omega U/O$ whereas the natural way to construct an infinite loop structure on $\Omega U/O$ uses direct sum in $U/O$ and ignores the loop factor. In Section \ref{sec: fuk no loc} we carefully carry out an explicit Eckmann-Hilton-type argument to bridge the gap between these two notions.
    
    Some similar ideas appear in work of Bonciocat \cite{Bonciocat} in the setting of moduli spaces of strips.

 To incorporate $\Psi$-local systems into the spectral Fukaya category, we `twist' the flow category $\cM^{LK}$ associated to a pair of $\Psi$-branes $L, K$ equipped with local systems $\xi_L, \xi_K$ to produce a new flow category $\cM^{\xi_L,\xi_K}$. The data of the local systems is used to construct maps from the moduli spaces $\cM^{LK}_{xy}$ into Thom spaces of finite rank index bundles over spaces of abstract holomorphic caps, and the twisted flow category moduli spaces are given by zero-sets of these maps.

    Central to our technical results (both here and in \cite{PS4}) is the notion of a flow category \emph{over} a manifold $L$ (possibly equipped with an $E_\Psi$-orientation, cf. Definition \ref{defn:E-oriented flow cat}). This consists of a flow category $\cF$, along with maps of spaces $ev: \cF_{xy} \to \Omega L$, compatible with concatenation on both sides. These arise naturally in Floer theory: the flow category $\cM^{LK}$ associated to a pair of Lagrangians has such a map, essentially by restricting a holomorphic strip to one of its boundary components. A similar construction has appeared in \cite{Barraud-Cornea,Abouzaid:based_loops, BDHO}. 
    
There is a category of ($E_{\Psi}$-oriented) flow categories over $L$; in the sequel \cite{PS4} we further exploit these ideas to construct a Viterbo-style restriction functor
    \[
    \scrF(X;\Psi) \longrightarrow {\Flow}^{E_{\Psi}}_{/L}.
    \]

\begin{rmk}
   There is also a truncation of $\Flow_{/L}$ which recovers $C_*(\Omega L)\mhyphen$mod, making contact with well known ideas of  Floer theory with `universal local systems', cf. for instance \cite{BDHO, BDHO2}.
\end{rmk}

\subsubsection*{Algebraic $K$-theory and the Hopf map}
    For the ordinary ($A_\infty$) Fukaya category, the \emph{open-closed} map is a map from Hochschild homology to ordinary homology, $\cO\cC: HH_*^\bZ(\scrF(X; \bZ)) \to H_{*+d}(X)$. From the TFT perspective, this can be viewed as composing the unit morphism of an object with a `co-unit'-type moduli space of caps, as arises in trace maps in more abstract frameworks such as \cite{HSS}. 
    
    More concretely, the ordinary Fukaya category can be enlarged to contain Lagrangians with local systems of arbitrary rank, as is done in \cite{Konstantinov}. We may further restrict to the full subcategory $\scrF^L(X)$ of arbitrary-rank local systems supported on a single Lagrangian $L$. Then the open-closed map $HH_*^\bZ(\scrF^L(X; \bZ)) \to H_{*+d}(X)$ can be seen to be a form of trace map more directly; for example in the construction of the open-closed map with higher-rank $\bZ$-local systems in \cite[Section 3.2.2]{Konstantinov}, one takes traces of parallel transport matrices explicitly.

    In the spectral setting, the spectral open-closed map is expected to be a map  \[
    \cO\cC: \pi_*THH^{R_\Psi}(\scrF(X; \Psi)) \to \Omega^{E_\Psi, \cO\cC}_*(X, \phi)\]
    from (relative) topological Hochschild homology, the `spectral analogue' of ordinary Hochschild homology. 
    Amongst other things, this would require an $A_\infty$ version of our spectral Fukaya category, which has not yet been built. 

    Recall that algebraic $K$-theory $K(R)$ of a ring spectrum $R$ is obtained from the colimit $\lim_{k \to \infty} BGL_k(R)$ by applying Quillen's plus construction. For $R$ commutative, there is a standard map from $K$-theory, $K(R) \to \Omega^\infty R$. 
    
    Crucially, Schlichtkrull computed the composition of these standard maps when restricted to $BGL_1(R) \subseteq K(R)$, surprisingly finding an instance of the Hopf map $\eta$:
    \begin{thm}[{\cite{Schlichtkrull}}]\label{thm:Schlichtkrull}
        Let $R$ be a commutative ring spectrum. The restriction of the trace map to $BGL_1(R)$ is determined by the Hopf map:
        \begin{equation}
            \xymatrix{
                BGL_1(R) 
                \ar[r]_{\subseteq}
                \ar[d]_{\eta_*} 
                &
                K(R)
                \ar[d]
                \\
                GL_1(R)
                \ar[r]_\subseteq
                &
                \Omega^\infty R
            }
        \end{equation}
        commutes up to homotopy.
    \end{thm}

    This is the motivation for $\eta$ arising in the study of spectral local systems and the open-closed map. Our model of the open closed map is constructed similarly to some model for the map $K(R) \to \Omega^\infty R$, and to prove Theorem \ref{thm:main2} in Section \ref{sec: OC} we pass through several cyclic bar-type constructions to be able to compare to the models for $\eta$ arising in \cite{Schlichtkrull}.

    \begin{rmk}
        For $R$ an ordinary ring, $\eta_*: BGL_1(R) \to GL_1(R)$ is nullhomotopic, reflecting Remark \ref{rmk: eoirkhgourhg}.
    \end{rmk}

\subsubsection*{Speculation on higher ranks}
    One should also be able to incorporate higher rank spectral local systems on a Lagrangian $L$ into the spectral Fukaya category. These can be modelled as maps $\xi_L: L \to BGL_n(R_\Psi)$ for $n > 1$ (or, more generally, maps $\xi_L: L \to B\Aut(M)$ for $M$ a perfect $R_\Psi$-module). The natural generalisation of Theorem \ref{thm:main2} would be that $[(L, \xi)] = [L] \cap [Tr\,\xi_L]$, where $Tr \,\xi_L$ is the composition of the map $BGL_n(R) \subseteq K(R) \to R$ and $\xi_L: L \to BGL_n(R)$. 

    Unlike Schlichtkrull's full computation in the $n=1$ case, we are not aware of any general computations of the map $BGL_n(R) \to R$ for $n > 1$, though there is has been some study of a related filtration on algebraic $K$-theory \cite{Rognes}. 

    Consider a plumbing of two closed manifolds along a point, $X=T^*Q_1 \# T^*Q_2$. \cite{AS:plumbings} shows that, once $\dim(Q_i) \geq 3$,  any closed exact graded Spin Lagrangian $L \subseteq X$ is generated (in the integral Fukaya category) by local systems over the core components. Since trace maps from $K$-theory send exact sequences to sums, one may expect that the open-closed map sends $L$ to a sum of classes of the form $[(Q_i, \xi)]$, for $\xi$ a (possibly high-rank) spectral local system on $Q_i$. Combined with higher-rank generalisations of Schlichtkrull's computation, this could be used to constrain bordism classes of Lagrangians in plumbings.

\noindent \textbf{\emph{Acknowledgements.}} N.P. is supported by EPSRC standard grant EP/W015889/1.  I.S. is partially supported by UKRI Frontier Research grant EP/X030660/1 (in lieu of an ERC Advanced grant).  We are grateful to Mohammed Abouzaid, Daniel Álvarez-Gavela, Sylvain Courte, Jeremy Hahn, Alice Hedenlund, Thomas Kragh, Alex Oancea, Oscar Randal-Williams, John Rognes for helpful discussions. N.P. is grateful to the Max Planck Institute for Mathematics in Bonn for its hospitality and financial support.

\section{Homotopical preliminaries}\label{sec: htpy prel}

Our main results make essential use of commutativity of tangential pairs (as we have explained in the introduction). A convenient model for infinite loop spaces is commutative $\cI$-monoids, which are a special kind of $\cI$-space. An $\cI$-space is a family of spaces indexed by elements of a set $\cI$ (e.g. the natural numbers), with good functoriality properties under inclusions of subsets (e.g. maps $[n] \to [m]$). For instance, $BO$ can be constructed as an infinite loop space in many ways; the commutative $\cI$-monoid structure does this using maps $\oplus: BO(n) \times BO(m) \to BO(n+m)$ and stabilisations $BO(n) \to BO(n')$. The `underlying’ space one is naively thinking of is built as a homotopy colimit, but having access to the `finite levels’ of these spaces is very natural in symplectic topology (where we stabilise Cauchy-Riemann operators or complex vector bundles by finite-dimensional spaces, even if one has to pass to larger and larger such spaces to see the underlying homotopy type emerge).  Working with $\cI$-spaces is also convenient for enabling us to directly leverage the results of  \cite{Schlichtkrull} when proving Theorem \ref{thm:main2}. 

Whilst $\cI$-spaces permeate the paper, some of the more delicate homotopical  preliminaries on bar constructions, $\Gamma$-spaces, etc are only used essentially in Section \ref{sec: OC}.

\subsection{Homotopy colimits}\label{sec: hocolim}

    Let $Spc$ be the category of (compactly generated, Hausdorff) topological spaces, and $Spc_*$ the category of (well-)based such spaces. We write $Spc_{(*)}$ to mean either of these two categories.
    \begin{defn}[{\cite{Bousfield-Kan}, cf. also \cite[Section 2.5]{Brun}}]
        Let $\cC$ be a small category, and $F: \cC \to Spc_{(*)}$ a functor. The \emph{homotopy colimit} of $F$, $\hocolim_\cC F$, is defined to be the realisation of the simplicial space:
        \begin{equation}\label{eq: hocolim}
            [k] \mapsto \bigvee\limits_{x_0 \leftarrow \ldots \leftarrow x_k \in \cC} F(x_k)
        \end{equation}
        where $\vee$ means $\sqcup$ in $Spc$, and the wedge is over sequences of $k$ composable morphisms in $\cC$.

        The face maps come from composition in $\cC$ and the degeneracy maps from inserting the identity of $x_i$.

        For an object $x \in \cC$, we write $\iota_0$ for the inclusion into the 0-simplices $F(x) \to \hocolim_F \cC$.
    \end{defn}
    $\hocolim$ takes functors to spaces, and natural transformations to maps of spaces:
    \begin{defn}\label{def: hocolim func}

        Homotopy colimits are functorial in the following way. Let $\cC$ and $\cD$ be small categories, and $F: \cC \to Spc_{(*)}$ and $G: \cD \to Spc_{(*)}$ be functors. Let $A: \cC \to \cD$ be a functor, and let $T: F \to G \circ A$ be a natural transformation.
    
        This induces a map $(A,T)_*: \hocolim_\cC F \to \hocolim_\cD G$, constructed as follows. The map sends each term $F(x_k)$ of the $k$-simplices corresponding to the sequence $x_0 \leftarrow \ldots \leftarrow x_k$ in $\cC$ to the term $F(A(x_k))$ corresponding to the sequence $A(x_0) \leftarrow \ldots \leftarrow A(x_k)$, and on this term is given by the map $F(x_k) \xrightarrow{T} G( A(x_k))$.
    \end{defn}
    The following is used in Section \ref{sec: OC}.
    \begin{lem}\label{lem: hocolim hom}
        Let $\cC, \cD$ be small categories, and $F: \cC \to Spc_{(*)}$ and $G: \cD \to Spc_{(*)}$ be functors. 

        Let $A, A': \cC \to \cD$ be functors, and $T: F \to G \circ A$, $T': F \to G \circ A'$ natural transformations.

        Let $U: A \to A'$ be a natural transformation, such that the two natural transformations $F \to G \circ A'$ given by $T'$ and $(G \circ U) \circ T$ agree.

        Then the two maps $(A,T)_*, (A',T')_*: \hocolim_\cC F \to \hocolim_\cD G$ are homotopic (via a naturally-constructed homotopy).
    \end{lem}
    
    \begin{proof}
        Let $A_1$ be the category with objects $0,1$, and one nontrivial arrow $a: 0 \to 1$. Let $H: \cC \times A_1 \to \cD$ be the functor given by $F$ on $\cC \times 0$, $G$ on $\cC \times 1$ and $U$ on $\cC \times a$. $T$ and $T'$ together define a natural transformation $S: F \to G \circ H$ of functors $\cC \times A_1 \to Spc_{(*)}$. 

        From Definition \ref{def: hocolim func}, we obtain a map $(H,S)_*: \hocolim_{\cC \times A_1} F \to \hocolim_\cD G$. There is an equivalence $\hocolim_{\cC \times A_1} F \simeq (\hocolim_\cC F) \times [0,1]$, and $(H,S)_*$ restricts to $(A,T)_*$ and $(A',T')_*$ over $0$ and $1$ respectively, thus providing the desired homotopy.
    \end{proof}

\subsection{$\cI$-spaces}
    In this section, we recap the notions of $\cI$-spaces and $\cI$-monoids, following \cite[Section 2]{Schlichtkrull}.

    Let $\cI$ be the category with objects nonnegative integers $n$, and morphisms $n \to m$ given by (not necessarily order-preserving) injections $[n] \to [m]$, where $[n] := \{1, \ldots, n\}$. 

    \begin{defn}
        An \emph{$\cI$-space} is a functor from $\cI$ to the category of spaces. 
    \end{defn}
    For any category $\cC$, one can similarly consider an \emph{$\cI$-object in $\cC$} as a functor $\cI \to \cC$.
    
    Given an $\cI$-space $X$, we define $X_{h\cI}=\operatorname{hocolim}_{n \in \cI} X(n)$ to be the homotopy colimit of the $X(n)$.

    We say an $\cI$-space $X$ is \emph{convergent} if the connectivity of the maps $X(n) \to X(n')$ goes to $\infty$ as $n \to \infty$.
    \begin{lem}[{\cite[Lemma 2.5.1]{Brun}}]
        If $X$ is a convergent $\cI$-space, the connectivity of the inclusion map $X(n) \to X_{h\cI}$ goes to $\infty$ as $n \to \infty$.
    \end{lem}
    \begin{rmk}
        All examples of $\cI$-spaces we encounter throughout this paper are convergent. 
    \end{rmk}
    
    $\cI$ is symmetric monoidal, with monoidal structure given by $(n, m) \mapsto n+m$. The unit of the symmetric monoidal structure is the empty set $0$ and the symmetric structure comes from the bijection $n+m \to m+n$ which sends the first $n$ elements to the last $n$ elements (preserving their order) and the last $m$ elements to the first $m$ elements (also preserving their order).

    \begin{rmk}
        Explicitly, the monoidal structure is defined by $(X \otimes Y)(n) = \operatorname{colim}_{[n_1] \sqcup [n_2] \to [n]} X(n_1) \times Y(n_2)$. If $\cC$ is a cocomplete (symmetric) monoidal category, the category of $\cI$-objects in $\cC$ is also (symmetric) monoidal.

        $\otimes$ is not always well-behaved homotopically, though it does have a better-behaved derived version, cf. \cite[Section 2]{Sagave-Schlichtkrull:Gp} for further discussion.
    \end{rmk}
    \begin{ex}
        Other categories of interest $\cC$ that satisfy these conditions include simplicial sets, and spaces over a fixed space $R$.
    \end{ex}

    \begin{defn}\label{def: Imon}
        An \emph{$\cI$-monoid} is a monoidal object in the category of $\cI$-spaces. Explicitly, this consists of an $\cI$-space $X$ equipped with maps $\mu_{mn}: X(m) \times X(n) \to X(m+n)$, which form an associative and unital natural transformation $X(\cdot) \times X(\cdot) \to X(\cdot+\cdot)$, where we view both sides as functors $\cI^2$ to spaces. Here unitality means that the basepoint acts as a two-sided identity element in the natural way.

        An $\cI$-monoid $X$ is \emph{commutative} if the following diagram commutes for all $m,n$:
        \begin{equation}\label{eq: Imon comm}
            \xymatrix{
                X(m) \times X(n)
                \ar[r]_\mu
                \ar[d]_{\operatorname{swap}}
                &
                X(m+n)
                \ar[d]_{\sigma_{mn}}
                \\
                X(n) \times X(m)
                \ar[r]_\mu
                &
                X(n+m)
            }
        \end{equation}  
        where $\sigma_{mn}$ swaps the first $m$ factors of $[m+n]$ with the final $n$ factors.
    \end{defn}
    \begin{rmk}\label{rmk: inj prob}
        Let $X$ be a commutative $\cI$-monoid, and let $S$ be a(n unordered) finite set. Let $n_s \in \cI$ for each $s \in S$, and let $N \in \cI$. Let $f: \sqcup_{s \in S} [n_s] \to [N]$ be an injective map.

        Choose an ordering on $S$, i.e. an isomorphism $g:S \cong [k]$ for some $k$, and consider the composition:
        \begin{equation}\label{eq: prod}
            \prod_{s \in S} X(n_s) \xrightarrow{g} X(n_{g^{-1}(1)}) \times \ldots \times X(n_{g^{-1}(k)}) \xrightarrow{\mu} X(n_{g^{-1}(1)} + \ldots + n_{g^{-1}(k)}) \xrightarrow{X(f)} X(N)
        \end{equation}
        By commutativity of $X$, this map does \emph{not} depend on the choice of ordering on $S$: it only depends on the map $f$. We will use this construction in several places in Sections \ref{sec: cI Gam} and \ref{sec: OC}.
    \end{rmk}   

    For an $\cI$-monoid $X$, $X_{h\cI}$ forms a topological monoid, with product given by the composition:
    \begin{equation}\label{eq: hI prod}
        X_{h\cI} \times X_{h\cI} = \operatorname{hocolim}\limits_{(m,n) \in \cI^2} X(m) \times X(n) \to \operatorname{hocolim}_{(m,n) \in \cI^2} X(m+n) \to X_{h\cI} 
    \end{equation}
    The first arrow uses the product maps $\mu_{mn}$ of $X$ and the second uses the monoidal structure on $\cI$.

    \begin{defn}
        An $\cI$-monoid $X$ is \emph{group-like} if the monoid $X_{h\cI}$ is group-like, i.e. $\pi_0 X_{h\cI}$ is a group.
    \end{defn}

    \begin{rmk}
        $\cI$-spaces model spaces, group-like $\cI$-monoids model loop spaces, and group-like commutative $\cI$-monoids model infinite loop spaces. More precisely, each of these is realised by a Quillen equivalence between the appropriate categories \cite{Sagave-Schlichtkrull:Diagram}.
    \end{rmk}
    \begin{ex}
        The groups $U(n)$ are functorial under injections $n \to m$, and so form a commutative $\cI$-monoid in the category of Lie groups rather than spaces. Since the bar construction is functorial, applying $B$ gives us an $\cI$-monoid in spaces $n \mapsto BU(n)$; we write $BU$ for this $\cI$-monoid. This construction applies similarly to other classical groups.

        A similar example is $U/O$: here each $U/O(n)$ is the space of totally real subspaces of $\bC^n$. This is naturally based at $\bR^n \subseteq \bC^n$; another example is then $\Omega U/O$, where we apply $\Omega$ levelwise.

        Another example is the (levelwise) universal cover $\widetilde{U/O}$ of $U/O$.
    \end{ex}
    \begin{ex}\label{ex: G as I}
        Let $G(n)$ be the topological monoid of based homotopy autoequivalences of $S^n$, viewed as the 1-point compactification of $\bR^n$ (with base-point at infinity). By identifying $S^{n+m} \cong S^n \wedge S^m$, along with the symmetric group action by permuting co-ordinates, $G(\cdot)$ becomes a commutative $\cI$-monoid in the category of (group-like) topological monoids. Applying the bar construction levelwise then provides a commutative $\cI$-monoid $BG$.
    \end{ex}
    \begin{ex}
        The maps $O(n) \to G(n)$, via the obvious linear actions of $O(n)$ on $S^n$, induces a map of commutative $\cI$-monoids $J: BO \to BG$, called the \emph{$J$ homomorphism}.
    \end{ex}
    \begin{ex}\label{ex: BSU}
        The fibration $\det^2: U(N) \to S^1$ is a group homomorphism so its (homotopy) fibre $S_\pm U(N) $ is a Lie group- explicitly, it is the subgroup of $U(N)$ consisting of matrices of determinant $\pm 1$, a $\bZ/2$-extension of $SU(N)$. Its delooping $BS_\pm U(N)$ is  the homotopy fibre of the map $BU(N) \to K(\bZ,2)$ classifying $2c_1$. The squared determinant homomorphism is compatible with the maps $U(N) \to U(M)$ induced by injections of finite sets $[N] \to [M]$. Hence the diagram
        \[
        \xymatrix{
        BS_\pm U(N) \ar[r] \ar@{.>}[d] & BU(N) \ar[r] \ar@{.>}[d] &  \bC\bP^{\infty} \ar@{.>}[d] \\ 
        S_\pm U(N) \ar[r] & U(N) \ar[r] & S^1
        }
        \]
        (with vertical arrows indicating taking based loops) exhibits $N \mapsto BS_\pm U(N)$ as an $\cI$-space.
    \end{ex}
    \begin{lem}
        There is a commutative diagram of commutative $\cI$-monoids:
        \begin{equation}
            \xymatrix{
                \widetilde{U/O}
                \ar[r]_{Re}
                \ar[d]
                &
                BO
                \ar[r]_{\cdot \otimes \bC}
                \ar[d]
                &
                BS_\pm U
                \ar[d]
                \\
                U/O
                \ar[r]_{Re}
                &
                BO
                \ar[r]_{\cdot \otimes \bC} 
                &
                BU
            }
        \end{equation}  
        where the rows are homotopy fibre sequences.
    \end{lem}
    \begin{rmk}
        Working with $BS_\pm U(N)$ (rather than the homotopy fibre $BSU(N)$ of the map classifying $c_1$) is natural from the viewpoint of gradings, and necessary for applications to the cotangent bundle $T^*Q$ of a non-orientable manifold $Q$, since then $c_1(T^*Q)$ is 2-torsion.
    \end{rmk}

    \begin{defn}
        A map of $\cI$-spaces $f: X \to Y$ is \emph{$k$-connected} if the induced map $X_{h\cI} \to Y_{h\cI}$ is $k$-connected. Similarly we set $\pi_* X := \pi_* X_{h\cI}$.
    \end{defn}

    \begin{defn}
        Let $f: X \to Y$ be a map of $\cI$-spaces. Its \emph{homotopy fibre} obtained by taking homotopy fibres of spaces levelwise: $\hofib(f: X \to Y)$ is the $\cI$-space sending $n$ to $\hofib(X(n) \to Y(n))$.

        Homotopy pullbacks are similarly defined levelwise.
    \end{defn}
    \begin{lem}\label{lem: hofib I}
        Let $f: X \to Y$ be a map of convergent $\cI$-spaces. Then $\hofib(f)$ is a convergent $\cI$-space, and $\hofib(f)_{h\cI} \to X_{h\cI} \to Y_{h\cI}$ is a homotopy fibre sequence of spaces.

        Similarly, taking homotopy pullbacks commutes with applying $\cdot_{h\cI}$.
    \end{lem}
    \begin{proof}
        If $X(n) \to X(m)$ and $Y(n) \to Y(m)$ are both $k$-connected, so is $\hofib(f)(n) \to \hofib(f)(m)$; the convergence claim follows.

        The canonical nullhomotopies of each composition $\hofib(f)(n) \to X(n) \to Y(n)$ assemble to a nullhomotopy of the composition $\hofib(f)_{h\cI} \to X_{h\cI} \to Y_{h\cI}$; there is therefore a map $\hofib(f)_{h\cI} \to \hofib(f_{h\cI})$, which induces an isomorphism of homotopy groups.

        The same argument applies to homotopy pullbacks.
    \end{proof}

    For an $\cI$-space $X$, we let $\Omega X$ be the $\cI$-space obtained from $X$ by applying $\Omega \cdot$ levelwise: $(\Omega X)(n) := \Omega(X(n))$.

    Applying Lemma \ref{lem: hofib I} to the map of $\cI$-spaces $\pt \to X$, we find:
    \begin{cor}
        Let $X$ be a convergent $\cI$-space. Then $\Omega X$ is convergent, and there is an equivalence of spaces $(\Omega X)_{h\cI} \simeq \Omega(X_{h\cI})$.
    \end{cor}

\subsection{Bar constructions}\label{sec: bar}
    We use bar constructions in various parts of the paper, particularly Sections \ref{sec: mon loco sys} and \ref{sec: OC}.

    \begin{defn}
        Let $G$ be a topological monoid. The \emph{bar construction} $BG$ is defined to be the realisation of the simplicial space $BG_\bullet$ with $k$-simplices $BG_k := G^k$. The face and degeneracy maps are given by:
        \begin{align}
            d_i(x_1, \ldots, x_k) 
            &=
            {\begin{cases}
                (x_2, \ldots, x_k)
                & \textrm{ if } i=0\\
                (x_1, \ldots, x_{k-1})
                & \textrm{ if } i=k\\
                (x_1, \ldots, x_i x_{i+1}, \ldots x_k)
                & \textrm{ otherwise}
            \end{cases}}
            \\
            s_i(x_1, \ldots, x_k) 
            &= 
            (x_1, \ldots, x_i, 1_G, x_{i+1}, \ldots, x_k)
        \end{align}

        This is a based space with base-point given by the unique 0-simplex $BG_0=G^0$.
        
        Let $M$ be a small topological category with object set $O$. Let $A$ be a right module and $D$ a left module over $M$. Then the \emph{bar construction} $B(A,M,D)$ is defined to be the realisation of the simplicial space $B(A,M,D)_\bullet$ with $k$-simplices:
        \begin{equation}
            B(A,M,D)_k := \bigsqcup\limits_{x_0, \ldots, x_k \in O} A(x_0) \times M(x_0, x_1) \times \ldots \times M(x_{k-1}, x_k) \times D(x_k)
        \end{equation}
    \end{defn}
    \begin{lem}
        $B$ defines a functor from topological monoids to spaces.

        If $f: G \to G'$ is a map of topological monoids such that $G$ and $G'$ are both group-like (meaning $\pi_0 G$ is a group) and $f$ is a weak equivalence of spaces, $Bf: BG \to BG'$ is also a weak equivalence of spaces.
    \end{lem}
    \begin{rmk}
        In a few constructions, it is more convenient to work with simplicial sets rather than topological spaces. Note that the (cyclic) bar constructions throughout Section \ref{sec: OC} still make sense as written when $G$ is a Kan complex rather than space; each construction is then the realisation of a bisimplicial set, which can be modelled by taking the diagonal simplicial set: in symbols, if $S_{\bullet \bullet}$ is a simplicial simplicial set, $|S_{\bullet\bullet}|$ is the simplicial set with $|S_{\bullet\bullet}|_k := S_{kk}$.
        
        We obtain similar constructions in the cases of topological spaces by using the equivalence $|\Sing_\bullet(Y)| \to Y$, where $\Sing_\bullet(Y)$ is the singular simplicial set of a space $Y$.
    \end{rmk}
    We model the standard topological $k$-simplex as $\Delta^k := \{(s_1, \ldots, s_k) \in [0,1]^k\,|\, 0 \leq s_1 \leq \ldots \leq s_k \leq 1\}$. We define maps $r_k: \Delta^k \to [0,k]$, to send $(s_1, \ldots, s_k) \mapsto s_1+\ldots+s_k$. 
    \begin{lem}\label{lem: b deloop}
        Let $Y$ be a path-connected based space which is homotopy equivalent to a CW complex.
        
        Then there is a weak equivalence of spaces $ev: B \Omega Y \to Y$.
    \end{lem}
    
    \begin{proof}
        We give a construction of the map here; that this is an equivalence is standard.

        We first define a map of simplicial spaces $ev': B\Sing_\bullet(\Omega Y) \to \Sing_\bullet Y$ as follows. On $k$-simplices, for $\gamma_1, \ldots, \gamma_k \in \Sing_k \Omega Y^k$, we define $ev'_k$ to send this tuple to the $k$-simplex in $Y$ sending $s \in \Delta^k$ to:
        \begin{equation}
            ev'_k(\gamma_1,\ldots,\gamma_k)(s) := \left[\overline \gamma_1^s \star \ldots \star \overline \gamma_k^s\right](r_k(s))
        \end{equation}
        For $\gamma$ a simplex in $\Sing_k\Omega Y$ and $s \in \Delta^k$, we write $\gamma^s$ for the corresponding loop in $\Omega Y$.

        We define the map $ev$ of spaces by requiring that the following diagram commutes up to homotopy:
        \begin{equation}
            \xymatrix{
                |B\Sing_\bullet\Omega Y|
                \ar[r]_{ev'}
                \ar[d]_\simeq 
                &
                |\Sing_\bullet Y|
                \ar[d]_\simeq
                \\
                B\Omega Y
                \ar[r]_{ev}
                &
                Y
            }
        \end{equation}
    \end{proof}
    
    The space of 1-simplices in $BG$ is exactly $G$. Since all such 1-simplices have boundary on the unique 0-simplex, we obtain a canonical map of spaces $\iota_1: G \to \Omega BG$. This is a weak equivalence of spaces. Unfortunately, $\iota$ only respects the product structure on $G$ and $\Omega BG$ up to (coherent) homotopy, and not on the nose. Instead, we have:
    \begin{lem}\label{lem: zig zag}
        Assume $G$ is group-like (meaning $\pi_0 G$ is a group). Then there is a zig-zag:
        \begin{equation}\label{eq: G zig zag}
            G \xleftarrow{\simeq} WG \xrightarrow{\simeq} \Omega BG
        \end{equation}
        where the intermediate space $WG$ is a group-like topological monoid, and both maps are maps of topological monoids which are homotopy equivalences of (unbased) spaces.

        Furthermore, this can be chosen so that the equivalence of spaces $BG \simeq B\Omega G$ obtained by applying $B$ to (\ref{eq: G zig zag}) agrees with the one from applying Lemma \ref{lem: b deloop} to $BG$.
    \end{lem}
    \begin{proof}
        This follows immediately from the statement of \cite[Theorem 1.3]{Vogt}.

        For completeness, we give a rough overview of the construction. Associated to a topological monoid $M$ there is another group-like topological monoid $WM$ (\cite[Definition 2.2]{Vogt}), along with a monoid map $\eps: WM \to M$ which is a homotopy equivalence of spaces\footnote{We remind the reader \emph{shrinkable} maps as considered in \cite[Definition 2.2]{Vogt} are always homotopy equivalences.}. $WM$ has the property that for another topological monoid $N$, monoid maps $WM \to N$ biject with \emph{(unital) homotopy homomorphisms} $M \to N$ (\cite[Definition 2.1]{Vogt}): roughly this is a map $M \to N$ which respects the products of $M$ and $N$ up to coherent higher homotopy (this is similar to a map of $\bE_1$-spaces). 
        
        Applying \cite[Theorem 4.4]{Vogt} to $M=G$ and $X=BG$, the map of based spaces $B\eps: BWG \to BG$ is induced by a (unique up to homotopy) map of monoids $WG \to W\Omega BG$; composing with $\eps$ gives a map $WG \to \Omega BG$. This completes the required zig-zag.

        The idea behind these constructions is that in $WM$, unlike in $M$, the product map $WM \times WM \to WM$ is an embedding with a ``collar neighbourhood'', and the collar neighbourhood coordinates can be used to keep track of the homotopy coherence.
    \end{proof}
    \begin{defn}
        We call the zig-zag (\ref{eq: G zig zag}), from $G$ to $\Omega BG$, $\widetilde \iota_1$.
    
        We write $\rightsquigarrow$ instead of $\to$ for a zig-zag: this will either be the zig-zag $\widetilde \iota_1$ of topological monoids as in (\ref{eq: G zig zag}), or the zig-zag of spaces obtained by applying a functor such as $B$ (or variants such as $B^{cyc}$) to this zig-zag.
    \end{defn}
    \begin{defn}
        Let $G$ be a topological monoid. The \emph{(1-pointed) cyclic bar construction} $B^{cyc}G$ is defined to be the realisation of the simplicial space $B^{cyc}G_\bullet$ with $k$-simplices $B^{cyc}G_k := G^{k+1}$. The face and degeneracy maps are given by:
        \begin{align}
            d_i(x_0, \ldots, x_k) 
            &=
            \begin{cases}
                (x_k x_0, \ldots, x_{k-1})
                &
                \textrm{ if } i=k
                \\
                (x_0, \ldots, x_i x_{i+1}, \ldots, x_k)
                &
                \textrm{ otherwise } 
            \end{cases}
            \\
            s_i(x_0, \ldots, x_k)
            &=
            (x_1, \ldots, x_i, 1_G, x_{i+1}, \ldots, x_k)
        \end{align}
        The \emph{2-pointed cyclic bar construction} $\BB^{cyc}G$ is defined to be the realisation of the simplicial space $\BB^{cyc}G_\bullet$ with $k$-simplices $G^{2k+2}$, and face and degeneracy maps given by:
        \begin{align}
            d_i(u, x_0^\pm, \ldots, x_k^\pm, v) 
            &=
            \begin{cases}
                (x_0^+ u x_0^-, x_1^\pm, \ldots, x_k^\pm, v)
                & \textrm{ if } i=0\\
                (u, x_1^\pm, \ldots, x_{k-1}^\pm, x_k^+ v x_k^-)
                & \textrm{ if } i=k\\
                (u, x_1^\pm, \ldots, x_{i}^\pm x_{i+1}^\pm, \ldots, x_k^\pm, v)
                & \textrm{ otherwise}
            \end{cases}
            \\
            s_i(u, x_0^\pm, \ldots, x_k^\pm, v) 
            &=
            (u, x_0^\pm, \ldots, x_i^\pm, (1_G, 1_G), x_{i+1}^\pm, \ldots, x_k^\pm, v)
        \end{align}
        where each $x_i^\pm$ represents a pair of elements in $G$, and $x_i^\pm x_{i+1}^\pm$ means pairwise multiplication.

        There is a map $\BB_1: \BB^{cyc}G \to B^{cyc}G$ (compare with \cite[(2.196)]{Sheel:Duality}) induced by a map $(\BB_1)_\bullet$ of simplicial spaces, defined on $k$-simplices by:
        \begin{equation}\label{eq:2b1}
            (\BB_1)_k(u, x_0^\pm, \ldots, x_k^\pm, v) = (v x_k^+ \ldots x_0^+ u, x_1^-, \ldots, x_k^-)
        \end{equation}
    \end{defn}

\begin{figure}[ht]
\begin{center}
\begin{tikzpicture}[scale=0.7] 
\draw (0,0) circle (3);
\draw[thick] (3,0) arc (0:180:3);
\draw[fill, black] (3,0) circle (0.15);
\draw[fill,black] (-3,0) circle (0.15);
\draw[fill,gray] (-3/2, 2.6) circle (0.1);
\draw[fill,gray] (3/2, 2.6) circle (0.1);
\draw[fill,gray] (-3/2, -2.6) circle (0.1);
\draw[fill,gray] (0,-3) circle (0.1);
\draw (3.5,0) node {$v$};
\draw (-3.5,0) node {$u$};

\end{tikzpicture}
\caption{Contracting the upper boundary arc gives a schematic for $\BB_1$. Later, in Section \ref{sec: OC}, a domain with two distinguished punctures gives an element in a 2-pointed cyclic bar complex, and contracting the upper arc amounts to passing from $\partial \Delta$ to $\partial \Delta^{\sim} :=\partial \Delta / \partial \Delta^+$ in the notation of Section \ref{sec: OC lab}\label{fig:from 2-pointed to 1-pointed}}
\end{center}
\end{figure}

    \begin{defn}\label{def: ev0 bar}
        Let $Y$ be a based space. 

        There is a map $ev_0: \BB^{cyc}\Omega Y \to Y$ which is defined on $k$-simplices to send $(u, x_1^\pm, \ldots, x_k^\pm, v)$ to $v(c/2)$, where $c$ is the Moore length of $v$. 
        
        We define the space $\oBB^{cyc} \Omega Y$ to be the realisation of the simplicial subspace of $\BB^{cyc} \Omega Y$ with $k$-simplices consisting of the subspace of tuples $(u, x_1^\pm, \ldots, x_k^\pm, v)$, such that:
        \begin{itemize} 
            \item Each $x_i^-$ is the reverse loop of $x_i^+$.
            \item If $c$ is the Moore length of $u$, then $u|_{[0, c/2]}$ is the reverse path of $u|_{[c/2, c]}$, and similarly for $v$.
        \end{itemize}
    \end{defn}

\subsection{Free loop spaces}
    In Section \ref{sec: OC} we will need to use the free loop space and its relationship to the bar construction.
    \begin{defn}
        Let $Y$ be a space. Its \emph{(Moore) free loop space} $\cL$ is the space of pairs $(\gamma, r)$, where $r \geq 0$ and $\gamma: [0,r] \to Y$ is a map such that $\gamma(0)=\gamma(r)$. We call $r$ the \emph{Moore length} of such a loop. 

        The inclusion of constant loops (with Moore length 0) defines a map $\operatorname{const}: Y \to \cL Y$.
    \end{defn}

    For a Moore path $\gamma$, we write $\overline \gamma$ for its reparametrisation to an ordinary path: if $\gamma$ has Moore length $l$, $\overline \gamma(t) = \gamma(t/l)$ (note this is well-defined since if $l=0$, $\gamma$ is constant). We write $\star$ for Moore concatenation of loops: i.e. $\gamma \star \delta$ is the Moore loop with Moore length given by the sum of those of $\gamma$ and $\delta$.
    \begin{defn}\label{def: cyc bar free loop}
        Let $Y$ be a space. We define a map of simplicial sets $e': B^{cyc}\Sing_\bullet(\Omega Y) \to \cL\Sing_\bullet (Y)$ as follows. On $k$-simplices, $e'$ is defined to be adjoint to the map $S^1_k \times \Sing_k(\Omega Y)^{k+1} \to \Sing_k(Y)$ which sends a simplex $(\tau^k_i, \gamma_0, \ldots, \gamma_k)$ to the singular simplex in $Y$ sending $s \in \Delta^k$ to:
        \begin{equation}
            \left[\overline \gamma_{i+1}^s \star \ldots \star \overline \gamma^s_k \star \overline \gamma^s_0 \star \ldots \star \overline \gamma^s_{i-1}\right](r_k(s))
        \end{equation}

        We define a map of spaces $e: B^{cyc}\Omega Y \to \cL Y$ by requiring that the following diagram commutes up to homotopy:
        \begin{equation}
            \xymatrix{
                |B^{cyc}\Sing_\bullet(\Omega Y)|
                \ar[r]_{e'}
                \ar[d]_\simeq 
                &
                |\cL \Sing_\bullet(Y)|
                \ar[d]_\simeq 
                \\
                B^{cyc}\Omega Y
                \ar[r]_e
                &
                \cL Y
            }
        \end{equation}
    \end{defn}
    \begin{ex}\label{ex: circ}
        We recap the definition of the standard simplicial circle, which (by abuse of notation) we write as $S^1$. This has $k$-simplices $\{\tau_0^k, \ldots, \tau^k_k\}$. The face and degeneracy maps are defined by:
        \begin{align}
            d_j\tau^k_i = \begin{cases}
                \tau^{k-1}_{i-1} 
                &
                \textrm{ if } j \leq i
                \\
                \tau_i^{k-1} 
                & 
                \textrm{ if } j > i
            \end{cases}
            &&
            s_j \tau^k_i = \begin{cases}
                \tau^{k+1}_{i+1}
                &
                \textrm{ if } j \leq i
                \\
                \tau^{k+1}_i
                &
                \textrm{ if } j > i
            \end{cases}
        \end{align}
        Note $\tau^1_0$ is the unique non-degenerate 1-simplex.
        
        This can be viewed as the sub-simplicial set of $\Sing_\bullet(S^1)$, where $\tau_i^k$ corresponds to the singular simplex $\Delta^k \to S^1$ sending $s=(s_1, \ldots, s_k) \in \Delta^k$ to $1$ if $i=k$ and to $s_{i-1}$ if $i \leq k - 1$.
    \end{ex}
    \begin{defn}[{\cite{Jones:cyc}}]\label{def: q}
        Let $G$ be a group-like monoid in the category of Kan complexes. We define a map of simplicial sets $q': B^{cyc}G_\bullet \to \cL BG_\bullet := \hom_{sSet}(S^1_\bullet, BG_\bullet)$ as follows.

        On $k$-simplices, $q'_k$ is defined to be adjoint to the map $S^1_k \times G^{k+1}_k \to G^k_k$, defined by:
        \begin{equation}
            (\tau^k_i, g_0, \ldots, g_k) \mapsto (g_{i+1}, \ldots, g_k, g_0, \ldots, g_{i-1})
        \end{equation}
        
        For $G$ a group-like topological monoid, we define a map of spaces $q: B^{cyc} G \to \cL BG$ by requiring that the following diagram commutes up to homotopy:
        \begin{equation}
            \xymatrix{
                B^{cyc}G
                \ar[r]_q
                \ar[d]_\simeq 
                &
                \cL BG
                \ar[d]_\simeq 
                \\
                |B^{cyc}\Sing_\bullet G| 
                \ar[r]_{q'}
                &
                |\cL BG_\bullet|
            }
        \end{equation}
    \end{defn}
    The two ways of relating cyclic bar constructions to free loop spaces agree in the following sense:
    \begin{lem}\label{lem: bar pre 1}
        Let $Y$ be a based space. Then the following diagram commutes up to homotopy:
        \begin{equation}\label{eq: B11}
            \xymatrix{
                B^{cyc}\Omega Y
                \ar[d]_{q}
                \ar[dr]^e
                &
                \\
                \cL B\Omega Y
                \ar[r]_{\cL ev}^\simeq
                &
                \cL Y
            }
        \end{equation}
    \end{lem}
    \begin{proof}
        It suffices to check this when $Y$ is instead assumed to be a Kan complex $Y_\bullet$, and $q,e,ev$ are replaced by $q',e',ev'$ respectively. The corresponding diagram then commutes on the nose: this is a straightforward check from the definitions.
    \end{proof}
    \begin{cor}
        Let $G$ be a group-like topological monoid. Then the following diagram commutes up to homotopy:
        \begin{equation}\label{eq: bar pre 1 revisited}
            \xymatrix{
                B^{cyc}G 
                \ar@{~>}[r]^{\iota_1} 
                \ar[dr]_q 
                &
                B^{cyc}\Omega BG 
                \ar[d]_e 
                \\
                &
                \cL BG 
            }
        \end{equation}
    \end{cor}
    \begin{proof}
        Consider the following diagram:
        \begin{equation}
            \xymatrix{
                B^{cyc} G 
                \ar[d]_q
                \ar@{~>}[r]^{B^{cyc}\iota_1} 
                &
                B^{cyc}\Omega BG
                \ar[d]_q
                \ar[dl]_e
                \\
                \cL BG
                \ar@{~>}[r]_{\cL B\iota_1} 
                &
                \cL B\Omega BG
            }
        \end{equation}
        The outer square commutes up to homotopy by naturality of $q$ and $\iota_1$. The lower triangle commutes up to homotopy by Lemma \ref{lem: bar pre 1} applied to $BG$, recalling $ev$ is a homotopy inverse to $B\iota_1$ by construction. Since all the maps are equivalences, it follows that the upper triangle also commutes up to homotopy; this is exactly (\ref{eq: bar pre 1 revisited}).
    \end{proof}

    \begin{lem}\label{lem: eoghrtoughor}
        Let $Y$ be a based space. Then the following diagram commutes up to homotopy:
                \begin{equation}\label{eq: bar pre 1}
            \xymatrix{
                \prescript{}2{\overline B}^{cyc}\Omega Y
                \ar[d]
                \ar[rr]_{ev_0}
                &&
                Y
                \ar[d]_{\operatorname{const}}
                \\
                \prescript{}2B^{cyc}\Omega Y
                \ar[r]_{\BB_1}
                &
                B^{cyc}\Omega Y
                \ar[r]_e
                &
                \cL Y
            }
        \end{equation}
    \end{lem}
    \begin{proof}
        Note that (\ref{eq: bar pre 1}) commutes after further concatenating with the map $ev_0: \cL Y \to Y$.
    
        Let $\overline \cL Y \subseteq \cL Y$ be the subspace of loops $\gamma: S^1 \to Y$ which, up to reparametrisation (by an autoequivalence of $S^1$ fixing the basepoint $0 \in S^1$), are given by some path concatenated with its inverse, i.e. of the form $\delta * \overline \delta$, $\delta: [0,1] \to Y$. The concatenation given by going down and then right twice in (\ref{eq: bar pre 1}) naturally factors through $\overline \cL$; let $f$ be this factorisation. Then $\operatorname{const} \circ ev_0 = ev_0 \circ f$. The claim then follows from homotopy commutativity of the following diagram:
        \begin{equation}
            \xymatrix{
                &
                Y
                \ar[d]^{\operatorname{const}}
                \\
                \overline \cL Y
                \ar[r]
                \ar[ur]^{ev_0}
                &
                \cL Y
            }
        \end{equation}
    \end{proof}
\subsection{$\Gamma$-spaces}

    We briefly recap the theory of Segal $\Gamma$-spaces \cite{Segal}, following Schlichtkrull's treatment \cite{Schlichtkrull}. Let $Spc_*$ be the category of (compactly generated, Hausdorff, well-based) based spaces.

    \begin{defn}
        Let $\Gamma^{op}$ be the category of based finite sets and based maps. A \emph{$\Gamma$-space} is a functor $E: \Gamma^{op} \to Spc_*$ such that $E(\pt) = \pt$.

        We say a $\Gamma$-space $E$ is \emph{special} if the map $E(S \vee T) \to E(S) \times E(T)$ is a weak equivalence for all $S, T \in \Gamma^{op}$.

        For a special $\Gamma$-space $E$, $E(S^0)$ obtains the structure of a homotopy-commutative $H$-space in a natural way, via the zig-zag:
        \begin{equation}
            E(S^0) \times E(S^0) \xleftarrow{\simeq} E(S^0 \vee S^0) \to E(S^0)
        \end{equation}
        We say $E$ is \emph{very} if the ensuing unital monoid $\pi_0 E(S^0)$ is a group.
    \end{defn}
    \begin{ex}\label{ex: prot gamm spac}
        Let $G$ be a commutative topological monoid. Let $E(S) = G^{\overline S}$ (where $\overline S$ is $S$ without the basepoint) and, for a map $f: S \to T$ in $\Gamma^{op}$, let $E(f): E(S) \to E(T)$ be given by multiplying together all of the factors in each fibre of $f$. This is well-defined because $G$ is commutative. 

        This defines a special $\Gamma$-space $E$, which is very if $G$ is group-like.

    \end{ex}
    \begin{rmk}
        Example \ref{ex: prot gamm spac} is the prototypical example of a $\Gamma$-space. However, most $\Gamma$-spaces do not arise in this way; see for example \cite[Corollary 4K.7]{Hatcher:book}.
        
        In general, (very) special $\Gamma$-spaces model (group-like) coherently homotopy-commutative and homotopy-associative topological monoids; more precisely, there is an equivalence (in the sense of Quillen \cite{Quillen:book}) between very special $\Gamma$-spaces and infinite loop spaces \cite[Theorem 5.8]{Bousfield-Friedlander}.
    \end{rmk}

    Let $E$ be a $\Gamma$-space. We extend the functor $E: \Gamma^{op} \to Spc_*$ to all based sets by left Kan extension. More explicitly, for a based set $S$, $E(S)$ is the (na\"ive, i.e. not homotopy) colimit over $E(T)$, for $T \subseteq S$ a finite subset. 

    We may also extend $E$ to based simplicial sets. For a based simplicial set $S_\bullet$, we define $E(S_\bullet)$ to be the realisation of the simplicial based space with $k$-simplices $E(S_k)$, and face/degeneracy maps induced by those of $S$.
\subsubsection{From $\cI$ to $\Gamma$}\label{sec: cI Gam}
    In this section, we construct a functor from commutative $\cI$-monoids to $\Gamma$-spaces. A similar construction appears in \cite[Section 2]{Segal}; here we follow \cite[Section 5.2]{Schlichtkrull}.

    \begin{defn}
        Let $S$ be a based finite set. We write $\overline S$ for $S$ minus its basepoint, and $\cP(S)$ for the category of subsets of $\overline S$, with morphisms given by inclusions of nested subsets.

        We define $\cS(S)$ to to be the category with objects functors $\theta: \cP(S) \to \cI$, such that for disjoint $U, V \subseteq \overline S$, the images of $\theta(U)$ and $\theta(V)$ in $\theta(U \cup V)$ are disjoint. The morphisms are natural transformations of functors.
    \end{defn}
    \begin{rmk}
        $\cP(\cdot)$ defines a contravariant functor from finite based sets to categories: for $f: S \to T$ and $U \in \cP(S)$, we obtain a functor $f^*: \cP(T) \to \cP(S)$, which is given on objects by $f^*(U) := f^{-1}U$. This preserves inclusions and so is naturally defined on morphisms too.

        $\cS(\cdot)$ defines a covariant functor from finite based sets to categories, by precomposition with $f^*$ as above.
    \end{rmk}
    \begin{rmk}
        There is a canonical equivalence of categories $\cS(S) \to \cI^{\overline S}$, but no canonical functor in the other direction without making a choice of ordering on $S$ (which cannot be done coherently in our subsequent constructions).
    \end{rmk}
    \begin{rmk}
        We think of objects of $\cS(S)$ as a way of coherently adding together tuples in $\cI$ indexed by subsets of $\overline S$: for each $s \in \overline S$ we have some $\theta(s) \in \cI$, corresponding to a finite approximation $E(\theta(s))$ of $E_{h\cI}$, and for each subset $U \subseteq \overline S$, we have a ``coherent'' way of applying the multiplication in $E$ to the $E(\theta(s))$ for $s \in U$. 
    \end{rmk}

    \begin{defn}\label{def: sS cat}
        Let $E$ be a commutative $\cI$-monoid, and let $S$ be a finite based set. $E$ defines a functor $E_S: \cS(S) \to Spc_*$, defined as follows.

        For an object $\theta \in \cS(S)$, $E_S(\theta)$ is defined to be $\prod_{s \in \overline S} E\left(\theta(s)\right)$.

        For a natural transformation $\theta \to \phi$ of functors $\cP(S) \to \cI$, this induces morphisms $\theta(s) \to \phi(s)$ for all $s \in \overline S$, and by applying the functor $E$ this induces maps of spaces $E_S(\theta) \to E_S(\phi)$.

        Let $f: S \to T$ be a map of finite based spaces. There is a natural transformation $E_{ST}: E_S \to E_T \circ f_*$ of functors $\cS(S) \to Spc_*$ constructed as follows. Let $\theta \in \cS(S)$. We define:
        \begin{equation}
            E_{ST}: \cS(S)(\theta) = \prod\limits_{s \in \overline S} \theta(s) \to \prod\limits_{t \in \overline T} \theta(f^{-1}\{t\}) = E_T\circ f_*(T)(\theta)
        \end{equation}
        so that the map is induced by multiplication in $E$ induced by the injection $\sqcup_{s \in f^{-1}\{t\}}[\theta(s)] \to [\theta(f^{-1}\{t\})]$ induced by $\theta$, for each $t \in \overline T$.

    \end{defn}
    \begin{defn}\label{def: Gam spc from I mon}
        Let $E$ be a commutative $\cI$-monoid. We define a $\Gamma$-space $E^\Gamma$ as follows.

        Let $S$ be a finite based set. We set:
        \begin{equation}
            E^\Gamma(S) := \underset{\theta \in \cS(S)}{\hocolim}\, E_S(\theta)
        \end{equation}

        For a morphism $f: S \to T$ of finite based sets, applying the construction of Definition \ref{def: hocolim func} to the functor $f_*: \cS(S) \to \cS(T)$ and the natural transformation $E_S \to E_T \circ f_*$ defines a map of based spaces $E^{\Gamma}(S) \to E^\Gamma(T)$. It is clear that this is compatible with composition and so $E^\Gamma$ defines a functor $\Gamma^{op} \to Spc_*$.
    \end{defn}
    \begin{lem}[{\cite[Proposition 5.3]{Schlichtkrull}}]
        Let $E$ be a commutative $\cI$-monoid.

        Then $E^\Gamma$ is a special $\Gamma$-space. If $E$ is group-like, $E^\Gamma$ is a very special $\Gamma$-space.
    \end{lem}
    \begin{rmk}
        \cite{Sagave-Schlichtkrull:Diagram} shows that the functor $\cdot^\Gamma$ is part of a Quillen equivalence between group-like commutative $\cI$-monoids and very special $\Gamma$-spaces.
    \end{rmk}

\subsubsection{Circles}

    \begin{ex}
        Let $E$ be a commutative $\cI$-monoid. We compute $E^\Gamma(S^0_\bullet)$. We consider $S^0_\bullet$ with its standard model as a simplicial set: there are exactly two $k$-simplices for each $k$, one of which is over the basepoint.

        Since $\overline{S^0} = \pt$, its $k$-simplices $\overline{S^0_k}$ are each a single point. Therefore each $E^\Gamma(S^0_k)$ is naturally identified with $E^\Gamma(S^0)$, with all the face and degeneracy maps being the identity. It follows that the inclusion of the zero-simplices $E^\Gamma(S^0_0)$ into the realisation $E^\Gamma(S^0_\bullet)$ is a homeomorphism.

        $\cS(S^0)$ is naturally isomorphic to $\cI$, so we find that $E^\Gamma(S^0)$ and hence $E^\Gamma(S^0_\bullet)$ are both naturally homeomorphic to $E_{h\cI}$. 
    \end{ex}

    \begin{lem}[{\cite[Proposition 5.3]{Schlichtkrull}}]\label{lem: gam delop}
        Let $E$ be a convergent commutative $\cI$-monoid. Then there is a weak equivalence of spaces $BE_{h\cI} \to E^\Gamma(S^1)$.
    \end{lem}

    \begin{lem}[{\cite[Proposition 6.3]{Schlichtkrull}}]\label{lem: S6.3}
        Let $E$ be a convergent commutative $\cI$-monoid. Then there is a weak equivalence of spaces $F: B^{cyc}E_{h\cI} \simeq E^\Gamma(S^1_+)$. Here $S^1_+$ is $S^1$ with a disjoint basepoint added.
    \end{lem}
    The equivalence of Lemma \ref{lem: S6.3} is constructed as a zig-zag. We recap the construction here as it will be used explicitly in Section \ref{sec: OC}. The intermediate space in the zig-zag, $V^{cyc}E$, is defined to be the realisation of the simplicial space $V^{cyc}E_\bullet$, with $k$-simplices:
    \begin{equation}\label{eq: Vcyc}
        V^{cyc}E_k = \underset{(n_0, \ldots, n_k, \theta)\in\cI^{k+1} \times \cS((S^1_+)_k)}{\hocolim} E(n_0+\theta(\tau_0)) \times \ldots \times E(n_k+\theta(\tau_k))
    \end{equation}
    The natural functors $\cI^{k+1} \to \cI^{k+1} \times \cS((S^1_+)_k) \leftarrow \cS((S^1_+)_k)$ induce maps of spaces $B^{cyc}E_{h\cI} \to V^{cyc}E \leftarrow \cS(S^1_+)$, which are shown to be equivalences in \cite{Schlichtkrull} using the convergence condition.

    \begin{lem}\label{lem: p668}
        Let $E$ be a very special $\Gamma$-space. Then there is a natural equivalence $E(S^1_+) \to \cL E(S^1)$.
    \end{lem}
    This equivalence is constructed on \cite[Page 668]{Schlichtkrull}. 

    For $E$ a commutative $\cI$-monoid, we now have two equivalences between the free loop space and the cyclic bar construction, coming from Definition \ref{def: q} and Lemma \ref{lem: p668}. They agree in the appropriate sense:
    \begin{lem}\label{lem: 330}
        Let $E$ be a commutative $\cI$-monoid. Then the following diagram commutes up to homotopy:
        \begin{equation}\label{eq: 330}
            \xymatrix{
                B^{cyc}E_{h\cI}
                \ar[r]
                \ar[d]_q
                &
                V^{cyc}E
                &
                E^\Gamma(S^1_+) 
                \ar[d]
                \ar[l]
                \\
                \cL BE_{h\cI} 
                \ar[rr]_{\cL F} 
                &&
                \cL E^\Gamma(S^1)
            }
        \end{equation}
        where $q$ is as in Definition \ref{def: q} and the right vertical arrow is the map from Lemma \ref{lem: p668}.
    \end{lem}
    \begin{proof}
        This exact statement is proved in \cite[Proof of Theorem 1.2]{Schlichtkrull}. We remark on the differences on notation: what Schlichtkrull calls $W^{cyc}$ we call $V^{cyc}$, and what he calls $GL_1(R)$\footnote{Note that the $X$ appearing in \cite[Proof of Theorem 1.2]{Schlichtkrull} is a typo and should say $GL_1(R)$.} can be replaced with our $E$ without affecting the argument. 
    \end{proof}

\subsection{Delooping}\label{sec: de loop}
    \begin{defn}
        Let $E$ be a $\Gamma$-space. We define the $\Gamma$-space $BE$ so that $BE(S)$ is the realisation of the simplicial space $[k] \mapsto E(S \times [k])$.
    \end{defn}
    The category of $\Gamma$-spaces is presentable \cite[Appendix A]{Schwede:Gamma}; $B$ is constructed as a colimit and hence preserves colimits, so $B$ has a right adjoint $\Omega$. For very special $E$, the unit $E \to \Omega B E$ of the adjunction is a weak equivalences (i.e. induces an equivalence of spaces for all $S \in \Gamma^{op}$) \cite{Segal}. 
    \begin{lem}[{\cite{Segal}}]
        Let $E$ be a very special $\Gamma$-space. Then $BE$ is also very special.
    \end{lem}
    \begin{lem}
        Let $X_\bullet$ be a based simplicial set and $E$ a very special $\Gamma$-space. Then there are natural equivalences $E(X_\bullet \wedge S^i) \simeq B^i(X_\bullet)$ for all $i$.
    \end{lem}
    \begin{proof}
        Since both sides preserve colimits send disjoint unions to products, it suffices to check this when $X_\bullet$ is $S^0$. For $i=1$, this is Lemma \ref{lem: gam delop}, and for $i \geq 2$, we may apply Lemma \ref{lem: gam delop} iteratively using $S^i = (S^1)^{\wedge i}$.
    \end{proof}
    The stable homotopy groups of spheres act on very special $\Gamma$-spaces in the following way.
    \begin{defn}\label{def: stab stem act gam}
        Let $E$ be a very special $\Gamma$-space. Let $\alpha \in \pi_i\bS$. We define $\alpha_*: E(S^i) \to E(S^0)$ to be the composition:
        \begin{equation}\label{ewojfwrpjgerj}
            E(S^i) \simeq \Omega^k E(S^{i+k}) \xrightarrow{\Omega^k E(\alpha')} \Omega^k(S^k) \simeq E(S^0)
        \end{equation}
        where $\alpha': S^{i+k} \to S^k$ is an unstable representative of $\alpha$.

        This is a well-defined map of spaces, up to homotopy.
    \end{defn}
    \begin{rmk}\label{rmk: eta act i spac}
        Let $X$ be a commutative $\cI$-monoid, and $\alpha \in \pi_1\bS$. By Lemma \ref{lem: gam delop}, $\alpha$ then induces a map $BX_{h\cI} \to X_{h\cI}$, well-defined up to homotopy.
    \end{rmk}
    
    Let $\eta \in \pi_1\bS$ be the stable Hopf map. Schlichtkrull proves the action of $\eta$ on an infinite loop space admits the following alternative description; this will be crucial in Section \ref{sec: OC}:
    \begin{prop}[{\cite[Proposition 7.3]{Schlichtkrull}}]\label{prop: 331}
        Let $E$ be a very special $\Gamma$-space. Then the following composition is homotopic to $\eta_*$:
        \begin{equation}
            E(S^1) \xrightarrow{\operatorname{const}} \cL E(S^1) \xleftarrow{\simeq} E(S^1_+) \xrightarrow{S^1_+ \to S^0} E(S^0)
        \end{equation}
        where the backwards arrow is the equivalence from Lemma \ref{lem: p668}.
    \end{prop}

\subsection{Symmetric spectra}\label{sec: spec}
    \begin{defn}[{\cite[Definition 1.2.2]{Hovey-Shipley-Smith}, cf. also \cite{Mandell-May-Schwede-Shipley}}]
        A \emph{symmetric spectrum} $Z$ consists of:
        \begin{itemize}
            \item A sequence of based spaces $Z_n$, along with (left) actions of the symmetric group $\Sym_n$.
            \item Maps $\eps_n: S^1 \wedge Z_n \to Z_{n+1}$.
        \end{itemize}
        such that the compositions:
        \begin{equation}
            \eps^i_n: S^i \wedge Z_n = (S^1)^{\wedge i} \wedge Z \to (S^1)^{\wedge (i-1)} \wedge Z_{1+n} \to \ldots \to Z_{i+n}
        \end{equation}
        are $\Sym_i \times \Sym_n$-equivariant.
        
        A symmetric spectrum $Z$ is \emph{convergent} if the connectivity of the maps $\eps^i_n$ goes to $\infty$ as $n \to \infty$. 
    \end{defn}
    We will often refer symmetric spectra just as spectra.
    \begin{defn}
        A \emph{ring spectrum} is a spectrum $R$ along with:
        \begin{itemize}
            \item $\Sym_n$-equivariant maps $\iota: S^n \to R_n$.
            \item $\Sym_i \times \Sym_j$-equivariant maps $\mu: R_i \wedge R_j \to R_{i+j}$.
        \end{itemize}
        such that for all $i,j,k$,
        \begin{itemize}
            \item $\iota_{i+j} = \eps^i(\id_{S^i} \wedge \iota_j)$
            \item (Unitality): The following diagrams commute:
            \begin{equation}
                \xymatrix{
                    S^i \wedge R_j 
                    \ar[r]_{\iota_i}
                    \ar[dr]_{\eps^i_j} 
                    &
                    R_i \wedge R_j 
                    \ar[d]_\mu 
                    &
                    R_i \wedge S^j
                    \ar[d]_{\operatorname{swap}}
                    \ar[r]_{\id_{R_j} \wedge \iota_j} 
                    &
                    R_i \wedge R_j
                    \ar[dr]_\mu 
                    &
                    \\
                    &
                    R_{i+j}
                    &
                    S^j \wedge R_i 
                    \ar[r]_{\eps^j_i}
                    &
                    R_{i+j} 
                    \ar[r]_{\sigma_{ij}}
                    &
                    R_{i+j}
                }
            \end{equation}
            where $\sigma_{ij} \in \Sym_{i+j}$ swaps the first $i$ and last $j$ entries.
            \item (Associativity): The following diagram commutes:
            \begin{equation}
                \xymatrix{
                    R_i \wedge R_j \wedge R_k 
                    \ar[r]_\mu
                    \ar[d]_\mu 
                    &
                    R_i \wedge R_{j+k}
                    \ar[d]_\mu
                    \\
                    R_{i+j} \wedge R_k
                    \ar[r]_\mu
                    &
                    R_{i+j+k}
                }
            \end{equation}
        \end{itemize}
        $R$ is \emph{commutative} if the following diagram also commutes:
        \begin{equation}
            \xymatrix{
                R_i \wedge R_j
                \ar[r]_\mu 
                \ar[d]_{\operatorname{swap}}
                &
                R_{i+j}
                \ar[d]_{\sigma_{ij}}
                \\
                R_j \wedge R_i
                \ar[r]_\mu
                &
                R_{i+j}
            }
        \end{equation}
    \end{defn}
    \begin{rmk}
        The category of symmetric spectra is a symmetric monoidal \cite[Section 2]{Hovey-Shipley-Smith}, and a (commutative) ring spectrum is exactly a (commutative) ring object in this category.
    \end{rmk}

\subsection{Units of ring spectra\label{subsec:units}}
    \begin{defn}[{\cite[Section 2.3]{Schlichtkrull}}]
        Let $R$ be a commutative ring spectrum. There is a commutative $\cI$-monoid $\Omega^\infty R$ sending $n$ to $\Omega^n R_n$, with $\cI$-space structure maps constructed as follows. For a morphism $f: i \to j$ in $\cI$ and $\alpha \in \Omega^i R_i$, $f$ sends $\alpha$ to the composition:
        \begin{equation}
            S^j \xrightarrow{\tilde f^{-1}} S^j \xrightarrow{=} S^{j-i} \wedge S^i \xrightarrow{\alpha} S^{j-i} \wedge R_i \xrightarrow{\eps_i^{j-i}} R_j \xrightarrow{\tilde f} R_j
        \end{equation}
        where $\tilde f: [j] \to [j]$ is defined so that if $x > j-i$, $\tilde f = f(x-j+i)$, and is order-preserving on the first $j-i$ entries. The product structure is the composition:
        \begin{equation}
            \Omega^i R_i \wedge \Omega^j R_j \to \Omega^{i+j} R_i \wedge R_j \xrightarrow{\Omega^{i+j} \mu} \Omega^{i+j} R_{i+j}
        \end{equation}
        By abuse of notation, we also sometimes write $\Omega^\infty R$ for $\Omega^\infty R_{h\cI}$.
    \end{defn}

    \begin{defn}[{\cite[Section 2.3]{Schlichtkrull}}]\label{def: GL one}
        Let $R$ be a commutative symmetric ring spectrum. There is a corresponding group-like commutative $\cI$-monoid $GL_1(R)$. Each $GL_1(R)(i)$ is defined to be the union of components of $\Omega^\infty R(i)$, which admit an element $c$ such that there is some $j$ and $c' \in \Omega^\infty R(j)$ such that $\mu(c,c') \in \Omega^\infty R(i+j)$ lies in the trivial path component. The product structure is inherited from $\Omega^\infty R$.
    \end{defn}
    Note that the assignment $R \mapsto GL_1(R)$ is functorial in $R$.
    \begin{rmk}\label{rmk:A}
        $GL_1(R)$ corresponds to an infinite loop space under the equivalence of \cite{Sagave-Schlichtkrull:Diagram} and hence gives rise to a connective spectrum, denoted $gl_1(R)$, and similarly delooping this infinite loop space gives another connective spectrum $bgl_1(R)$, which is equivalent to the shift $\Sigma gl_1(R)$. The map of infinite loop spaces induced by a class $\alpha \in \pi_1 \bS$ as in Definition \ref{def: stab stem act gam} gives a map of spectra $\alpha_*: bgl_1(R) \to gl_1(R)$ whose delooping is (\ref{ewojfwrpjgerj}). Explicitly, this is the map $\Sigma gl_1(R) \to gl_1(R)$ induced by $\alpha$ using the action of $\pi_*\bS$ on any spectrum.
    \end{rmk}
    \begin{ex}
        When $R=\bS$ is the sphere spectrum, $GL_1(\bS) \simeq G$ is as considered in Example \ref{ex: G as I}. Note however that for an arbitrary commutative ring spectrum $R$, each $GL_1(R)(n)$ is not a monoid, so we cannot apply $B\cdot$ levelwise to obtain a new commutative $\cI$-monoid $BGL_1(R)(\cdot)$: one must apply $\cdot_{h\cI}$ and then apply $B\cdot$.
    \end{ex}
    \begin{rmk}
        If $R$ is a commutative ring spectrum such that $R$ is convergent, so are both $\Omega^\infty R$ and $GL_1(R)$.
    \end{rmk}

\subsection{Thom spaces}
    \begin{defn}
        A \emph{Thom space} is a pair $(X, E)$, where $X$ is a space and $E$ is a vector bundle over $X$. A \emph{morphism} $(X, E) \to (X', E')$ is a map of spaces $f: X \to X'$ along with an isomorphism of vector bundles $E \to f^*E$. We write $\Thom$ for the ensuing category of Thom spaces.

        The category of Thom spaces is symmetric monoidal, with $(X, E) \times (X', E') := (X \times X', E \boxplus E')$.

        The category of Thom spaces is tensored over both the category of spaces by setting $(X, E) \times Y := (X \times Y, E)$, and the category of (finite-dimensional) vector spaces by setting $(X, E) \times V := (X, E \oplus V)$. Similarly, if $E' \to X$ is another vector bundle, we set $(X, E) \oplus E'$ to be $(X, E \oplus E')$.
    \end{defn}

    Given a Thom space $(X, E)$, we may take its Thom space (in the classical sense) $\Thom(E \to X)$; the convention that we use is that this is the quotient of the fibrewise one-point compactification of the total space $\Tot(E \to X)$ by the section at infinity.

    \begin{defn}\label{def: Thom}
        A \emph{Thom $\cI$-space} $X$ consists of Thom spaces $X(n)$ for each $n \in \cI$, and for each morphism $f: n \to m$ in $\cI$, a map of Thom spaces $f(n): X(n) \oplus \bR^{m-n} \to X(m)$, where $m-n$ denotes the finite set $[m] \setminus f[n]$.

        We require that, for morphisms $n \to m \to l$, the following diagram commutes:
        \begin{equation}
            \xymatrix{
                X(n) \oplus \bR^{m-n} \oplus \bR^{l-m} 
                \ar[r]_=
                \ar[d]
                &
                X(n) \oplus \bR^{l-n}
                \ar[d]
                \\
                X(m) \oplus \bR^{l-m}
                \ar[r]
                &
                X(l)
            }
        \end{equation}

        A \emph{morphism} $X \to Y$ of Thom $\cI$-spaces consists of maps $X(n) \to Y(n)$ for each $n$, such that for any morphism $n \to m$ in $\cI$ the following diagram commutes:
        \begin{equation}
            \xymatrix{
                X(n) \oplus \bR^{m-n}
                \ar[r]
                \ar[d]
                &
                Y(n) \oplus \bR^{m-n}
                \ar[d]
                \\
                X(m)
                \ar[r]
                &
                Y(m)
            }
        \end{equation}
        The category $\Thom^{\cI}$ of Thom $\cI$-spaces is symmetric monoidal, via the same construction of the monoidal structure on $\cI$-spaces (cf. Remark \ref{def: Imon} and \cite[Section 2]{Sagave-Schlichtkrull:Gp}): for example, a map $X \otimes Y \to Z$ consists of the same data as a sequence of maps in $\sSet\Vect$ $X(n) \times Y(m) \to Z(n+m)$, satisfying a similar requirement to Definition \ref{def: Imon}. There is a `shift functor' on Thom $\cI$-spaces, where $X[1]$ is defined so that $X[1](n) = X(n) \oplus \bR$.

        A \emph{(commutative) Thom $\cI$-monoid} is a (commutative) monoid in the category $\Thom^\cI$.
    \end{defn}
    \begin{rmk}
        Note that a Thom $\cI$-space is not the same as an $\cI$-object in the category $\Thom$.
    \end{rmk}

Since commutative $\cI$-monoids model infinite loop spaces, if $X$ is a commutative $\cI$-monoid then $X_{h\cI}$ admits an ``inversion" map $-1$.  This is not typically realised on the finite approximations $X(n)$, so we introduce the following:

\begin{defn}\label{defn:inverse in  I-monoid}
    Let $X$ be a group-like commutative $\cI$-monoid and $f: Z \to X(n)$ a map.  Then an \emph{inverse} to $f$ comprises
    \begin{enumerate}
        \item positive integers $m$ and $l$ with $[n] \sqcup [m] \hookrightarrow [l]$;
        \item a map (schematically denoted) $f^{-1}: Z \to X(m)$;
        \item a nullhomotopy of the composite 
        \[
        Z \to X(n)  \times X(m) \to X(l)
        \]
    \end{enumerate}
\end{defn}

\begin{lem}\label{lem:inverse maps to I-monoids exist}
    If $Z$ is a finite-dimensional cell complex, then for $m$ and $l$ sufficiently large, a map $f: Z\to X(n)$  admits an inverse. Moreover, the connectivity of the space of choices (i)-(iii) as in Definition \ref{defn:inverse in  I-monoid} becomes unbounded as $m,l \to \infty$.
\end{lem}

\begin{proof}
    Both statements follow from the fact that $X_{h\cI}$ has the homotopy type of an infinite loop space.
\end{proof}

Now if $E$ is a commutative Thom $\cI$-monoid and $(Z,F) \to E(n)$ is a map of Thom spaces, then choosing an inverse to $f: Z \to \Base E(n)$ as above, one obtains 
\begin{equation}\label{eqn:trivialisation from inverse monoid map}
F \oplus (f^{-1})^* E(m) \cong \bR^{\mathrm{rk}(E(l))}
\end{equation}
and this trivialisation is canonical up to homotopy.

    \begin{defn}\label{def: thom from thom}
        There is a symmetric monoidal functor from $\Thom$ $\cI$-spaces to symmetric spectra, sending $(X,E)$ to the symmetric spectrum sending $n$ to the Thom space $\Thom(E(n) \to X(n))$.

        There is a symmetric monoidal functor $\Base$ from $\Thom$ ($\cI$-)spaces to ($\cI$-)spaces, sending $Y=(X, E)$ to $X$.
    \end{defn}
    \begin{rmk}\label{rmk: comm thom ring spec}
        Let $(X, E)$ be a (commutative) Thom $\cI$-monoid. Its image under the functor to symmetric spectra is a (commutative) ring spectrum; explicitly, its $n^{th}$ space is $\Thom(E(n) \to X(n))$. If $(X, E)$ is convergent, so is the corresponding ring spectrum.
    \end{rmk}
    \begin{defn}
        Let $X, Y$ be Thom ($\cI$-)spaces, and $f: X \to Y$ a map.

        $f$ is an \emph{equivalence} if $\Base(f)$ is and that $f$ is \emph{$k$-connected} if $\Base(f)$ is. If $Z$ is a Thom $\cI$-space, $Z$ is \emph{convergent} if $\Base(Z)$ is convergent.
    
    \end{defn}

    \begin{lem}
        Let $f: (X, E) \to (X', E')$ be a $k$-connected map of Thom spaces, and $g: (Z, F) \to (X', E')$ a map of Thom spaces. Consider the lifting problem:
        \begin{equation}\label{eq: reiwoghe}
            \xymatrix{
                &
                (X,E)
                \ar[d]_f
                \\
                (Z,E)
                \ar[r]_g
                \ar@{-->}[ur]
                &
                (X',E')
            }
        \end{equation}
        If $Z$ is a CW complex of dimension $\leq k$, then there exists a homotopy lift $h$ of Thom spaces as in (\ref{eq: reiwoghe}). Explicitly, this means a map $h$ as in the diagram, along with a homotopy between the two ways around the diagram.

        If instead $Z' \subseteq Z$ is a CW subcomplex, a homotopy lift $h$ of (\ref{eq: reiwoghe}) is given on $(Z', E|_{Z'})$ and $Z$ is obtained from $Z'$ by attaching cells of dimension $\leq k$, then there exists a homotopy lift over all of $Z$, extending the given one over $Z'$.
    \end{lem}
    \begin{proof}
        This is standard in spaces. The case of Thom spaces follows from the fact that an isomorphism of vector bundles over $Z \times \{0\}$ can always be extended to one over $Z \times [0,1]$.
    \end{proof}

    Many of our constructions are more naturally done on the level of simplicial sets, so it will be convenient to work with the following variants of the above definitions.
    \begin{defn}
        For a simplicial set $S_\bullet$, a \emph{vector bundle} on $S_\bullet$ is a vector bundle (in the usual sense) on its geometric realisation $|S_\bullet|$.
    \end{defn}
    \begin{defn}
        A \emph{Thom simplicial set} consists of a pair $(X_\bullet, E)$, where $X_\bullet$ is a simplicial set and $E$ is a vector bundle over $X_\bullet$. The category of Thom simplicial sets is symmetric monoidal, by setting $(X_\bullet, E) \otimes (X'_\bullet, E') := (X_\bullet \times X'_\bullet, E \boxplus E')$ (using the canonical homeomorphism $|X_\bullet| \times |X'_\bullet| \cong |X_\bullet \times X'_\bullet|$).

        The \emph{geometric realisation} $|(X_\bullet,E)|$ of a Thom simplicial set or Thom $\cI$-simplicial set is given by applying the geometric realisation functor to the underlying simplicial set(s).
    \end{defn}
    The functor $\Base$ from Thom ($\cI$-)simplicial sets to ($\cI$-)simplicial sets, as well as the notions of connectivity and equivalence, are defined the same as in spaces.

\section{Abstract discs\label{sec:abstract discs}}

    In this section, we give models for Bott periodicity well-adapted to index theory for Cauchy-Riemann operators; this is largely similar to \cite[Section 3]{PS2}, with some modifications to account for commutativity, and now incorporating the language of Section \ref{sec: htpy prel}. 
\subsection{Tangential pairs}
    \begin{defn}\label{def: comm tang pair}
        A \emph{(graded) commutative tangential pair} $\Psi$ consists of a pair $(\Theta \to \Phi)$ of group-like connected commutative $\cI$-monoids, equipped with maps to $BO$ and $BS_\pm U$ fitting into the following commutative diagram:
        \begin{equation}
            \xymatrix{
                \Theta 
                \ar[r]
                \ar[d]
                &
                \Phi
                \ar[d]
                \\
                BO
                \ar[r]
                &
                BS_{\pm} U
            }
        \end{equation}
        where $BS_{\pm} U$ is as in Example \ref{ex: BSU}. We require that the homotopy fibre $F$ of $\Theta \to \Phi$ is simply connected.

        We say that $\Phi$ is \emph{oriented} if the map on homotopy fibres $F \to \widetilde{U/O}$ induces the zero map on $\pi_2$, and \emph{unoriented} otherwise. 
        
    \end{defn}
    One of the main differences between our set-up here and in \cite{PS2} is that in \cite{PS2}, we did not impose any sort of commutativity.
    \begin{rmk}
        Lemma \ref{lem: pi zero R} motivates the $\pi_2$ condition for orientability.
    \end{rmk}
    \begin{ex}
        We define $fr$ to be the commutative tangential pair $(\Theta=\pt, \Phi=\pt)$.
        
        Commutative tangential pairs form a category in a natural way, and $fr$ is initial in the category of commutative tangential pairs.
    \end{ex}
    \begin{ex}\label{ex: tang str}
        It is often convenient to simplify the discussion by working with commutative tangential pairs where $\Phi=\pt$ and $\Theta=F$ some commutative $\cI$-monoid $F$ over $\widetilde{U/O}$:
        \begin{equation}
            \xymatrix{
                &
                F
                \ar[r]
                \ar[dl]
                \ar[d]
                &
                \pt
                \ar[d]
                \\
                \widetilde{U/O}
                \ar[r]
                &
                BO
                \ar[r]
                &
                BS_\pm U
            }
        \end{equation}
    \end{ex}
    \begin{rmk}
        In \cite{PS2}, this corresponds to what we called a \emph{tangential structure} (as opposed to pair).
    \end{rmk}
    In Example \ref{ex: tang str}, a $\Phi$-orientation on a Liouville domain $X$ is a stable framing; we may modify the construction to be able to obtain similar Fukaya-categorical structures when $X$ only has a polarisation:
    \begin{defn}\label{defn: tang str pol}
        Let $f: F \to \widetilde{U/O}$ be a map of commutative $\cI$-monoids. Let $\Theta_F = BO \times F$, $\Phi_F=BO$, equipped with the maps shown in the right square of the following commutative diagram, with rows homotopy fibration sequences:
        \begin{equation}\label{eq: iregneirgn}
            \xymatrix{
                F
                \ar[rrr]_{\pt \times \id} 
                \ar[d]_f
                &&&
                BO \times F 
                \ar[rrr]_{Proj_1}
                \ar[d]_{Proj_1 + (Re \circ f)} 
                &&&
                BO
                \ar[d]_{\cdot \otimes \bC} 
                \\
                \widetilde{U/O}
                \ar[rrr]_{Re}
                &&&
                BO
                \ar[rrr]_{\cdot \otimes \bC} 
                &&&
                BS_\pm U
            }
        \end{equation}
        where $Proj_1$ denotes projection to the first factor.

        We call this $(\Theta_F, \Phi_F)$ the \emph{polarised tangential pair} of $F$. 
    \end{defn}
    \begin{ex}\label{ex: UOor}
        We let $\widetilde{U/O}^{or}$ be the commutative $\cI$-monoid obtained from $\widetilde{U/O}$ by taking 2-connected cover for each $n \in \cI$. Note that for each $n$, we have a homotopy fibre sequence:
        \begin{equation}
            \widetilde{U/O}^{or}(n) \to \widetilde{U/O}(n) \xrightarrow{w_2 \circ Re} K(\bZ/2,2)
        \end{equation}
        where the second map classifies the second Stiefel-Whitney class of a totally real subbundle of $\bC^n$.

        For a commutative $\cI$-monoid $F$ over $\widetilde{U/O}^{or}$, the polarised tangential pair of $F$ is oriented, since $\pi_2\widetilde{U/O}^{or}=0$.
    \end{ex}

\subsection{Abstract discs: fixed domain case}\label{sec:abst disc fix case}
    In this section, using Cauchy-Riemann operators, we define a model for $BO$ as an infinite loop space, along with a universal bundle over it. We use the models and language developed in the previous sections to describe it. 

    Consider the standard Riemann disk with a single puncture, $D^2 \setminus \{1\}$. We choose an \emph{output strip-like end} near 1 (cf. \cite[Section 7.3]{PS}); explicitly, this is a biholomorphism $\eps$ from $\bR_{\geq 0} \times [0,1]$ onto a neighbourhood of the puncture. We choose a Riemannian metric on $D^2 \setminus \{1\}$, standard over the strip-like end.
    \begin{defn}\label{def: std pun data}
        Consider the trivial vector bundle $\bC \times (\bR_{\geq 0} \times [0,1]) \to \bR_{\geq 0} \times [0,1]$, and the real subbundle $\bR \times (\bR_{\geq 0} \times \{1\}) \sqcup i\bR \times (\bR_{\geq 0} \times \{0\}) \to \bR_{\geq 0} \times \{0, 1\}$.

        We define a map of pairs $(\widetilde e, \widetilde f): (\bR_{\geq 0} \times [0,1], \bR_{\geq 0} \times \{0,1\}) \to (BO(1), BS_\pm U(1))$ compatibly classifying the above vector bundles as follows. 
        
        We set $\widetilde e$ to be constant at the basepoint. Let $\gamma: [0,1] \to {U/O}(1)$ be the path given by $\gamma(t) = e^{\frac{i\pi}2} \cdot i\bR$, and $\widetilde{\gamma}$ its unique lift to the universal cover $\widetilde{U/O}(1)$ sending $0$ to the basepoint. We take $\widetilde f$ to be the composition of $\widetilde \gamma$ with the map $\widetilde{U/O}(1) \to BO(1)$.

        We call this choice of vector bundle and pair of classifying maps the \emph{standard puncture datum}.
    \end{defn}

    \begin{defn}\label{def: fix dom case}
        For each $v, n \in \cI$, there is a Thom simplicial set $(\bU(v, n)_\bullet, \bV(v, n))$, defined as follows. A $0$-simplex in $\bU(v,n)_\bullet$ consists of the following data:
        \begin{itemize}
            \item A complex vector bundle $E$ over $D^2\setminus\{1\}$ of rank $n$. 
            
            \item A real subbundle over the boundary $F \subseteq E|_{\partial D}$, cf. Figure \ref{Fig:abstract disc}. 

            The pair $(E, F)$ is called the \emph{bundle pair}; $E$ is the \emph{complex part} and $F$ the \emph{real part}.

            We call $n$ the \emph{rank}.

            \item A classifying map $D^2\setminus \{1\} \to BS_\pm U(n)$ for $E$.

            \item A classifying map $\partial D^2 \setminus 1 \to BO(n)$ for $F$, compatible with the classifying map for $E$. 

            We require that the classifying maps agree with (the $n$-fold direct sum of) the standard puncture data over the strip-like end.

            \item A 1-form $Y \in \Omega^{0,1}(E)$, and a Hermitian connection (and metric) $\nabla$ on $E$.
            
            We require some additional data and conditions (cf. \cite[Definition 3.9]{PS2} for details). We define a \emph{Cauchy-Riemann operator}:

            \begin{equation}
                D^{CR} = \overline\partial^\nabla + Y: \Gamma_{W^{2,\kappa}}\left(D^2\setminus \{1\}, E, F\right) \to \Omega^{0,1}_{W^{2,\kappa-1}}\left(D^2\setminus \{1\}, E\right)
            \end{equation}
            where the subscripts $\cdot_{W^{2,\cdot}}$ denote passing to suitable Sobolev completions. This operator is always Fredholm, and in particular its cokernel is finite-dimensional.

            \item A linear map $f: \bR^v \to \Omega^{0,1}_{W^{2,\kappa-1}}(D^2\setminus \{1\}, E)$, such that $D^{CR} + f$ is surjective. 
        \end{itemize}
        
        A $k$-simplex consists of a smoothly-varying family $(E_t, F_t, Y_t, \ldots)_{t \in \Delta^k}$ of all of the above data, parametrised by the standard $k$-simplex $\Delta^k$.

        We call $\bU(v,n)$ (and its variants we encounter later) a \emph{space of abstract discs}; we think of a 0-simplex as such an abstract disc.

        The vector bundle $\bV(v,n)$ over $|\bU(v,n)_\bullet|$ is defined so that for a simplex $(\sigma_t)_{t \in \Delta^k} = (E_t, F_t, \ldots)_{t \in \Delta^k}$ as above, the restriction of $\bV(v,n)$ to the realisation of $(\sigma_t)_t$ has fibre over $t \in \Delta^k$ given by the kernel of the associated Cauchy-Riemann operator constructed with the Cauchy-Riemann data over $t$, $\ker(D^{CR}_t)$. The total space has a natural topology on it, hence forming a vector bundle.
    \end{defn}

\begin{figure}[ht]
\begin{center}
\begin{tikzpicture}

\draw[semithick] (2,0.5) arc (-90:-180:0.25);
\draw[semithick] (1.75,0.75) arc (0:180:1.25);
\draw[semithick] (2,-0.5) arc (90:180:0.25);
\draw[semithick] (1.75,-0.75) arc (0:-180:1.25);
\draw[semithick] (-0.75,-0.75) -- (-0.75,0.75);
\draw[semithick] (2,0.5) -- (2.5,0.5);
\draw[semithick,dashed] (2.5,0.5) -- (3,0.5);
\draw[semithick] (2,-0.5) -- (2.5,-0.5);
\draw[semithick,dashed] (2.5,-0.5) -- (3,-.5);
\draw[semithick,dotted] (2,-0.5) -- (2,0.5);
\draw (2.5,0) node {$\mathbb{C}^n$};
\draw (3.5,-0.5) node {$i\mathbb{R}^n$};
\draw (3.5,0.5) node {$\mathbb{R}^n$};
\draw[decoration = {calligraphic brace, mirror, amplitude=5pt}, decorate] (2,-0.7) -- (3,-0.7);
\draw (2.5,-1.05) node {$\varepsilon$};

\draw (0.5,0) node {$E$};
\draw (-0.2,2.25) node {$F$};

\end{tikzpicture}
\end{center}
\caption{Boundary conditions of an abstract disc\label{Fig:abstract disc}.}
\end{figure}
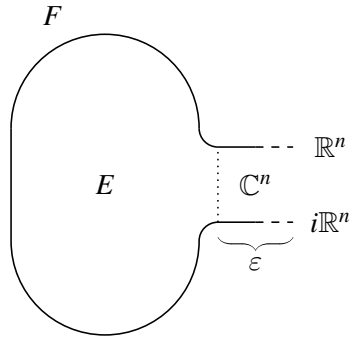
    \begin{rmk}
        Together, the classifying maps of the bundle pair and the puncture datum define a map of pairs $(D^2\setminus\{1\}, \partial D^2\setminus\{1\}) \to (BS_\pm U(n), BO(n))$.
    \end{rmk}
    \begin{defn}
        Let $[v] \to [v']$ be an injection. There is an induced map $\bU(v,n)_\bullet \to \bU(v',n)_\bullet$, defined as follows. This sends a simplex $(\sigma)_t = (E_t, F_t, Y_t, \nabla_t, f_t, \ldots)$ to $(E_t, F_t, Y_t, \nabla_t, f'_t, \ldots)$, where $f'_t: \bR^v = \bR^{v'} \oplus \bR^{v-v'} \to \Omega^{0,1}_{W^{2,\kappa-1}}(D^2\setminus \{1\}, E)$ is obtained from $f$ by precomposing with projection to $\bR^{v'}$. All the extra auxiliary data is taken to be the same.

        There is a natural isomorphism between the pullback of $\bV(v',n)$ and $\bV(v,n) \oplus \bR^{v'-v}$.

        Thse maps assemble to endow the pair $(\bU(\cdot, n)_\bullet, \bV(\cdot, n))$ with the structure of a Thom $\cI$-simplicial set.
    \end{defn}

    \begin{defn}
        Let $[v] \sqcup [v'] \to [v'']$ and $[n] \sqcup [n] \to [n'']$ be isomorphisms in $\cI$ (i.e. bijections). We define a map of Thom simplicial sets:
        \begin{equation}
            \oplus: \left(\bU(v, n)_\bullet, \bV\right) \times \left(\bU(v',n')_\bullet, \bV\right) \to \left(\bU(v'',n'')_\bullet, \bV\right)
        \end{equation}
        This takes two abstract discs (or families of them, parametrised by $\Delta^k$), and direct sums all pieces of data. The classifying map for the complex (resp. real) parts of the boundary conditions is given by the sum of those on the summands, using the map $BS_\pm U(n) \times BS_\pm U(n') \to BS_\pm U(n'')$ (resp. $BO(n) \times BO(n') \to BO(n'')$) induced by the given isomorphism in $\cI$.
    \end{defn}
    
    \begin{defn}\label{def: stab abst disc}
        We fix, once and for all, $D_0 \in  \bU(0,1)_0$ a choice of abstract disc, with $E = \bC$ trivial (note the space of choices of $F$ is given by the space of loops in the universal cover of $U/O(1)$ and is therefore contractible; the space of further choices is contractible too: \cite[Theorem C.1.10]{McDuff-Salamon} implies the Cauchy-Riemann operator is automatically surjective, without stabilisation). Note the index vector space $\bV_{D_0}$ (i.e. the fibre of $\bV$ over the 0-simplex $D_0$) is 0.

        Let $f: [n] \to [n']$ be an injection. Let $D_0^{\oplus n'-n}$ be the 0-simplex in $\bU(0,n'-n)$ given by applying $\oplus$ to $(n'-n)$ copies of $D_0$. We define a map of Thom simplicial sets:
        \begin{equation}
            \left(\bU(v,n)_\bullet, \bV\right) \to \left(\bU(v,n')_\bullet, \bV\right)
        \end{equation}
        given by applying $\oplus$ to the bijection $[n] \sqcup [n'-n] \to [n']$ which is $f$ on the first factor (and can be anything in the second factor, it does not affect the construction), and inserting $D_0^{\oplus n'-n}$.

        This assembles to make $(\bU(v, \cdot), \bV)$ into an $\cI$-object in $\Thom$.
    \end{defn}
    From \cite[Proposition 11.13]{Seidel:book}, we may compute the rank of the index bundle:
    \begin{lem} \label{lem:rank of index bundle}
        The rank of the vector bundle $\bV(v,n)$ is $v$.
    \end{lem}
    The above structures commute with each other in the appropriate sense. Combining them, we obtain:
    \begin{prop} \label{prop: wrifjeogiejoigv}
        $\left(\bU(\cdot, \cdot)_\bullet, \bV\right)$ forms a commutative monoid in the category of $\cI$-objects in the category of Thom $\cI$-simplicial sets.

        In particular, taking the diagonal $v=n$, we obtain a commutative Thom $\cI$-monoid $n \mapsto \left(\bU(n,n)_\bullet, \bV\right)$.
    \end{prop}
    It is this diagonal version that we will use throughout much of the paper: it is much more convenient to deal with a single $\cI$-factor.
    \begin{Not}
        We drop the $\cdot_\bullet$ to denote taking realisation of this simplicial space: for example, $\bU(v,n) := |\bU(v,n)_\bullet|$. 
        
        When it is clear which index bundle we are referring to, we often drop decorations and write $\bV$ for $\bV(v,n)$ (or its variants).
    \end{Not}
    \begin{lem}\label{lem: ev equiv}
        There is a weak equivalence of commutative $\cI$-monoids:
        \begin{equation}
            \bU \to \hofib(\Omega BO \to \Omega BS_\pm U)
        \end{equation}
    \end{lem}
    \begin{proof}
        Sending an abstract disc to the classifying map for the real part of its boundary conditions determines a map $\bU(n) \to \Omega BO(n)$. The classifying map for the complex part determines a nullhomotopy of the composite to $\Omega BS_\pm U(n)$; together, this defines the desired map. That this is an equivalence follows from the proof of \cite[Lemma 3.25]{PS2}: essentially this is because for fixed boundary conditions $(E,F)$, the space of all extra data as in Definition \ref{def: fix dom case} is contractible as $n \to \infty$.
    \end{proof}
    \begin{rmk}
        All the above constructions work if $BO$ and $BS_\pm U$ are replaced with $BO$ and $BU$ respectively. In that case, \cite[Lemma 3.27]{PS2} shows the classifying map for $\bV$, $\bU \to BO$, is an equivalence. Combined with Lemma \ref{lem: ev equiv}, this provides an index-theoretic proof of Bott periodicity $\Omega U/O \simeq BO$ as infinite loop spaces We elaborate on this further in Section \ref{sec: cat asp bott}.
    \end{rmk}
    \begin{defn}\label{def: E Psi}
        Let $\Psi = (\Theta, \Phi)$ be a commutative tangential pair. We define $\bU^\Psi$ to be the homotopy pullback of the following diagram:
        \begin{equation}\label{eq: hopullb}
            \xymatrix{
                \bU^\Psi 
                \ar[r]
                \ar[d]
                &
                \bU
                \ar[d]
                \\
                \hofib(\Omega \Theta \to \Omega \Phi) 
                \ar[r]
                &
                \hofib(\Omega BO \to \Omega BS_\pm U)
            }
        \end{equation}
        and $\bV^\Psi(n)$ to be pullback of $\bV(n)$. The pair $(\bU^\Psi(\cdot), \bV^\Psi)$ forms a commutative Thom $\cI$-monoid, which we write as $E_\Psi$. 
        
        Here $\bU$ is the ``diagonal'' $\cI$-monoid as in Proposition \ref{prop: wrifjeogiejoigv}, sending $n \mapsto |\bU(n,n)_\bullet|$.

    \end{defn}
    Note that the left vertical arrow of (\ref{eq: hopullb}) is an equivalence. 
    \begin{rmk}\label{rmk: I V}
        Let $\cV$ be the topologically enriched category of finite-dimensional real vector spaces equipped with inner products, and morphisms given by isometric embeddings. $\cI \subseteq \cV$ embeds a subcategory, by sending $[n]$ to $\bR^{[n]}$. 
    
        Throughout this section we worked with diagrams of spaces indexed by $\cI$ rather than $\cV$; one reason for this is to stay closer to the framework of \cite{Schlichtkrull}, which will be helpful in Section \ref{sec: OC}. In \cite{PS2} we indexed all constructions by $\cV$. All the constructions given in this section work over $\cV$ without issue.

        In particular, if $f,f': [n] \to [m]$ are two injective maps and $n<m$, then the maps $\bU^\Psi(n) \to \bU^\Psi(m)$ induced by $f$ and $f'$ are homotopic: using the above observation, it follows from the fact that the two maps $\bR^n \to \bR^m$ induced by $f$ and $f'$ are homotopic through isometric embeddings.
    \end{rmk}
    \begin{rmk}
        It follows from \cite[Lemma 3.24]{PS2} that $\bU^\Psi$, and hence $E_\Psi$, are convergent.
    \end{rmk}

\subsection{Categorical aspects of Bott periodicity}\label{sec: cat asp bott}

    We define the category of \emph{oriented polarised tangential structures} to be the category of simply connected commutative $\cI$-monoids $F$ over $\widetilde{U/O}^{or}$.  Definition \ref{defn: tang str pol} determines a functor $F \mapsto \Psi^F$ from polarised tangential structures to oriented commutative tangential pairs.

    In practice, sometimes the bordism theory we would like to study is given to us (in the form of a commutative Thom $\cI$-monoid, say) instead of the commutative tangential pair. In this subsection we prove the following proposition, which allows us to go back and forth between these two notions, at least in the polarised case. This is important for applications in \cite{PS4}.

    \begin{prop}\label{prop: mode baas sull}

        The assignment $F \mapsto E_{\Psi^F}$ defines a functor from oriented polarised tangential structures to connected oriented commutative Thom $\cI$-monoids, which induces an equivalence on homotopy categories.
    \end{prop}
    The proof involves combining existing results in the literature.

    We write $\cC\cS^\cI$ for the category of commutative $\cI$-monoids. For $X \in \cC\cS^\cI$, we write $\cC\cS^\cI_{/X}$ for the slice category of objects over $X$.
    
    $\cC\cS^\cI$ admits a model structure with weak equivalences those that induce isomorphisms on all homotopy groups \cite[Proposition 3.1]{Sagave-Schlichtkrull:Gp}. $\cC\cS^\cI$ is Quillen equivalent to (and hence has an equivalent homotopy category to) the category of $\bE_\infty$-algebras in spaces \cite[Theorem 1.2]{Sagave-Schlichtkrull:Diagram}.
    
    The slice category of objects over $X$ inherits a model category structure too, with weak equivalences (respectively fibrations, cofibrations) those whose underlying map in $\cC\cS^\cI$ is a weak equivalence (respectively fibration, cofibration) \cite[Theorem 7.6.5]{Hirschhorn:book}. 

    The functor $\Omega(\cdot): \cC\cS^\cI \to \cC\cS^\cI$ given by taking based loop spaces levelwise has a left adjoint $\bB$, which together form a Quillen adjunction \cite[Corollary 4.9]{Sagave-Schlichtkrull:Gp}. $\Omega$ shifts homotopy groups up by one on fibrant objects \cite[Lemma 2.4]{Sagave-Schlichtkrull:Gp}, $\bB$ shifts homotopy groups down by one on cofibrant objects \cite[Proposition 4.2]{Sagave-Schlichtkrull:Gp}. They therefore induce a Quillen equivalence (cf. \cite[Definition 1.3.12]{Hovey}), and hence an equivalence of homotopy categories, between connected commutative $\cI$-monoids and all commutative $\cI$-monoids.

    \begin{rmk}
        By the above arguments, there is an equivalence $\Omega \bB O \simeq O$ of commutative $\cI$-monoids. By applying the $\Omega-B$ adjunction levelwise there is also an equivalence $\Omega BO \simeq O$. It follows that there is an equivalence $BO \simeq \bB O$.
    \end{rmk}

    \begin{proof}[Proof of Proposition \ref{prop: mode baas sull}]
        We write $E$ as a composition of several functors, and argue that they all induce equivalences on homotopy categories:
        \begin{enumerate}
            \item Applying the functor $\Omega(\cdot)$.
            
            \item Taking homotopy pullback as in (\ref{eq: hopullb}) (but with $\bU^{or}$ instead of $\bU$).

            \item The above gives some object in $\cC\cS^\cI_{/\bU^{or}}$; we then push forwards along the classifying map for the index bundle $\bU^{or} \to BSO$.
        \end{enumerate}

        As observed above, \cite{Sagave-Schlichtkrull:Gp} shows (1) induces an equivalence on homotopy categories. (2) is obtained by applying fibrant replacement, follows by base-change along an equivalence, and hence induces an equivalence of homotopy categories by \cite[Proposition 2.7]{Rezk}. \cite[Proposition 2.7]{Rezk} also implies (3) induces an equivalence on homotopy categories (note $\cC\cS^\cI$ is proper \cite[Proposition 3.5]{Sagave-Schlichtkrull:Diagram}).
    \end{proof}

\subsection{Abstract discs}\label{sec:abstract discs and the Thom I-space}
    In order to connect to previous work in \cite{PS2}, we will use some generalisations of the construction in Section \ref{sec:abst disc fix case}.

    \begin{defn}\label{def: RSij}
        We define $\cR\cS_{ij}$ to be the simplicial set whose $k$-simplices consist of a smooth family of Riemann surfaces $D \to \Delta^k$, parametrised by the standard $k$-simplex. Each fibre is the complement of $i+j$ boundary points in a nodal Riemann disc (one whose smoothing is biholomorphic to $D^2$), for some\footnote{We will mostly care about the cases $j\in \{0,1\}$.} $i, j\geq 0$; $i$ of the boundary punctures are ordered anticlockwise and labelled incoming, and the remaining $j$ are ordered similarly and are labelled outgoing. Over each open stratum of $\Delta^k$ the topological type of the fibres is constant. The fibres are allowed to develop more nodes along deeper boundary strata.  Each simplex additionally includes (i) when $j=0$,  a marked point $p_x$ in the interior of each $D_x$ with a distinguished direction in $T_{p_x} D_x$, varying continuously in $x$, and which is required to ``point to the boundary of $D$'' (i.e. not to a puncture: more precisely, the geodesic in that direction hits the boundary $\partial D$ at a non-puncture); (ii) data of strip-like ends near each node and incoming / outgoing strip-like ends at the corresponding boundary punctures, as well as (iii) a metric on the Riemann surface which is standard over the strip-like ends. All this data  satisfies some mild constraints, cf. \cite[Sections 3.2-3.3]{PS2}.
    \end{defn}

    \begin{lem}
        The resulting simplicial sets $\cR\cS_{ij}$ are contractible Kan complexes.
    \end{lem}
    For each $1 \leq l \leq i'$, there is a map $\#_l: \cR\cS_{i1} \times \cR\cS_{i'j} \to \cR\cS_{i+i'-1,j}$ given by concatenating the output of a nodal Riemann disc to the $l^{th}$ input of another. These maps are suitably associative.

    Fix some commutative tangential pair $\Psi = (\Theta, \Phi)$. Let $F$ be the homotopy fibre of the map $\Theta \to \Phi$; $F$ maps naturally to $\widetilde{U/O}$.
    \begin{defn}
        \emph{A $\Psi$-oriented puncture datum} of \emph{rank $N$} consists of:
        \begin{itemize}
            \item A path $\gamma$ from $i\bR^n$ to $\bR^n$ in $U/O(N)$.
            
            Let $\widetilde \gamma$ be the unique lift of this path to the universal cover $\widetilde{U/O}(N)$ sending $1$ to the basepoint.
            \item A homotopy lift of $\widetilde \gamma$ to a path in $F(N)$.
        \end{itemize}
        By taking products with $\bR_{\geq 0}$ (or $\bR_{\leq 0}$), $\Psi$-oriented puncture data determines a map from $\bR_{\geq 0} \times [0,1]$ (or $\bR_{\leq 0} \times [0,1]$) to $\widetilde{U/O}(N)$ and a homotopy lift to $F$.
        
        By abuse of notation we sometimes refer to this latter data as the puncture datum.
    \end{defn}
    \begin{rmk}\label{rmk: 2 pts punc dat}
        Any $\Psi$-oriented puncture datum $\bE$ of rank $N$ determines a pair of points $p_\bE^\pm \in \Theta(N)$, taken to be the given lifts of $i\bR^n$ and $\bR^n$ respectively to $F(N)$, pushed forwards to land in $\Theta(N)$.
    \end{rmk}
    Any puncture data has an associated Maslov index $\mu \in \bZ$ (cf. \cite[Section 11(h)]{Seidel:book}).
    \begin{ex}
        The standard puncture datum from Definition \ref{def: std pun data} gives an example of $(BO, BS_\pm U)$-oriented puncture datum of rank $1$ and Maslov index 0. 
    \end{ex}
    \begin{rmk}
        We may take direct sum of $\Psi$-oriented puncture datum in a natural way: for $\bE$ and $\bE'$ $\Psi$-oriented puncture data of ranks $n$ and $n'$ respectively and a bijection $[n] \sqcup [n'] \to [n'']$, $\bE \oplus \bE'$ is a $\Psi$-oriented puncture datum of rank $n''$.

        For an injection $[n] \to [m]$, we may direct sum $\bE$ with $m-n$ copies of the standard puncture datum to obtain a puncture datum of rank $m$.
        
        Together, this endows the set of $\Psi$-oriented puncture data with the structure of a commutative $\cI$-monoid (in sets).
    \end{rmk}
    The following definition generalises Definition \ref{def: fix dom case} to allow the domain Riemann surface to vary within $\cR\cS_{ij}$ (in particular, it can have more than one puncture), and to incorporate other puncture data. 
    \begin{defn}\label{def: big abst disc defn}
        Let $\Psi$ be a commutative tangential pair. Let $i, j \geq 0$. Let $\bE=(\bE^{in}_1, \ldots, \bE^{in}_i, \bE^{out}_1, \ldots, \bE^{out}_j)$ an ordered tuple of puncture data.
    
        For each $v, n \in \cI$, there is a Thom simplicial set $\bU_{\bE}^{\Psi}(v,n)=(\bU^{\Psi}_{\bE}(v,n)_\bullet, \bV^{\Psi}_{\bE}(v,n))$, constructed as follows. A $k$-simplex in $\bU^{\Psi}_{\bE}(v,n)_k$ consists of:
        \begin{itemize}
            \item A simplex $D \to \Delta^k$ in $\cR\cS_{ij}$.
            \item A complex vector bundle $E$ over $D$ of rank $n$, along with a classifying map $E: D \to BU(n)$.
            \item A real subbundle over the boundary $F \subseteq E|_{\partial D}$, along with a classifying map $F: \partial D \to BO(n)$, compatible with that of $E$ (so $F \otimes \bC \simeq E$ over $\partial D$). 
            \item Compatible homotopy lifts of the classifying maps of $F$ and $E$ to $\Theta$ and $\Phi$.

            We require that over each strip-like end near each boundary puncture, the classifying maps and vector bundles agree with those determined by the puncture data $\bE_x := \bE_l^{in \textrm{ or }out}$ corresponding to each puncture $x$.
            \item A smooth family of 1-forms $Y_x \in \Omega^{01}(D_x, E_x)$, for $x \in \Delta^k$.
    
            We call all the data so far ($\Psi$-oriented) \emph{Cauchy-Riemann data on $D$}. For each fibre over $x \in \Delta^k$, using the Cauchy-Riemann data one can construct a Fredholm\footnote{More precisely: this is Fredholm after passing to appropriate Sobolev completions on both the domain and codomain,  and assuming that the Lagrangian boundary conditions are transverse at all boundary punctures, cf. \cite[Definition 3.10]{PS2}.} operator, called a \emph{Cauchy-Riemann operator}:
            \begin{equation}
                D^{CR}_x = \overline \partial + Y_x: \Gamma(E|_{D_x}, F|_{\partial D_x}) \to \Omega^{01}(D_x, E|_{D_x})
            \end{equation}

            \item A family of linear maps (called \emph{stabilisations}) $f_x: \bR^v \to \Omega^{01}_{D_x} (E|_{D_x})$ such that $f_x$ is compactly supported and $D^{CR}_x + f_x$ is surjective, for all $x$.

        \end{itemize}
    
        We impose mild assumptions on all this data, such as local constancy over the strip-like ends. We call these simplicial sets $\bU_{\bE}^{\Psi}(v,n)_\bullet$ \emph{spaces of abstract discs}.
    \end{defn}

    Forgetting all data except the domain $D$ defines maps $\dom: \bU^\Psi_{\bE}(v,n) \to \cR\cS_{ij}$ for all $v,n$.
 
\begin{rmk}
We will sometimes suppress $\Psi$ from the notation, to mean we work with the universal case $\Psi = (BO \to BS_\pm U)$. 
\end{rmk}

    \begin{lem}
        Each $\bU^\Psi_{\bE}(v, n)_\bullet$ is Kan. 
    \end{lem}
    \begin{proof}
        This is a consequence of a standard linear gluing argument.
    \end{proof}
    A similar argument shows:
    \begin{lem}
        Let $W_\bullet$ be the subsimplicial set of $\bU^\Psi_{\bE}(v, n)_\bullet$ consisting of abstract discs whose domain has no nodes. Then the inclusion $W_\bullet \to \bU^\Psi_{\bE}(v, n)_\bullet$ is an equivalence.
    \end{lem}

    \begin{lem}[{\cite[Lemma 3.18]{PS2}}]
        There is a canonical (up to contractible choice) vector bundle $\bV_{\bE}(v,n)$ over $|\bU_{\bE}(v,n)_\bullet|$, called the \emph{index bundle}.

        Each $(\bU_{\bE}(\cdot, n)_\bullet, \bV_{\bE}(\cdot, n))$ then defines a Thom $\cI$-simplicial set in a natural way. 
        
    \end{lem}
    \begin{rmk}\label{rmk: vect}
        The fibres of the corresponding vector bundle are given by the kernels of the stabilised Cauchy-Riemann operator. Note however that even though the set of fibres is canonical, the structure of a vector bundle is not automatic, this requires construction (and some mild choices).

        The structure of a Thom $\cI$-simplicial set is given by maps which roughly replace the stabilisations $f_x: \bR^v \to \Omega^{01}(\ldots)$ with $f_x \oplus 0: \bR^v \oplus \bR^{w-v} \to \Omega^{01}(\ldots)$.
    \end{rmk}

        Given two sets of Cauchy-Riemann data and stabilisation data $(E,F,Y,f)$ and $(E',F',Y',f')$ on the same $k$-simplex $D \to \Delta^k$ in $\cR\cS_{ij}$, one can take their direct sum $(E\oplus E', F \oplus F', Y \oplus Y', f \oplus f')$ fibrewise over $\Delta^k$, and their index bundles are naturally additive under this operation. This defines maps of Thom simplicial sets: 
        \[
        \oplus: \bU^\Psi_{\bE}(v,n)_\bullet \times_{\cR\cS_{ij}} \bU^\Psi_{\bE'}(v',n')_\bullet \to \bU^\Psi_{\bE \oplus \bE'}(v+v',n+n')_\bullet
        \]
        
        These are compatible with stabilisation in the $v$ direction and so defines maps of Thom $\cI$-simplicial sets 
        \[
        \bU_{\bE}^\Psi(n)_\bullet \times_{\cR\cS_{ij}} \bU^\Psi_{\bE'}(n')_\bullet \to \bU^\Psi_{\bE\oplus\bE'}(n + n')_\bullet
        \]
        
        This is commutative, and if we consider the same construction applied to iterated fibre products, it is associative.

\begin{lem}[{\cite[Lemma 3.31]{PS2}}]\label{lem:wkhtpy}
    Let $\bE_{01}$ be a tuple of $\Psi$-oriented puncture data with no inputs and 1 output.

    There are weak equivalences of $\cI$-spaces $\bU_{\bE_{01}}(\cdot, n) \to \Omega U/O (n)$ and, more generally, $\bU^\Psi_{\bE_{01}}(\cdot, n) \to \Omega F$ (viewing the codomain as constant $\cI$-spaces in both cases).
\end{lem}

\begin{lem}
    Let $\bE_{01}$ be as in Lemma \ref{lem:wkhtpy}.

    The inclusion 
\[
\bU_{\bE_{01}}(n) \times_{\cR\cS_{01}} \bU_{\bE_{01}}(n') \to \bU_{\bE_{01}}(n) \times \bU_{\bE_{01}}(n')
\]
is an equivalence of Thom $\cI$-simplicial sets.
\end{lem}

\begin{proof}
    This follows from \cite[Lemma 3.31]{PS2}.
\end{proof}
    \begin{rmk}
       One can thus attempt to define a product 
       \[
       \bU_{\bE_{01}}(n) \times \bU_{\bE_{01}}(n') \to \bU_{\bE_{01}}(n+n')
       \]
       via a zig-zag:
        \begin{equation}\label{eq: zdfgegwsed}
            \bU_{\bE_{01}}(n) \times \bU_{\bE_{01}}(n') \xleftarrow{\simeq} \bU_{\bE_{01}}(n) \times_{\cR\cS_{01}} \bU_{\bE_{01}}(n') \xrightarrow{\oplus} \bU_{\bE_{01}}(n+n')
        \end{equation}
        This is useful as a way to think about this operation, but we do not give it an axiomatic treatment here.
    \end{rmk}

We may also concatenate puncture data as follows. Given sets of puncture data $\bE = (\bE^{in}_1, \ldots, \bE^{in}_i, \bE^{out})$ and $\bF = (\bF^{in}_1, \ldots, \bF_{i'}^{in}, (\bF^{out}))$ (we use the brackets $(\bF^{out})$ do indicate that we omit $\bF^{out}$ if appropriate) and some $1 \leq l \leq i'$, we may form a new set of puncture data:
\begin{equation}
    \bE \#_l \bF = \left(\bF^{in}_1, \ldots, \bF^{in}_{l-1}, \bE^{in}_1, \ldots, \bE^{in}_i, \bF^{in}_{l+1}, \ldots, \bF^{in}_i, (\bF^{out})\right)
\end{equation}

Given a pair of abstract discs $\bD \in \bU^\Psi_{\bE}(v, n)$ and $\bD' \in \bU^\Psi_{\bF}(v', n)$, and some $1 \leq l \leq i'$ such that the puncture data at the output of $\bD$ agrees with that of the $l^{th}$ input of $\bD'$, we may form a new abstract disc:
\begin{equation} 
    \bD \#_l \bD' \in \bU^\Psi_{\bE\#_l \bF}(v+v', n)
\end{equation}
This can be upgraded to a map of simplicial sets, and can be made compatible with the index bundles:
\begin{lem}[{\cite[Lemma 3.18]{PS2}}]
    Let $[v] \sqcup [v'] \to [v'']$ be an injection. We write $v''-v-v'$ for the complement of the image.

    There are maps of Thom simplicial sets:
    \begin{equation}
        \#_l: \left(\bU^\Psi_{\bE}(v, n), \bV\right) \times \left(\bU^\Psi_{\bE'}(v', n), \bV\right) \oplus \bR^{v''-v-v'} \to \left(\bU^\Psi_{\bE \#_l \bE'}(v'', n), \bV\right)
    \end{equation}

    These maps are suitably associative, and commute with stabilisation in the $v$-direction. Furthermore they are compatible with the concatenation maps $\#_l$ on the $\cR\cS_{ij}$ under the map $\dom$.
\end{lem}

\begin{rmk}
    Though it does not require any choices to define the set of fibres of the index bundles $\bV$, to endow it with the structure of a vector bundle requires making some contractible choices. The proof of \cite[Lemma 3.18]{PS2} makes these choices via an inductive scheme over simplices, so that they can be made compatibly with concatenation.
\end{rmk}

By construction, the concatenation maps $\#_l$ are also compatible with direct sum in the $n$-direction, as well as stabilisation in the $v$ direction.

\begin{lem}[{\cite[Section 3.3]{PS2}}]
    Let $\bE=(\bE^{in}_1, \ldots, \bE^{in}_i, \bE^{out}_1, \ldots, \bE^{out}_j)$ be a choice of puncture data. Then:
    \begin{equation}\label{eq: rank}
        \operatorname{Rank}(\bV^\Psi_{\bE}(v,n)) = v+\mu(\bE) + n(1-j)
    \end{equation}
    
    where $\mu(\bE)$ is the \emph{total Maslov index} $\sum_l\mu(\bE^{in}_l) - \sum_k\mu(\bE^{out}_k)$.
\end{lem}

Let $\bE^{std, fr}$ be some choice of rank one framed puncture data of Maslov index 0, and $\bE^{std, fr}(n)$ its $n$-fold direct sum. Let $\bE^{std, \Psi}(n)$ be the induced $\Psi$-oriented puncture data (of any rank $n$). Let $\bE^{std, \Psi}_{ij}(n)$ be the tuple of puncture data with $i$ inputs and $j$ outputs, all of which are $\bE^{std, \Psi}(n)$. We write $\bU^\Psi_{ij}(v, n)$ for $\bU^\Psi_{\bE^{std,\Psi}_{ij}(n)}(v,n)$.

\begin{rmk}\label{rmk: fix dom case}
    Note that the space $\bU^\Psi(n)$ from Section \ref{sec:abst disc fix case} is a subspace of $\bU^\Psi_{01}(n, n)$: it is a fibre of $\dom$.
\end{rmk}

The maps $\dom$ have coherent sections:
\begin{lem}\label{lem: cr sect}

    There are sections $st=st_{ij}: \cR\cS_{ij} \to \bU^\Psi_{ij}(v, n)$ of $\dom$.
    These can be chosen to be compatible with concatenation $\#_l$.

\end{lem}
When $j=0$, by (\ref{eq: rank}) we have that $\bV^\Psi_{ij}(0,n) = 0$, so the maps $st$ are automatically maps of Thom spaces.
\begin{proof}
    We carry this out in the case $v=0$; by stabilising the general case then follows.

    Because Cauchy-Riemann operators are not necessarily surjective, $\bU^{fr}_{ij}(0,n)$ is not necessarily contractible. However, in the case $n=1$, \cite[Theorem C.1.10]{McDuff-Salamon} implies that it is contractible. Hence we may choose sections $st: \cR\cS_{ij} \to \bU^{fr}_{ij}(0,1)$ for all $i,j$, and the space of such choices is contractible. By making these choices over simplices inductively over the difference between the maximum and minimum number of nodes of the Riemann surface (similarly to the proof of \cite[Lemma 3.18]{PS2}), we may choose these to be compatible with concatenation, as required.

    By taking iterated direct sums, we obtain the desired maps in the case $n > 1$; by composing with the map of commutative tangential pairs $fr \to \Psi$, we obtain them for any $\Psi$.
\end{proof}

Using this, when there are $j=0$ outputs, there are maps of Thom simplicial sets for each injection $[n] \to [n']$:
\begin{equation}
    \left(\bU^\Psi_\bE(v, n)_\bullet, \bV\right) \to \left(\bU^\Psi_{\bE \oplus \bE^{std, \Psi}_{ij}(n'-n)}(v, n')_\bullet, \bV\right)
\end{equation}
On base spaces, this sends a simplex $\sigma \mapsto \sigma \oplus st\circ\dom(\sigma)$; there is a canonical isomorphism of vector bundles covering this coming from the fact that the index bundle over the codomain of $st_{\bE^{std, \Psi}_{ij}(n)}$ is $0$ when $j=0$. These maps are compatible with all other structures constructed in this section thus far.

There is a similar story when there are $j=1$ outputs. In this case, using a similar inductive argument, we can choose trivialisations $\bV^{fr}_{i0}(0,1) \cong \bR$ compatibly with concatenation, and hence obtain maps of Thom spaces $st: (\cR\cS_{i0}, \bR) \to (\bU^{fr}_{i0}(0,1), \bV)$. Direct summing with these maps, similarly to above (but now with $j=1$), provides us maps of Thom simplicial sets for each injection $[n] \to [n']$:
\begin{equation}
    \left(\bU^\Psi_\bE(v,n)_\bullet, \bV\right) \oplus \bR^{n'-n} \to \left(\bU^\Psi_{\bE \oplus \bE^{std, \Psi}_{ij}(n'-n)}(v, n')_\bullet, \bV\right)
\end{equation}

\begin{rmk}
    All structures constructed in this section are functorial in $\Psi$.
\end{rmk}

\subsection{Homotopy types of spaces of abstract discs}\label{sec: htpy type}
    Let $\Psi=(\Theta \to \Phi)$ be a commutative tangential pair.
    \begin{defn}
        Let $\bE = (\bE^{in}_1, \ldots, \bE^{in}_i, (\bE^{out}))$ be a tuple of puncture data with $i$ inputs and $j \in \{0,1\}$ outputs. Each puncture datum $\bF \in \bE$ determines a pair of points $p_\bF^\pm$ in $\Theta(n)$, whose images in $\Phi(n)$ agree, as in Remark \ref{rmk: 2 pts punc dat}.
        
        If $\bE$ is non-empty (i.e. $i+j \neq 0$), we define $\Omega_\bE \Theta(n)$ to be the space of $(i+j)$ paths in $\Theta(n)$ with endpoints specified as follows; cf. Figure \ref{fig: qqq}. These are paths $p^-_{\bE^{in}_l}$ to $p^{+}_{\bE^{in}_{l+1}}$ for $0 \leq l \leq i-1$; then if $j=0$, a path $p^-_{\bE^{in}_i}$ to $p^+_{\bE^{in}_0}$, and if $j=1$ two paths: one from $p^-_{\bE^{in}_i}$ to $p^-_{\bE^{out}}$ and one from $p^+_{\bE^{out}}$ to $p^+_{\bE^{in}_1}$.

\begin{figure}[ht]
\begin{center}
\begin{tikzpicture}

\draw[fill] (3,1) circle (0.1);
\draw[fill] (3,-1) circle (0.1);
\draw[fill] (-3,3) circle (0.1);
\draw[fill] (-3,2) circle (0.1);
\draw[fill] (-3,1) circle (0.1);
\draw[fill] (-3,-1) circle (0.1);
\draw[fill] (-3,-2) circle (0.1);
\draw[fill] (-3,-3) circle (0.1);

\draw (-2.5,0) circle (0.05);
\draw (-2.5,0.25) circle (0.05);
\draw (-2.5,-0.25) circle (0.05);

\begin{scope}[thick,decoration={
    markings,
    mark=at position 0.5 with {\arrow{triangle 60}}}
    ] 

\draw[semithick,rounded corners=10mm,postaction={decorate}] (-3,3) -- (0,3) -- (0,1) -- (3,1);
\draw[semithick,rounded corners=10mm,postaction={decorate}] (-3,-3) -- (0,-3) -- (0,-1) -- (3,-1);
\draw[semithick, rounded corners=5mm,postaction={decorate}] (-3,2) -- (-2,2) -- (-2,1) -- (-3,1);
\draw[semithick, rounded corners=5mm,postaction={decorate}] (-3,-1) -- (-2,-1) -- (-2,-2) -- (-3,-2);
\end{scope}

\draw (4,1) node {$p^+_{\bE^{out}}$};
\draw (4,-1) node {$p^-_{\bE^{out}}$};

\draw (-4,3) node {$p^+_{\bE_1^{in}}$};
\draw (-4,2) node {$p^-_{\bE_1^{in}}$};
\draw (-4,1) node {$p^+_{\bE_2^{in}}$};
\draw (-4,-2) node {$p^+_{\bE_i^{in}}$};
\draw (-4,-3) node {$p^-_{\bE_i^{in}}$};

\end{tikzpicture}
\end{center}
\caption{Schematic picture of an element of $\Omega_\bE \Theta(n)$
          \label{fig: qqq}}
\end{figure}
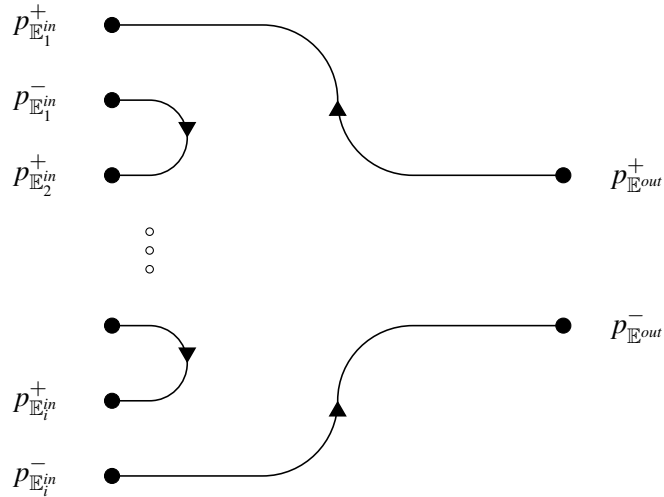

        $\Omega_\bE \Theta(n)$ admits a natural map to the free loop space $\cL \Phi(n)$, by composing all the paths with the map $\Theta(n) \to \Phi(n)$, and then concatenating them. We define $\Psi_\bE(n)$ to be the space of pairs $(\gamma, d)$, where $\gamma \in \Omega_\bE(\Theta(n))$ and $d: D^2 \to \Phi(n)$ is a filling of the corresponding loop in $\cL \Phi(n)$.
        
        If $\bE$ is empty, we define: $\Psi_{00}(n) = \Psi_{\emptyset}(n)$ to be the homotopy pullback of:
        \begin{equation}
            \xymatrix{
                &
                \Phi(n)
                \ar[d]
                \\
                \cL \Theta(n)
                \ar[r]
                &
                \cL \Phi(n)
            }
        \end{equation}  
        where the vertical arrow is the inclusion of constant loops.
    \end{defn}
    \begin{rmk}
        For $\bE$ non-empty, $\Omega_\bE \Theta(n)$ is non-canonically homotopy equivalent to $(\Omega \Theta(n))^{i+j}$. Additionally choosing a distinguished puncture determines an equivalence:
        \begin{equation}
            \Psi_\bE(n) \to \hofib(\Omega_\bE \Theta(n) \to \Omega \Phi(n))
        \end{equation}
        Note that the homotopy types of these spaces only depends on $i$ and $j$.
    \end{rmk}
    Concatenation defines natural maps:
    \begin{equation}
        \#_l: \Psi_\bE(n) \times \Psi_{\bE'}(n) \to \Psi_{\bE \#_l \bE'}(n)
    \end{equation}
    \begin{rmk}
        Thus far, none of the constructions in Section \ref{sec: htpy type} have used the monoidal structure on $\Theta$ or $\Phi$.
    \end{rmk}
    For a bijection $[n] \sqcup [n'] \to [n+n']$, the monoid structures on $\Theta$ and $\Phi$ determine maps:
    \begin{equation}
        \oplus: \Psi_{\bE}(n) \times \Psi_{\bE'}(n') \to \Psi_{\bE \oplus \bE'}(n+n')
    \end{equation}
    when $\bE$ and $\bE'$ have the same numbers of inputs and outputs.

    \begin{rmk}
        These constructions simplify in the case $\Phi = \pt$, and $\Theta$ comes from some $F \to \widetilde{U/O}$ as in Example \ref{ex: tang str}. In this case, if $\bE$ is non-empty, $\Psi_\bE(n)$ is equivalent to $(\Omega F(n))^{i+j}$, $\#_l$ concatenates loops in an appropriate way, and $\Psi_\emptyset(n) \simeq \cL F(n)$.
    \end{rmk}

    \begin{defn}
        There are maps of spaces:
        \begin{equation}
            ev=ev^\Psi_\bE(v, n): \bU^\Psi_\bE(v, n) \to \Psi_\bE(n)
        \end{equation}
        defined by sending an abstract disc to the given homotopy lifts of its classifying maps (cf. Definition \ref{def: big abst disc defn}) to $\Theta$ and $\Phi$.
    \end{defn}
    The maps $ev$ are compatible with all the relevant structures (direct sum, stabilisation in the $v$ direction, and concatenation).

    Generalising Lemma \ref{lem:wkhtpy}, we have:

    \begin{lem}[{\cite[Lemma 3.31]{PS2}}]
        The connectivity of the map $ev^\Psi_\bV(v, n)$ goes to $\infty$ as $v \to \infty$.

        In particular, $ev^\Psi_\bV(\cdot, n)$ defines an equivalence of $\cI$-spaces from $\bU^\Psi_\bE(\cdot, n)$ to $\Psi_\bE(n)$, viewing the codomain as a constant $\cI$-space.
    \end{lem}

    \begin{rmk}
        The theory still makes sense without the final piece of data in Definition \ref{def: big abst disc defn}, but is useful for our purposes so that the resulting homotopy type is ``correct'': for example, without it, if $\Phi \simeq \pt$, we would have a highly-connected (as $v \to \infty$ map to the homotopy orbits $\bU_{00}^{\Theta,\pt}(v, n) \to (\cL \Theta(n))_{hS^1}$, instead of $\cL \Theta(n)$.
    \end{rmk}

\section{Flow categories over an $\cI$-monoid}\label{sec: flow}

Let $E$ denote a Thom $\cI$-monoid (i.e. a monoidal object in Thom $\cI$-spaces), which is furthermore commutative.

\subsection{Unoriented flow categories}

An unoriented flow category $\cF$ comprises
\begin{enumerate}
    \item a finite set of objects $x \in \cF$;
    \item for $x,x' \in \cF$ a smooth manifold with corners $\cF_{xx'}$;
   \item embeddings $\cF_{xx'} \times \cF_{x'x''} \to \cF_{xx''}$ whose images cover and enumerate the codimension one boundary strata. 
    \end{enumerate}

    We will always assume that we have a degree function
    \[
    |\cdot| : \mathrm{Ob}\,\cF \to \bZ
    \]
    such that $\cF_{xy}$ has dimension $|x|-|y|-1$.

 Let $M^n$ be a manifold with corners, and $p$ a point in $M$. Then there is a chart near $p$ $\phi: U \subseteq M \rightarrow V \subseteq \bR^{n-i} \times \bR_{\geq 0}^{i}$ sending $p$ to $0$, where $i = \Gamma(p)$ is the codimension of $p$ in $M$.

 \begin{defn}
     Let $M$ be a manifold with corners. Then $M$ is a manifold with faces if and only if for each point $p$ in $M$, $p$ lies in the closure of exactly $\Gamma(p)$ boundary faces.
 \end{defn}

 We write $D(M)$ for the set of codimension 1 connected faces of $M$, and $D({F \subseteq M})$ for the set of codimension 1 (connected) faces of $M$ which touch $F$.  
 
 A \emph{system of collars} $\cC$ on $M$ consists of collar neighbourhoods
            $$\cC=\cC^{FG}: F \times [0, \varepsilon)^{D({F \subseteq G})} \hookrightarrow G$$
            meaning an embedding onto a neighbourhood of $F \subseteq G$, whenever $F$ and $G$ are faces of $M$ such that $F \subseteq G$. We require that the collars restrict to each other on subfaces, and come with maps 
            $$\cC: F \times [0, \varepsilon) \hookrightarrow M$$ 
            for all boundary faces $F$ such that for all distinct pairs of boundary faces $F_i, F_j$, the following diagram commutes,
            \begin{equation}\label{star}
            \xymatrix{
                F_i \cap F_j \times [0, \varepsilon)^2 \ar[rrr]^{\cC^{F_i} \times \id_{[0, \varepsilon)}} \ar[d]_{\left(\cC^{F_j} \times \id_{[0, \varepsilon)}\right) \circ \tau}  &&& F_i \times [0, \varepsilon) \ar[d]^{\cC} \\
                F_j \times[0, \varepsilon) \ar[rrr]_{\cC} &&& M
            }
            \end{equation}
            where $\tau$ swaps the two factors in $[0, \varepsilon)^2$.

            A \emph{system of collars} on a flow category $\cF$ consists of a system of collars $\cC_{xy}$ on all $\cF_{xy}$, restricting to the product system of collars on $\cF_{xz} \times \cF_{zy}$ induced from $\cC_{xz}$ and $\cC_{zy}$.  Systems of collars always exist, are unique up to isotopy, and can be chosen to extend (compatibly) given collars on subsets of boundary faces.

The abstract index bundle $I^{\cF}_{xy} \to \cF_{xy}$ is $T\cF_{xy} \oplus \bR$, with the final factor generated by an element $\tau_y$.  

We now fix `stabilisation data' for the flow-category $\cF$, comprising non-negative integers $\nu_{xx'}$, for $x, x' \in \cF$. Each such defines a finite-dimensional inner product space $\bR^{\nu_{xx'}}$. We insist there are injections
\begin{equation} \label{eqn:inject stabilisation data}
[\nu_{xx'}] \sqcup [\nu_{x'x''}] \to [\nu_{xx''}]
\end{equation}
defining  isometric embeddings 
\[
\bR^{\nu_{xx'}} \oplus \bR^{\nu_{x'x''}} \to \bR^{\nu_{xx''}}
\]
and we insist that the injections \eqref{eqn:inject stabilisation data} are associative for quadruples, meaning in particular that 
\begin{equation*}
                \xymatrix{
                    \bR^{\nu_{xx'}} \oplus \bR^{\nu_{x'x''}} \oplus \bR^{\nu_{x''x'''}}
                    \ar[r]
                    \ar[d]
                    &
                    \bR^{\nu_{xx'}} \oplus \bR^{\nu_{x'x'''}}
                    \ar[d]
                    \\
                    \bR^{\nu_{xx''}} \oplus \bR^{\nu_{x''x'''}}
                    \ar[r]
                    &
                    \bR^{\nu_{xx'''}}
                }
            \end{equation*} 
commutes. 

\begin{defn}\label{defn:shoehorn in an I-space}
    The Thom $\cI$-space $(\cF_{xx'}, I^{\cF}_{xx'})$ has
\[
(\cF_{xx'}, I^{\cF}_{xx'})(n) := I^{\cF}_{xx'}\oplus\bR^n \to \cF_{xx'}
\]
\end{defn} 

\begin{rmk}
To simplify notation, we will often write $(\cF_{xx'},I^{\cF})$ in place of $(\cF_{xx'}, I^{\cF}_{xx'})$, or write $I^{\cF}(\cdot)$ to indicate that we are considering  the Thom $\cI$-space rather than the index bundle.\end{rmk}

In the setting of Floer theory, and a tangential pair $\Psi$ with associated Thom $\cI$-monoid $E_{\Psi} = (\bU^{\Psi}, \bV^{\Psi})$, the vector bundle $\bV^{\Psi}(\nu) \to \bU^{\Psi}(\nu)$ has rank $\nu$ given by the dimension of the stabilisation vector space.  On the other hand, the vector bundle arising in Definition \ref{defn:shoehorn in an I-space} has rank determined by both the $\cI$-index $\nu$ and the rank $|x|-|x'|$ of the index bundle $I^{\cF}_{xx'}$.

This leads us to introduce a variation of the notion of `stabilisation data'.

\begin{defn}\label{def: ind data}
    Index data for a flow category $\cF$ comprises choices of finite sets $d_{xx'}$ of cardinality $|x|-|x'|$, for each $x, x' \in \cF$, such that there are (not necessarily order-preserving) injections
    \[
    d_{xx'} \sqcup d_{x'x''} \to d_{xx''}
    \]
    which satisfy that for objects $x,x',x''$ the diagram
    \[
    \xymatrix{
    d_{xx'} \sqcup d_{x'x''} \sqcup d_{x''x'''} \ar[r] \ar[d] & d_{xx'} \sqcup d_{x'x'''} \ar[d] \\
    d_{xx''} \sqcup d_{x''x'''} \ar[r] & d_{xx'''}
    }
    \]
    commutes.
\end{defn}

\begin{rmk}\label{rmk:change order}
    There are analogues of stabilisation and index data for morphisms of flow categories, bordisms of morphisms, bilinear maps, associators etc.  When considering index data for a bilinear map, one encounters a diagram of the form
    \[
    \xymatrix{
    d_{xx'} \sqcup d_{yy'} \sqcup d_{x'y';z} \ar[r] \ar[d] & d_{xx'} \sqcup d_{x'y;z} \ar[d] \\
    d_{yy'} \sqcup d_{xy';z} \ar[r] & d_{xy;z}
    }
    \]
    which must commute. At this point it is essential that the various $d_{xy;z}, d_{xx'}$ etc are finite sets which are not canonically ordered, and the morphisms of finite sets in the definition of index data are injections which need not be order-preserving (as for morphisms in sets of $\cI$). Geometrically this reflects the fact that the two breakings at the inputs of a bilinear map have no preferred order.
\end{rmk}

\begin{rmk} There is no natural $E_\infty$ or indeed $E_2$ map $\bZ \to BO\times\bZ$ (e.g. one splitting projection), but there is a natural map $\Omega^{\infty}(\bS) \to BO \times \bZ$, which gives a homotopical explanation for the appearance of (unordered) finite sets in the construction of index data.
\end{rmk}

The following is perhaps not obvious in light of Remark \ref{rmk:change order}; a more detailed argument, in the Floer-theoretic settings where we require the result, follows from the same argument that proves Lemma \ref{lem: ex big S}.

\begin{lem}
    Given a finite collection of flow categories, morphisms, bilinear maps, associators, and bordisms of such, index data exists.
\end{lem}

\begin{proof}[Sketch]
    For a fixed flow category $\scrF$, this is proved inductively in the dimension of $\cF_{xx'}$. More generally one considers the finite set of all moduli spaces arising in constructing the required structure, and works inductively in their dimension. 
\end{proof}

\subsection{Orientation relative to an $\cI$-monoid}\label{sec: or flow cat I}

Fix a monoid $E$ in Thom $\cI$-spaces. Thus we have base spaces $\Base(E(n))$ and bundles $E(n) \to \Base(E(n))$, with suitably compatible addition maps. 

To define an $E$-orientation on a flow category $\cF$, we will use two auxiliary pieces of data:
\begin{enumerate}
    \item stabilisation data $\{\nu_{xx'}\}$ for $x,x' \in \cF$;
    \item \emph{index data} $\{d_{xx'}\}$ for $x,x' \in \cF$,
\end{enumerate}
as defined in the previous section. For a finite set $d$ we write $\bR^d$ for $\bR$-$\mathrm{Span}\langle d \rangle$.

\begin{defn}\label{defn:E-oriented flow cat}
    An $E$-oriented flow category  comprises a flow category $\cF$ with stabilisation data $\{\nu_{xx'}\}$ and index data $\{d_{xx'}\}$,  together with maps of Thom spaces
   \begin{equation} \label{eqn:structure morphism to E-orient flow category}
   \rho: (\cF_{xx'},I^\cF_{xx'})(\nu_{xx'}) \to E(\nu_{xx'}+d_{xx'})
   \end{equation}
    for any $x,x' \in \cF$, which are compatible with and associative under composition. Note that $\rho$ incorporates a map of base spaces $\rho_{xx'}: \cF_{xx'} \to \Base(E(\nu_{xx'}+ d_{xx'}))$ and a covering isomorphism of bundles $\st_{xx'}: I^{\cF}_{xx'} \oplus \bR^{\nu_{xx'}} \to E(\nu_{xx'} + d_{xx'})$.
\end{defn}

\begin{rmk}
    We do not formulate this in terms of maps of (i.e. functors of categories enriched in) Thom $\cI$-spaces, because a map $I^{\cF} \oplus \bR^n \to E(n)$ which arises as part of such a functor necessarily lands in the $\Sym_n$-fixed points of $\Base(E(n))$. This imposes an unwanted homotopical constraint.
\end{rmk}

To clarify the compatibility requirements in Definition \ref{defn:E-oriented flow cat}, we have that the $\rho_{xx'}$ fit into commuting diagrams
\begin{equation*}
                \xymatrix{
                    \cF_{xx'} \times \cF_{x'x''} 
                    \ar[rr]_-{\rho_{xx'} \times \rho_{x'x''}}
                    \ar[d]
                    &&
                    \Base E(\nu_{xx'} + d_{xx'}) \times \Base  E(\nu_{x'x''} + d_{x'x''})
                    \ar[d]
                    \\
                    \cF_{xx''}
                    \ar[rr]_{\rho_{xx''}}
                    &&
                    \Base E(\nu_{xx''}+d_{xx''})
                }
            \end{equation*}
(where the vertical map is the monoidal composition on $\Base(E)$).

The covering isomorphisms satisfy
\begin{equation*}
\xymatrix{
I^{\cF}_{xx'} + \bR^{\nu_{xx'}} + I^\cF_{x'x''} + \bR^{\nu_{x'x''}} \ar[r] \ar[d] & E(\nu_{xx'} + d_{xx'}) + E(\nu_{x'x''}+ d_{x'x''}) \ar[d] \\
I^{\cF}_{xx''} + \bR^{\nu_{xx''}} \ar[r] & E(\nu_{xx''} + d_{xx''})
}  
\end{equation*}

\begin{rmk}
    The commutation rules for associativity of these injections for quadruples of objects will involve moving $\nu$'s past one another and $d$'s past one another, but never interchanging those; the structure is compatible with a partial order on sets in $\cI$ in which $\nu \leq d$ for all $\nu,d$. This follows immediately from the fact that the Thom spaces on the RHS of the diagram arise as postcompositions with natural maps $E(\nu_{xx'}) \oplus \bR^{d_{xx'}} \to E(\nu_{xx'} + d_{xx'})$.
\end{rmk}

\begin{defn}\label{def: E or mor}
Let $\cF$ and $\cG$ be flow categories over $E$, and let $\cW$ be a morphism from $\cF$ to $\cG$. Assume that all of $\cF,\cG$ and $\cW$ have been equipped with stabilisation data and index data.

Then $\cW$ is $E$-oriented if for $x,x' \in \cF$ and $y,y' \in \cG$ we have 
maps of Thom spaces
   \begin{equation} \label{eqn:structure morphism to E-orient morphism}
   \rho: (\cW_{xy},I^{\cW}_{xy})(\nu_{xy}) \to E[1](\nu_{xy}+d_{xy})
   \end{equation}
   which fit into diagrams 
   \[
   \xymatrix{
   (\cF_{xx'},I^{\cF})(\nu_{xx'}) \times (\cW_{x'y},I^{\cW})(\nu_{x'y}) \ar[r] \ar[d] & E(\nu_{xx'}+d_{xx'}) \times E[1](\nu_{x'y} + d_{x'y}) \ar[d] \\
   (\cW_{xy},I^{\cW})(\nu_{xy}) \ar[r] & E(1+\nu_{xy} + d_{xy})
   }
   \]
   (and similarly for breaking in $\cG$ on the right).  These diagrams should furthermore be associative under further breaking.
\end{defn}

\begin{rmk}
    The shifts $[1]$ account for the fact that, when $E = E_\Psi$, if $x,x' \in \cF$ and $y\in \cG$ then
\[
I^{\cF}_{xx'} \oplus V_{xx'} = \rho_{xx'}^*\bV_{xx'}
\]
whilst
\[
I^{\cW}_{xy} \oplus V_{xy} = \bR \oplus \rho_{xy}^*\bV_{xy}.
\]
Here, in the notation from \cite[Lemma 3.18]{PS2},  $V_{xx'} = \bR^{\nu_{xx'}}$ is the stabilisation space, $\rho$ denotes the domain map to the space of abstract discs and $\bV_{xx'}, \bV_{xy}$ are the corresponding index bundles over the codomain of $\rho$.
\end{rmk}

A bordism $\cR$ between morphisms $\cW$ and $\cW'$, all equipped with index and stabilisation data, has maps of Thom spaces
\[
(\cR_{xy}, I^{\cR}_{xy})(\nu_{xy}) \to E[2](\nu_{xy}+d_{xy})
\]
which satisfy an analogous set of compatibilities under breaking. For instance, since $\cW_{xy} \sqcup \cW'_{xy}$ occur as codimension one unbroken faces in $\cR_{xy}$, there are isomorphisms of Thom spaces
\[
I^{\cR}_{xy}(\cdot)|_{\cW_{xy}} = I^{\cW}_{xy}(1+\cdot)
\]
and similarly for $\cW'$.
 An $E$-oriented bordism has maps (of bundles over the boundary stratum $\cW_{xy} \sqcup \cW'_{xy}$)
\[
\xymatrix{
(\cW_{xy},I^\cW_{xy})(1+\nu_{xy}) \sqcup (\cW'_{xy},I^{\cW'}_{xy})(\nu_{xy}) \ar[r] \ar[d]  & E[1](1+\nu_{xy} + d_{xy}) \ar[d] \\ (\cR_{xy},I^{\cR}_{xy})(\nu_{xy}) \ar[r] & E[2](\nu_{xy} + d_{xy})
}
\]
and similarly has maps
\[
\xymatrix{
(\cF_{xx'},I^{\cF}_{xx'})(\nu_{xx'}) \times (\cR_{x'y},I^{\cR}_{x'y})(\nu_{x'y}) \ar[r] \ar[d] & E(\nu_{xx'} + d_{xx'}) \times E[2](\nu_{x'y} + d_{x'y}) \ar[d] \\
(\cR_{xy},I^{\cR}_{xy})(\nu_{xy}) \ar[r] & E[2](\nu_{xy} + d_{xy})
}
\]
(and similarly for breaking in $\cG$ on the right).

\begin{lem}
    There is a category $\Flow^E$ of flow categories living over $E$, with morphisms being bordism classes of $E$-oriented morphisms.
\end{lem}

\begin{proof}
    Identical to \cite[Section 5.2]{PS2} subject to appropriate changes in notation.
\end{proof}

\begin{ex} \label{ex:concordance of Flow orientations}
    Consider two $E$-orientations $\iota, \tilde{\iota}$ on $\cF$, given by maps $\iota_{xx'}: \cF_{xx'} \to \Base(E)$ and $\tilde{\iota}_{xx'}: \cF_{xx'} \to \Base(E)$ covered by appropriate bundle data. We say these are \emph{concordant} if there are homotopies between $\iota_{xx'}$ and $\tilde{\iota}_{xx'}$ which are compatible with breaking, and lifts of the homotopies to $E$. The corresponding lifts of $\cF$ to objects of $\Flow^E$ are then equivalent.  (One can similarly define concordance of all the other data discussed above.)
\end{ex}

  \begin{defn}
      The \emph{shift} $\cF[i]$ of a flow category has the same objects and morphism spaces, but degree function $|x|_{\cF[i]} = |x|_{\cF}+i$.  

If $\cF$ is $E$-oriented, then we define an $E$-orientation on $\cF[i]$ to be given by the same maps $(\cF_{xx'}, I^\cF_{xx'})(\nu_{xx'}) \to E(\nu_{xx'}+d_{xx'})$. 

  \end{defn}

  As in \cite{PS,PS2} there are truncations $\tau_{\leq i}\cF$ of an $E$-oriented flow category, in which one prescribes all manifolds of dimension $\leq i$ and asserts the existence of manifolds of dimension $i+1$, satisfying the correct breaking. Equipping all such with $E$-orientations yields a category  $\tau_{\leq i}\Flow^E$, where $\Flow^E = \tau_{\leq \infty}\Flow^E$.  There are  truncation functors $\tau_{\leq i}\Flow^E \to \tau_{\leq j}\Flow^E$ whenever $j \leq i \leq \infty$, which are furthermore associative.

\subsection{Left and right modules}

An unoriented right module $\cN$ of degree $i$ over an unoriented flow category $\cF$ comprises moduli spaces $\cN_x$ for $x\in \cF$, of dimension $i - |x|$, with embeddings of boundary faces $\cN_{x'} \times \cF_{x'x} \hookrightarrow \cN_x$. 

Stabilisation and index data now comprises positive integers $\nu_{*x}$, and finite sets $d_{*x}$ of cardinality $i-|x|$, satisfying the obvious compatibilities.  An $E$-orientation of $\cN$ involves giving maps of Thom spaces
\[
I^\cN_{\ast x}(\nu_{*x}) \to E[1](\nu_{*x} + d_{*x})
\]
compatible under breaking.  Bordism of left modules is defined similarly, and there is a group $\Omega_*^E(\cF)$ of bordism classes of left modules. 

\begin{rmk}
    An $E$-oriented right module of degree $i$ is exactly a morphism $\ast[i] \to \cF$ of $E$-oriented flow categories.  The bordism group $\Omega_i^E(\cF)$ is the morphism group $[\ast[i],\cF]$ in the category $\Flow^E$.

    A degree $i$ left module over $\cF$ is a morphism $\cF \to \ast[-i]$, and the bordism group $\Omega_E^i(\cF)$ of $E$-oriented such  modules is the morphism group $[\cF,\ast[-i]]$ in $\Flow^E$.
\end{rmk}

Recall that a bilinear map $\cB: \cF \times \cG \to \cH$ of flow categories comprises manifolds with faces $\cB_{xy,z}$ for $x\in\cF,y\in\cG,z\in\cH$, of dimension $|x|+|y|-|z|$, with codimension one boundary faces covered by embeddings
\[
\cF_{xx'}\times\cB_{x'y,z}, \ \cG_{yy'} \times \cB_{xy',z}, \ \mathrm{and} \ \cB_{xy,z'}\times\cH_{z'z}.
\]
In the presence of stabilisation and index data, an $E$-orientation of $\cR$ comprises maps 
\[
(\cB_{xy;z},I^\cB_{xy;z})(\nu_{xy;z}) \to E[1](\nu_{xy;z} + d_{xy;z})
\]
which are compatible with breaking and associative for compatible breaking (cf. Remark \ref{rmk:change order} in this case). 

If $\cB: \cF \times \cG \to \cH$ is an $E$-oriented bilinear map of flow categories there is an induced map
\begin{equation} \label{eqn:composition from bilinear map}
\cB_*: \Omega_i^E(\cF) \times \Omega_j^E(\cG) \to \Omega_{i+j}^E(\cH).
\end{equation}

\begin{rmk}\label{rmk:bilinear map bijection}
    For $x\in\cF, y\in \cG, z \in \cH$ the association $\cB'_{yx;z} := \cB_{xy;z}$ shows that bilinear maps $\cB: \cF \times \cG \to \cH$ naturally biject with bilinear maps $\cB': \cG \times \cF \to \cH$, compatibly with $E$-orientations.
\end{rmk}

\subsection{Bordism groups\label{sec:bordism coefficients for obstruction theory}}

Let $E$ be a commutative Thom $\cI$-monoid.  Then there is a Thom commutative ring spectrum $\Thom(E)$ (cf. Remark \ref{rmk: comm thom ring spec}) spectrum which has an associated bordism theory, $\Omega^{E}_*$. Concretely, an element of $\Omega_i^E$ is realised by a closed $i$-dimensional manifold $M$ and a map of Thom spaces $(M,TM)(k) \to E(n_{i,k})$, where $M \to \Base(E)$ lands in a component where the rank of the bundle $E(n_{i,k})$ is $i+k$. One can always further stabilise to increase $k$ and accordingly $n_{i,k}$. Two such elements $(M,TM)(k) \to E(n_{i,k})$ and $(N,TN)(k') \to E(n_{i,k'})$ are bordant if there is some $\kappa$ with $k \sqcup k' \hookrightarrow \kappa$ and a commutative diagram

\[
\xymatrix{
(M,TM)(\kappa) \ar[r]\ar[d] & E(n_{i,\kappa}) \ar[d] \\
(W,TW)(\kappa) \ar[r] & E[1](n_{i,\kappa}) \\
(N,TN)(\kappa) \ar[r] \ar[u] & E(n_{i,\kappa}) \ar[u]
}
\]
where $W$ is a bordism between $M$ and $N$ and the diagram represents the required $E$-orientation data. We write $\Omega_*^E$ for the resulting groups.

The definition adapts directly to the case of maps $M \to X$ to a target $X$ which need not be a point. More generally, assume $X$ is equipped with a vector bundle $V \to X$ of rank $r$. We may define $\Omega^E_i(X, V)$ to be the group of closed $i$-manifolds $M$ equipped with a map of Thom spaces $(M, TM)(k + r) \to (X, V)\otimes E(n_{i,k})$, considered up to stabilisation and bordism as above.

Let $\cF$ be a(n unoriented) flow category. A left module $\cN$ \emph{lives over $X$} if there are maps $\cN_{x*} \to X$ which are compatible with the module action of $\cF$.

Given a commutative Thom $\cI$-monoid $E$ and an $E$-orientation on $\cF$, a left module $\cN$ over $\cF$ living over $X$ is \emph{$E$-oriented relative to $V$} if it is equipped with maps of Thom spaces $(\cN_{*x}, I^\cN_{*x})(\nu_{x*}+r) \to (X, V) \otimes  E[1](\nu_{x*}+d_{x*})$, which are compatible with breaking. The underlying maps of spaces $\cN_{*x} \to X$ are required to agree with the given ones.

We may define bordisms of modules, as well as bilinear maps $\cF \times \cG \to *[i]$, over $X$, and $E$-orientations relative to $V$, similarly.

Our eventual theory of interest, $\Omega^{E_\Psi, \cO\cC}(X, \phi)$, is a variant on this which is only defined for commutative Thom $\cI$-monoids arising from tangential pairs, cf. Section \ref{sec: tang fuk}.

\subsection{Unitality\label{Sec:unitality}}

We discuss units, following \cite[Section 6]{AB2}. Their key idea is to introduce a novel stratification of $B\times [0,1]$ for a manifold with corners $B$.

\begin{rmk} \label{rmk:order corner strata}
   We now assume that the set of boundary strata of a manifold with corners admits a partial order which is a total order on the facets (i.e. codimension one faces)  adjacent to any corner. This will hold for all the manifolds with corners arising as morphism spaces in Floer (or Morse) flow categories, cf. Theorems \ref{thmdef:unor fuk} and \ref{thm: unor OC mod} for a comprehensive list. 
\end{rmk}

    For a $k$-dimensional manifold with corners $X$, with the further data from Remark \ref{rmk:order corner strata}, \cite[Lemma 6.3 \& Definition 6.4]{AB2} construct a $(k+1)$-dimensional manifold with corners $\bD(X)$, the `conic degeneration of $X$'. If in a local chart $X$ is modelled on $[0,\infty)^k$, then $\bD X$ maps to $[0,\infty)^k$ with fibres a union of $i+1$-intervals over the locus where $i$ co-ordinates vanish. 
    
    \begin{lem}
        For a flow category $\cF$, there is a morphism $\bD\cF: \cF \to \cF$ with spaces
    \[
    (\bD\cF)_{xy} = \begin{cases} \bD(\cF_{xy}) & x\neq y \\ \{pt\} &  x=y\end{cases}
    \]
    \end{lem}

    \begin{proof}
        The boundary facets of the conic degeneration are enumerated as
    \[
    \partial \bD\cF_{xy} = \bD\cF_{xx'} \times \cF_{x'y} \cup \cF_{xy'} \times \bD\cF_{y'y}.
    \]
    See \cite[Lemma 6.7]{AB2}.
    \end{proof}

    \begin{lem}\label{lem: flow unit unor}
        $\bD\cF$ defines the unit for $\cF$ in a unital structure on the category $\Flow$. 
    \end{lem}

\begin{proof}
    It follows from \cite[Lemma 6.10]{AB2} (and the fact that the diagonal bimodule gives a system of degeneracies for their semisimplicial category) that for any morphism $\cW: \cF \to \cG$, there is a bordism $\cW \simeq (\cW \circ \bD \cF)$. 
\end{proof}

Write $p: \cD_n \to [0,\infty)^n$ for the local conic degeneration of the basic manifold with corners.  \cite[Lemma 6.11]{AB2} construct compatible isomorphisms
\[
T\cD_{n+1} \oplus \underline{\bR} \to T[0,\infty)^n \oplus \underline{\bR} \oplus \bR
\]
with the properties  that (i) on the interior, the first factor of the isomorphism is that induced by the projection $p$ up to a positive diagonal rescaling, and (ii) the map $\underline{\bR} \to \underline{\bR} \oplus \bR$ sends the generator $1$ to a point in the positive quadrant. Such isomorphisms are compatible with the inductive construction of the conic degenerations, so for any manifold with corners $B$ one gets analogous isomorphisms in the setting of the map  $p: \bD B \to B$ assembled from the local conic degenerations over strata.

\begin{lem}\label{lem:unit morphism is oriented}
    If $\cF$ is an $E$-oriented flow category, then the unit morphism $\bD \cF$ is naturally an $E$-oriented morphism.
\end{lem}

    \begin{proof}
    Suppose first $x\neq y$ are distinct objects of $\cF$.  The maps $\cF_{xy} \to E$ induce corresponding maps on $(\bD\cF)_{xy}$. 
The $E$-orientation on $\cF$ gives isomorphisms
\[
I_{xy}^{\cF}(\nu_{xy}) \to E(\nu_{xy} + d_{xy})
\]
for stabilisation spaces $\bR^{\nu_{xy}}$ and index data $d_{xy}$, 
compatible with breaking. The previous comments show that there are also coherent isomorphisms
\[
T \bD\cF_{xy} \oplus \bR_y \to p^*T\cF_{xy} \oplus \bR_y \oplus \bR
\]
  Combining these gives rise to isomorphisms
  \[
  I_{xy}^{\bD\cF}(\nu_{xy}) \to E[1](\nu_{xy} + d_{xy})
  \]
where the map from $\bD\cF$ to $\Base(E)$ is induced from the corresponding classifying map for $\cF$ and the forgetful map $p$.  
  
     Finally, for any $x$ we map $(\bD\cF)_{xx}$ to the basepoint in the (base of the) monoid $E$.      
     \end{proof}

\begin{lem}
    The homotopy category $\Flow^{E}$ is unital.
\end{lem}

\begin{proof}
    Combine the arguments in the proofs of Lemmas \ref{lem: flow unit unor} and \ref{lem:unit morphism is oriented}.
\end{proof}

      In \cite{PS2} we constructed a category $\Flow^{\Psi}$ of $\Psi$-oriented flow categories, for any graded tangential structure
    \begin{equation} \label{eqn:tangential infinite loop}
    \xymatrix{
    \Theta \ar[r]\ar[d] & \Phi \ar[d] \\
    BO \ar[r] & BS_{\pm}U
    }
    \end{equation}
 
\begin{conj} \label{conj:compare models of Flow}
    Suppose $\Psi$ is commutative and \eqref{eqn:tangential infinite loop} is a diagram of infinite loop spaces, so there is a commutative monoid $E_\Psi$ in Thom $\cI$-spaces. Then there is a unital functor $\Flow^{\Psi} \to \Flow^{E_\Psi}$.
\end{conj}

\begin{rmk}
    This is not immediate from the constructions, essentially because concatenation in $\Flow^{\Psi}$ relies on the product on $\Omega U/O$ coming from concatenation of loops, whereas concatenation in $\Flow^{E_\Psi}$ uses the infinite loop space addition on $U/O$.  It seems likely that one could adapt the construction of `twisting data' and arguments from Section \ref{sec: proo ind comm} below to prove Conjecture \ref{conj:compare models of Flow}.

    In particular, the functor  of Conjecture \ref{conj:compare models of Flow} should not be an equivalence. Let $R$ be the commutative ring spectrum $\Thom(E_\Psi)$. When $\Phi$ is contractible, we expect $\Flow^\Psi$ to be equivalent to the category of $R$-$R$ bimodules \cite[Conjecture 5.10]{PS2}, and $\Flow^{E_\Psi}$ to be equivalent to the category of $R$-modules.
\end{rmk}

\subsection{Cones}

    \begin{defn}
 Let $\cW: \cA \rightarrow \cB$ be a morphism of unoriented flow categories. There is an unoriented flow category $\Cone(\cW)$ with objects $\cA[1] \sqcup \cB$, with morphisms given by
            $$\Cone(\cW)_{xy} := \begin{cases}
                \cA_{xy} & \,\mathrm{if}\,x,y\in\cA\\
                \cB_{xy} & \,\mathrm{if}\,x,y\in\cB\\
                \cW_{xy} & \,\mathrm{if}\,x\in\cA,\,y\in\cB\\
                \emptyset & \,\mathrm{if}\,x\in\cB,\,y\in\cA
            \end{cases}$$
            Then $\Cone(\cW)$ contains $\cA[1]$ and $\cB$ as full subcategories. 
    \end{defn}
    If $\cA,\cB$ and $\cW$ are $E$-oriented then the shift $\cA[1]$ and the mapping cone $\Cone(\cW)$ each inherit an $E$-orientation, defined using the given stabilisation and index data and the associated maps (schematically denoted) $(\cA,I^\cA) \to E$, $(\cB,I^{\cB}) \to E$ and $(\cW,I^{\cW}) \to E[1]$.

    \begin{prop}\label{prop: cone prop}
        Let $\cW, \cW': \cA \to \cB$ be $E$-oriented morphisms of $E$-oriented flow categories. 
        \begin{enumerate}
            \item If $\cW$ and $\cW'$ are bordant (via a $E$-oriented bordism $\cR$), then $\Cone(\cW) \cong \Cone(\cW') \in \Flow^{E}$.
            \item If $\cV': \cA' \to \cA$ and $\cV'': \cB \to \cB'$ are isomorphisms, then $\Cone(\cW) \cong \Cone(\cV'' \circ \cW \circ \cV')$.
        \end{enumerate}
    \end{prop}
    \begin{proof}
        For the first part, we first write down a morphism $\cQ$ from $
    \Cone(\cW)$ to $\Cone(\cW')$. Precisely, we use letters $a,b$ for elements of $\Cone(\cW) = \cA[1] \sqcup \cB$ and letters $\alpha,\beta$ for elements of $\Cone(\cW') = \cA[1] \sqcup \cB$. Then
    \begin{align*}
    \cQ_{a\alpha} & =  \bD\cA_{a\alpha}\\
    \cQ_{a\beta} & =  \cR_{a\beta} \\
    \cQ_{b\beta} & =  \bD \cB_{b\beta} \\
    \cQ_{b\alpha} & = \emptyset
    \end{align*}
    Note that the boundary strata $\cW_{a\beta}$ and $\cW'_{a\beta}$ in $\partial \cR_{a\beta}$ are then accounted for in the boundary of $\cQ$ by the factors arising from the point strata $(\bD\cA)_{aa}$ and $(\bD\cB)_{\beta\beta}$. The $E$-orientation on $\cQ$ is inherited from those on $\cA,\cB,\cR$ using that the diagonal bimodules are canonically then $E$-oriented.

    $\cQ$ also defines a morphism $\tilde{\cQ}$ from $\Cone(\cW')$ to $\Cone(\cW)$. The composite $\tilde{\cQ} \circ \cQ$ is an endomorphism of $\Cone(\cW)$. The underlying spaces of this are given by appropriate partial gluings 
    \[
    (\tilde{\cQ} \circ \cQ)_{aa'} = \cup_{\alpha} \, \bD\cA_{a\alpha} \times \bD\cA_{\alpha a} 
    \]
which is bordant to $\bD\cA_{aa'}$ by the idempotent property of the unit morphism for $\cA$. The morphism spaces for pairs $(b,b') \in \Ob\Cone(\cW)\times \Ob\Cone(\cW)$ are described similarly. Finally, for a pair $(a,b)$, one obtains a quotient of 
\[
    (\tilde{\cQ} \circ \cQ)_{ab} = \cup_{a'} \, \bD\cA_{aa'} \times \cR_{a'b} \cup \cup_{b'} \cR_{ab'} \times \bD\cB_{b'b} 
    \]
    Taking $a=a'$ and $b'=b$ (when the diagonal bimodule morphism spaces are just a point) shows that there are two boundary faces equal to $\cW_{ab}$ (whilst the two corresponding $\cW'_{ab}$ faces have been glued together). The whole union is then bordant to $\bD\cW_{ab}$, which shows that the composite map is bordant to the unit, and so $\cQ$ itself was an equivalence.  The $E$-orientations come along for the ride.  
    
The second part is similar.
    \end{proof}

  \subsection{Morse complexes}

Let $\det(V)$ denote the `usual' determinant line of a finite-dimensional vector space or vector bundle $V$. 

\begin{defn}\label{defn:Thom is oriented}
    We say that a monoid in Thom $\cI$-spaces $E$ is \emph{oriented} if we are given trivialisations of $\det(E(n))$ which are compatible with both the maps $E(n) \oplus \bR^{m-n} \to E(m)$ and also with the  maps $E(n) \times E(m) \to E(n+m)$ underlying the monoidal structure  $E\times E \to E$. 
\end{defn}

If $E$ is oriented, and $Z$ is an $E$-oriented zero-manifold, then we write $[Z] \in \bZ$ for the associated signed count.

    Let $\cF$ be a $E$-oriented flow category and $M$ a $\bZ$-module.  We define the \emph{Morse chain complex} $CM_*(\cF;M)$ to be the chain complex defined as follows. As a graded abelian group,

        \begin{equation}
            CM_*(\cF) := \bigoplus\limits_{x \in \cF} M \cdot x
        \end{equation}
        where for each $x \in \cF$, we treat $x$ as a formal generator in degree $|x|$.

        The differential is defined by:
        \begin{equation}
            \partial x := \sum\limits_{|x|-|y|=1} \sum\limits_{p \in \cF_{xy}} [p] \cdot y
        \end{equation}
        where $[p] \in \{\pm 1\}$ denotes the $E$-orientation of $p$; this is then  extended to be $\bZ$-linear. This defines a chain complex, whose homology is written $HM_*(\cF;M)$.

As in \cite{PS2}, we have that $HM_*(\cF;M)$ is functorial under $E$-oriented morphisms. An $E$-oriented bilinear map induces a degree zero chain-map, which descends to homology. If the relevant four bilinear maps admit an $E$-oriented associator, the product is associative on homology. 

Write $[\cF,\cG]^{E,\leq i}$ for morphisms in the category of truncated $E$-oriented flow categories $\tau_{\leq i}\Flow^E$.

\begin{lem} \label{lem:truncation to morse complexes}
If $E$ is oriented, $\pi_0\Thom(E) \cong \bZ$, and $M$ is a $\bZ$-module,
    there is an isomorphism 
    \[
    [\cF,\cG]^{E,\leq 0} \otimes_{\bZ} M \longrightarrow \Hom(CM_*(\cF;M), CM_*(\cG;M)).
    \]
    For any $E$ with $\pi_0\Thom(E) \cong \bZ/2$, and any $\bZ/2$-module $M$, 
    there is an isomorphism 
    \[
    [\cF,\cG]^{E,\leq 0} \otimes_{\bZ/2} M \longrightarrow \Hom(CM_*(\cF;M), CM_*(\cG;M)).
    \]
\end{lem}

\begin{proof}
    The proofs are essentially identical to one another and to those given in \cite[Lemma 5.41]{PS} and \cite[Lemma 5.25]{PS2}.
\end{proof}

\subsection{Obstruction theory }\label{sec: ob thry}

The following is a version of \cite[Theorem 5.48]{PS}, incorporating $E$-orientations, and crucially allowing the case in which neither domain nor target flow category is a point. The statement and proof make reference to the bordism theory introduced in Section \ref{sec:bordism coefficients for obstruction theory}.

    \begin{prop}\label{prop: lift ob}
        Let $\cA$ and $\cB$ be $E$-oriented $\tau_{\leq i+1}$-flow categories, and $\cW: \cA \to \cB$ an $E$-oriented $\tau_{\leq i}$-morphism. Then there is a homology class

        \begin{equation}\label{eq: lift ob}
            [\cO^E] \in \operatorname{Hom}\left(CM_{*+i+2}(\cA), CM_*(\cB; \Omega^{E}_{i+1})\right)
        \end{equation}
        such that if $[\cO^E]$ vanishes, $\cW$ is the truncation of an $E$-oriented $\tau_{\leq i+1}$-morphism $\cW'$.
    
    \end{prop}

    \begin{proof}
        This is an adaptation of the proof of \cite[Proposition 5.38]{PS2}. Let $\cW': \cA \to \cB$ be an extension of $\cW$ to a $E$-oriented pre-$\tau_{\leq i+1}$-morphism. The obstruction to $\cW$ being a $\tau_{\leq i+1}$-morphism is constructed as follows. For $x \in \cA, y \in \cB$ with $|x|-|y|=i+2$, we set
        \begin{equation}
            \cY_{xy} = \bigcup_{x' \in \cA} \cA_{xx'} \times \cW'_{x'y} \cup \bigcup_{y' \in \cB} \cW_{xy'} \times \cB_{y'y}
        \end{equation}
        where we glue along the common faces of the form $\cA_{xx'} \times \cW_{x'y'} \times \cB_{y'y}$. As written, this is not automatically a smooth manifold; however, it can be smoothed to a closed smooth $(i+1)$-manifold. We have index and stabilisation data for $\cA, \cB$ and $\cW'$, and their $E$-orientations are compatible with breaking and hence with the identifications of faces that enter into defining $\cY$. Equipping $\cY$ with the product index and stabilisation data therefore equips it with an $E$-orientation, compare to \cite[Section 5.7]{PS2} or \cite[Section 5.5]{PS}. The resulting manifold is well-defined up to $E$-oriented cobordism, and in particular gives a well-defined class in $\Omega^{E}_{i+1}$. These bordism classes together form the coefficients for a linear map of chain complexes
        \begin{equation}
            \cY: CM_{*+i+2}(\cA) \to CM_*(\cB; \Omega^{E}_*)
        \end{equation}
        If this vanishes, we may choose a filling of each $\cY_{xy}$, and use this to construct the required extension of $\cW'$ (cf. \cite[Lemma 5.30]{PS2}). 
        
        Given an exact map of chain complexes $F: CM_{*+i+2}(\cA) \to CM_*(\cB; \Omega^{E}_*)$, we use the coefficients to modify the extension $\cW'$; the resulting obstruction map $\cR'$ is exactly $\cR+F$. Therefore if $\cY$ is exact for some choice of extension $\cW'$, we may choose a different extension $\cW'$ so that $\cY$ vanishes (cf. \cite[Lemma 5.32]{PS2}).

        It remains to argue that $\cY$ is a chain map, and so represents a homology class $[\cO^E]$. This is equivalent to showing that the manifold
        \begin{equation}
            \bigsqcup\limits_{|x|-|x'|=1} \cA_{xx'} \times \cW_{x'y} \sqcup \bigsqcup\limits_{|y'|-|y|=1} \cW_{xy'} \times \cB_{y'y}
        \end{equation}
        represents the zero class in $\Omega^{E}_{i+1}$ for all $x \in \cA$, $y \in \cB$ with $|x|-|y|=i+3$. We construct a nullbordism as follows. For such $x,y$, define
        \begin{equation}
            \cX_{xy} = \bigcup\limits_{|x|-|x'|>1} \cA_{xx'} \times \cW'_{x'y} \cup \bigcup\limits_{|y'|-|y|>1} \cW_{xy'} \times \cB_{y'y}
        \end{equation}
        where we glue along common faces of the form $\cA_{xx'} \times \cA_{x'x''} \times \cW'_{x''y}$, $\cW_{xy''} \times \cB_{y''y'} \times \cB_{y'y}$ and $\cA_{xx'} \times \cW'_{x'y'} \times \cB_{y'y}$, for $|x|-|x'|>1$ and $|y'|-|y|>1$. We again smooth this and equip it with a classifying map for its abstract index  bundle $I^\cX_{xy} \to E$ as before, using the obvious inherited index and stabilisation data. This now provides the required nullbordism.
    \end{proof}
    Such extensions always exist if one is allowed to modify the target flow category:
    \begin{prop}\label{prop: truncation modification}
        Let $\cA, \cB$ be $E$-oriented $\tau_{\leq i+1}$-flow categories, and $\cW: \cA \to \cB$ a $E$-oriented $\tau_{\leq i}$-morphism.

        Then there is a $E$-oriented $\tau_{\leq i+1}$-flow category $\cC$, and a $E$-oriented $\tau_{\leq i+1}$-equivalence $\cV: \cA \to \cC$, and such that $\tau_{\leq i} \cC = \tau_{\leq i} \cB$ and $\tau_{\leq i} \cV = \cW$.
    \end{prop}
    By induction on $i$, this implies:
    \begin{cor}\label{cor:replace 0 truncation}
        Let $\cA, \cB$ be $E$-oriented flow categories, and $\cW: \cA \to \cB$ an $E$-oriented $\tau_{\leq 0}$-morphism.

        Then there is a $E$-oriented flow category $\cC$, and a $E$-oriented equivalence $\cV: \cA \to \cC$, such that $\tau_{\leq 0}\cC = \tau_{\leq 0}\cB$ and $\tau_{\leq 0}\cV = \cW$.
    \end{cor}

    \begin{proof}[Proof of Proposition \ref{prop: truncation modification}]
        $\cC$ is defined to have objects the same as $\cB$ (with the same gradings), and
        \begin{equation}
            \cC_{yy'} = \begin{cases}
                \cB_{yy'} 
                &
                \textrm{ if } |y|-|y'|-1 \leq i
                \\
                \cB_{yy'} \sqcup P_{yy'}
                &
                \textrm{ if } |y|-|y'|-1 = i+1
            \end{cases}
        \end{equation}
        where for each $y, y' \in \cC$ with $|y|-|y'|-1 = i+1$, $P_{yy'}$ is a choice of closed $(i+1)$-manifold with tangent bundle stably classified by a map $P_{yy'} \to \Base(E)$ (we will choose the $P_{yy'}$ in due course). $\cC$ is a $E$-oriented pre-$\tau_{\leq i+1}$ flow category, where the maps to $E$ are inherited from those given on $\cB$ and the $P_{yy'}$. We write this as $\cC=\cC(P)$ to indicate the choice involved; in particular, $\cB = \cC(\emptyset)$.

        By taking a matrix whose coefficients are given by the bordism class of each $P_{yy'}$, we obtain a linear map of chain complexes
        \begin{equation}
            P: CM_{*+i+2}(\cB) \to CM_*(\cB; \Omega_{i+1}^{E})
        \end{equation}

        \begin{claim}
            $\cC$ is an $E$-oriented $\tau_{\leq i+1}$-flow category if and only if $P$ is a map of chain complexes.
        \end{claim}
        \begin{proof}[Proof of claim.]

            Let $(dP)_{xy} \in \Omega^{E}_{i+1}$ be the coefficients of the matrix representing the linear map
            \begin{equation}
                dP: CM_{*+i+3}(\cB) \to CM_*\left(\cB; \Omega_{i+1}^{E}\right)
            \end{equation}
            By construction, a nullbordism of (representatives of) all $(dP)_{xy}$ is exactly the data of an extension of the pre-$\tau_{\leq i+1}$-flow category $\cC$.
        \end{proof}

        $\cW$ defines a $E$-oriented $\tau_{\leq i}$-morphism $\cA \to \cC$. From Proposition \ref{prop: lift ob}, the obstruction to this admitting an extension to a $E$-oriented $\tau_{\leq i+1}$-morphism $\cV: \cA \to \cC$ is a class
        \begin{equation}
            [\cO^E(P)] \in \operatorname{Hom}\left(CM_{*+i+2}(\cA), \cM_{*}(\cB; \Omega^{E}_{i+1}) \right)
        \end{equation}
        Inspecting the construction of the class $[\cO^E(P)]$, we find that
        \begin{equation}
            [\cO^E(P)] = [\cO^E(0) + P \circ \cW_*]
        \end{equation}
        where $\cW_*: CM_*(\cA) \to CM_*(\cB)$ is the map induced by $\cW$. Since $\cW_*$ is a quasi-isomorphism, we may therefore choose $P$ so that the homology class of $[\cO^E(P)]$ vanishes, and so the $\Psi$-oriented $\tau_{\leq i+1}$-extension $\cV$ of $\cW$ exists for this $\cC=\cC(P)$.
    \end{proof}

\begin{rmk}\label{rmk:truncate iso}
   Just as Proposition \ref{prop: lift ob} adapts \cite[Theorem 4.48]{PS} to the case of $E$-orientations, one can analogously establish a version of \cite[Theorem 5.51]{PS} in our setting. As a consequence, if $\cW: \cF \to \cG$ is an $E$-oriented morphism of $E$-oriented flow categories and composition with $\tau_{\leq0}\cW$ is surjective on groups of right $\tau_{\leq0}$-modules, then composition with $\cW$ is in fact surjective as a map $[\ast[i],\cF]^{\Flow^E} = \Omega^E_i(\cF) \to \Omega^E_i(\cG) = [\ast[i],\cG]^{\Flow^E}$, for every $i$. 
\end{rmk}

\subsection{Flow categories over a target\label{sec:tangerine}}

Let $L$ be a finite cell complex. We write $\cP L$ for the \emph{Moore path space} of $L$, comprising pairs $(r,\alpha)$ where $r\in \bR_{\geq 0}$ and $\alpha: [0,r] \to L$.  This admits an obvious concatenation operation. We will call $r$ the \emph{Moore length} of a given path.

\begin{defn}
A \emph{flow category over $L$} comprises a flow category $\cF$, along with a functor of topologically enriched categories $\cF \to \cP L$.  
\end{defn}
\begin{rmk}\label{rmk:spid}
    Though Morse/Floer moduli spaces most naturally map to path spaces of $L$, it is sometimes convenient to work with based loop spaces (such as in Section \ref{sec: mon loco sys}). Given a flow category $\cF$ over $L$ and a basepoint in $\ell_0$, we may construct another flow category $\cF$ over $L$, with all $x \in \cF$ sent to $\ell_0$ (using an idea from \cite{Barraud-Cornea}).
    
    To do this, we choose a \emph{spider}, meaning a contractible set of arcs $Y \subseteq L$ containing all the points each $x \in \cF$ are sent to, and connecting them to $\ell_0$. The quotient map $L \to L/Y$ is a homotopy equivalence, so we may choose a homotopy inverse $\theta: L/Y \to L$, inducing a map $\theta: \Omega_{[\ell_0]} L/Y \to \Omega_{\ell_0}$. Then we take the compositions: $\cF_{xy} \to \cP_{xy}L \to \Omega_{[\ell_0]}L \to \Omega_{\ell_0}L$.
\end{rmk}

  Explicitly, this consists of a point in $L$ for each $x \in \cF$ (by abuse of notation we denote this by $x \in L$) and maps of spaces $\cF_{xy} \to \cP_{xy}L$, compatible with concatenation.  We can define morphisms and bordisms of flow categories over $L$ in the obvious way. To define composition of morphisms, note that the coequaliser diagram in \cite[Equation (33)]{PS2} is compatible with evaluation to $\cP L$, so the composition of morphisms over $L$ inherits the structure of a morphism over $L$.

    \begin{defn}
        An \emph{$E$-oriented flow category over $L$} consists of an $E$-oriented flow category $\cF$ which lives over $L$.
    \end{defn}
 Explicitly, this means that there are maps  $\cF_{xy} \to \Base(E)\times \cP_{xy} L$ covered by maps $(\cF_{xy},I^{\cF}_{xy}) \to E\times \cP L$ where $(E\times \cP L)(n) := E(n) \times \cP L$ and the vector bundle on the target is pulled back from $E(n)$.

    There is a category $\Flow^{E}_{/L}$ of $E$-oriented flow categories over $L$. Morphisms in this category are again given by bordism-over-$L$ classes of $E$-oriented morphisms over $L$.

The construction of the unit morphism $\bD\cF$ in Section \ref{Sec:unitality} applies unchanged when $\cF$ lives over $L$, and makes $\Flow^{E}_{/L}$ a unital category. Thus, $\cF$ and $\cF'$ are equivalent in $\Flow^{E}_{/L}$ exactly when there are $E$-oriented morphisms $\cW: \cF \to \cF'$ and $\cW': \cF' \to \cF$ over $L$ for which the compositions $\cW' \circ \cW$ and $\cW \circ \cW'$ are bordant to the corresponding identity morphisms.

    \begin{defn}
        To define $*[i]$ as an $E$-oriented flow cat over $L$, we choose a base-point $\ell_0 \in L$, and specify that the point in $L$ corresponding to the unique object of $*[i]$ is $\ell_0$.
    \end{defn}

\section{Floer theory and the Fukaya category without local systems}\label{sec: fuk no loc}
    We quickly recap some relevant features of the spectral (Donaldson-)Fukaya category constructed in \cite{PS}, beginning in the unoriented setting. We then incorporate tangential structures, importantly in a different way to \cite{PS2} and reflecting the fact that we now assume that our tangential structures are commutative. Let $X$ be a graded Liouville domain of (real) dimension $2d$. 

    \begin{rmk}
        For the same tangential pair $\Psi$, the construction of the spectral Fukaya category we work with here differs from that of \cite{PS2}: the `$\Psi$-oriented flow categories' of \textit{loc. cit.} are expected to correspond to $R_\Psi$-$R_\Psi$ bimodules (cf. \cite[Conjecture 5.10]{PS2}), whereas those in this paper should correspond to $R_\Psi$-modules. The primary reason for the different set-ups is that in \cite{PS2} we did not impose any commutativity assumptions on $\Psi$.

        Nevertheless we expect that the underlying Fukaya categories should be isomorphic to each other.
    \end{rmk}
\subsection{Unoriented version}\label{sec:unoriented category}
Let $\Psi_{unor} = (BO \to BS_\pm U)$ be the `universal' choice of commutative tangential pair. In this section we define a Fukaya category corresponding to this commutative tangential pair, before moving onto the general case in Section \ref{sec: tang fuk}. We assume $X$ is \emph{graded}, meaning it is equipped with a lift of the classifying map of $TX \oplus \bC^{N-d}$ to $BS_\pm U(N)$, for some $N \gg 0$.

    For a closed exact Lagrangian $L \subseteq X$, a \emph{grading} on $L$ is a homotopy between the two ways around the following diagram:
    \begin{equation}
        \xymatrix{
            L
            \ar[r]
            \ar[d]_{TL \oplus \bR^{N-d}}
            &
            X
            \ar[d]
            \\
            BO(N)
            \ar[r]_{\cdot \otimes \bC} 
            &
            BS_\pm U(N)
        }
    \end{equation}
    This is equivalent to the notion of grading from \cite{Seidel:graded}. The obstruction to existence of a grading is the \emph{Maslov class} in $H^2(X,L)$.

    We fix a finite set $\cL$ of closed exact graded Lagrangians in $X$. We call $\cL$ a choice of \emph{Lagrangian data}. These will be the objects of the category $\cF(X;\Psi_{unor})$. 

    For each pair of objects $L, K \in \cL$, we choose regular Floer data $(H_t, J_t)$, which in particular satisfies $\phi_{H_t}^1(L) \pitchfork K$. We write $\cX(L,K)$ for the set of time-1 Hamiltonian chords from $L$ to $K$, with respect to $H_t$, and for $x,y \in \cX(L,K)$, we write $\cM^{LK}_{xy}$ for the space of (possibly broken) Floer trajectories from $x$ to $y$, defined using the Floer data $(H_t, J_t)$. The grading data defines an integer $|x| \in \bZ$ for each $x \in Ob(\cM^{LK})$ (our grading conventions follow \cite[Remark 1.4]{PS}, i.e. homological grading with unit in degree zero). 

    \begin{thmdef}\label{thmdef:unor fuk}
        \begin{enumerate}
            \item Each $\cM^{LK}$ defines a(n unoriented) flow category.

            The morphism group $\cF_i(L,K;\Psi_{unor}) := \Omega_i(\cM^{LK})$ is defined to be the group of right modules of degree $i$ over the Floer flow category $\cM^{LK}$.
            \item For each triple $L,K,P \in \cL$, there is a bilinear map $\cM^{LKP}: \cM^{KP} \times \cM^{LK} \to \cM^{LP}$.

            Composition in $\cF(X; \Psi_{unor})$ is defined to be the bilinear map of groups $\cM^{LKP}_*$ induced by $\cM^{LKP}$, as in Equation \ref{eqn:composition from bilinear map}.
            \item For each $L$, there is a distinguished element $\cE^L \in \cF(L,L; \Psi_{unor})$. 

            There is a bordism between the composition $\cE^L \circ \cE^L$ and $\cE^L$. Using this, it follows \cite[Section 7.5]{PS} that $\cE^L$ defines a unit for $L$ in the category $\cF(X; \Psi_{unor})$.
            \item For each quadruple $L,K,P,Q \in \cL$, there is an associator $\cM^{LKPQ}$, ensuring that composition in $\cF(X; \Psi_{unor})$ is associative.
        \end{enumerate}
        
    \end{thmdef}
    The technical material that underlies this theorem is due to \cite{Large,FO3:smoothness}, and the idea of constructing a category with morphism groups being bordism classes of flow modules is due to \cite{AB2}.
    
    Roughly, these are built from spaces of pseudoholomorphic curves as follows; see \cite[Section 7]{PS} for more detail (including a recap of the relevant parts of \cite{Large,FO3:smoothness}):
    \begin{enumerate}
        \item For $x,y \in \cX(L,K)$, each $\cM^{LK}_{xy}$ is a moduli space of pseudoholomorphic discs with two punctures, one input and one output, and boundary on $L$ and $K$.
        \item Each $\cM^{LKP}_{xy;z}$ is a moduli space of pseudoholomorphic discs with three punctures; two inputs are sent to $x$ and $y$ respectively, and the output to $z$. The boundary is required to lie on $L$, $K$ and $P$.
        \item Each $\cE^L_x$ is a moduli space of discs with one output puncture sent to $x$, and boundary on $L$.
        \item Each $\cM^{LKPQ}_{xyz;w}$ is a moduli space of discs with four punctures; three are inputs sent to $x,y,z$ and the output is sent to $w$. 
    \end{enumerate}

    \begin{rmk}\label{rmk:abusively drop morse}
        In abuse of notation, we will write $\cM^{LL}$ for the flow category associated to the pair $(L, \phi_f^1(L))$ where $f$ is a $C^2$-small Morse function on $L$. Then $\cX(L,L)$ bijects with the critical points of the Morse function $f$.
    \end{rmk}
        \begin{rmk}\label{rmk:inputs and outputs}
    Inputs and outputs are distinguished by whether one fixes incoming or outgoing strip-like ends, which in turn determines the asymptotics of the associated Floer data, see \cite{Seidel:book,PS,PS2}.
    \end{rmk}

\begin{figure}[ht]
\begin{center}
\begin{tikzpicture}

\draw[semithick] (-6,-0.5) -- (-2,-0.5);
\draw[semithick] (-6,0.5) -- (-2,0.5);
\draw[->] (-6.25,0) -- (-5.75,0);
\draw (-6.5,0) node {$x$};
\draw[->] (-2.25,0) -- (-1.75,0);
\draw (-1.5,0) node {$y$};
\draw (-4,-1) node {$L$};
\draw (-4,1) node {$K$};

\draw (2,-0.5) arc (-90:90:0.5);
\draw (1.5,-0.5) -- (2,-0.5);
\draw (1.5,0.5) -- (2,0.5);
\draw (1.9,0) node {$K$};

\draw (1.5,-1) -- (2.5,-1);
\draw (1.5,1) -- (2.5,1);
\draw (2.5,-1) arc (-90:0:0.25);
\draw (2.75, -0.75) arc (180:90:0.25);

\draw (2.5,1) arc (90:0:0.25);
\draw (2.75, 0.75) arc (-180:-90:0.25);
\draw (3,0.5) -- (4,0.5);
\draw (3,-0.5)--(4,-0.5);

\draw[->] (1.25,-0.75) -- (1.75,-0.75);
\draw[->] (1.25, 0.75) -- (1.75,0.75);
\draw[->] (3.75,0) -- (4.25,0);

\draw (0.95,-.75) node {$x$};
\draw (0.95,.75) node {$y$};
\draw (4.5,0) node {$z$};

\draw (2.75,-1.25) node {$L$};
\draw (2.75,1.25) node {$P$};

\end{tikzpicture}
\end{center}
\caption{Labelling conventions: an element of $\mathcal{M}^{LK}_{xy}$ and of $\mathcal{M}^{LKP}_{xy;z}$}
\end{figure}
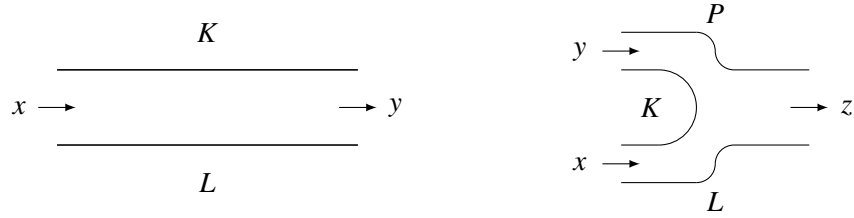

    \begin{rmk}
        Suppose we fix a spider on $L$. Then the category $\cM^{LK}$ naturally lives over $L$ in the sense of Section \ref{sec:tangerine}.
    \end{rmk}
    
    Along with the category $\cF(X; \Psi_{unor})$, we also have a version of the open-closed map:
    \begin{thm}[{\cite[Section 7.6]{PS}}]\label{thm: unor OC mod}
        \begin{enumerate}
            \item For each object $L \in \cL$, there is a distinguished left module $\Upsilon_L \in \Omega^d(\cM^{LL})$ living over $X$.
            \item There is a bordism over $X$, from the composition $\Upsilon_L \circ \cE^L$, which is a closed $d$-manifold lying over $X$, to $L$.
            \item For any pair $L, K \in \cL$, there is a bilinear map $\cW^L: \cM^{LK} \times \cM^{KL} \to \pt[d]$ lying over $X$.
            \item There are bordisms of bilinear maps between $\Upsilon^L \circ \cM^{LKL}$ and $\cW^L$, lying over $X$.
            \item For $L \neq K$, identifying $\cM^{LK} \times \cM^{KL}$ and $\cM^{KL} \times \cM^{LK}$ and recalling Remark \ref{rmk:bilinear map bijection}, the bilinear maps $\cW^L$ and $\cW^K$ are homotopic (and in particular, bordant) over $X$.
        \end{enumerate}
    \end{thm}
    \begin{rmk}\label{rmk: ups fact}
        Similarly to Remark \ref{rmk:abusively drop morse}, for appropriate choices of Floer data, the evaluation map from $\Upsilon_L$ to $X$ naturally factors through $L$, and so it defines a module over $L$. The same is true of the bordism in Theorem \ref{thm: unor OC mod}(2).
    \end{rmk}
 
    Any object $L \in \cL$ tautologically determines a class $\iota^L_*[L] \in \Omega^{unor}_d(X)$ in the unoriented bordism group of $X$.
    \begin{lem}\label{lem: OC same fund unor}
        Let $L, K \in \cL$ be objects in $\scrF(X; \Psi_{unor})$. Assume they are isomorphic in $\scrF(X; \Psi_{unor})$. 

        Then $\iota^L_*[L]=\iota^K_*[K] \in \Omega^{unor}_d(X)$, i.e. $L$ and $K$ are bordant over $X$.
    \end{lem}
    \begin{proof}
        Let $\cA \in \scrF_*(L, K; \Psi_{unor})$ and $\cB \in \scrF_*(K, L; \Psi_{unor})$ be a pair of inverse isomorphisms. This implies that $\cA \circ \cB$ is bordant to $\cE^K$, and $\cB \circ \cA$ to $\cE^L$.

        By Theorem \ref{thm: unor OC mod}(2), $L$ is bordant to $\Upsilon_L \circ \cE^L$ over $X$, and hence to $\Upsilon_L \circ (\cB \circ \cA)$. By Theorem \ref{thm: unor OC mod}(4), this is bordant to $\cW^L \circ (\cA, \cB)$ over $X$. By Theorem \ref{thm: unor OC mod}(5), this is bordant to $\cW^K \circ (\cB, \cA)$ over $X$; by the same argument as above, this is bordant to $K$ over $X$.
    \end{proof}
\subsection{Incorporating tangential pairs}\label{sec: tang fuk}
    Let $\Psi=(\Theta \to \Phi)$ be a commutative tangential pair; recall this is a pair of commutative $\cI$-spaces (satisfying a condition) sitting in a commutative diagram:
    \begin{equation}
        \xymatrix{
            \Theta 
            \ar[r]
            \ar[d]
            &
            \Phi 
            \ar[d]
            \\
            BO 
            \ar[r]
            &
            BS_\pm U
        }
    \end{equation}
    Let $E_\Psi$ be the commutative Thom $\cI$-monoid  defined in Section \ref{sec:abstract discs and the Thom I-space}. Recall its base is equivalent to $\hofib(\Omega \Theta \to \Omega \Phi)$.
    \begin{defn}\label{def: phi thet or}
        Let $X$ be a Liouville domain. A \emph{$\Phi$-orientation} $\phi$ on $X$ is a lift of the classifying map for $TX$, viewed as a stable complex vector bundle, to $\Phi$. Explicitly, this consists of a map $\phi: X \to \Phi(N)$ for some $N \gg 0$, such that the following diagram commutes:
        \begin{equation}
            \xymatrix{
                X
                \ar[r]^\phi 
                \ar[dr]_{TX \oplus \bC^{N-n}}
                &
                \Phi(N)
                \ar[d]
                \\
                &
                BU(N)
            }
        \end{equation}
    
        Let $L \subseteq X$ be a Lagrangian submanifold. A \emph{$\Theta$-orientation} $\theta$ on $L$ is a homotopy lift of the classifying map for $TL$ to $\Theta$, compatible with $\phi$. Explicitly, this consists of a map $\theta: L \to \Theta(N)$ lifting the classifying map for $TL$ (similarly to above for $X$ and $\phi$), and which is compatible with $\phi$, meaning the following diagram commutes up to a given homotopy: 
        \begin{equation}
            \xymatrix{
                L 
                \ar[r]
                \ar[d]_\theta
                &
                X
                \ar[d]_\phi
                \\
                \Theta(N)
                \ar[r]
                &
                \Phi(N)
            }
        \end{equation}  
    \end{defn}
    \begin{rmk}
        A $\Phi$-orientation on $X$ (and a compatible $\Theta$-orientation on $L$) equip $X$ (and $L$) with gradings in the sense of \cite{Seidel:graded}, as in Section \ref{sec:unoriented category}.
    \end{rmk}
  
    \begin{asmp}\label{asmp:tang}
        We assume $X$ is equipped with a $\Phi$-structure $\phi$, and that all objects $L \in \cL$ are equipped with compatible $\Theta$-structures. 
    \end{asmp}

    \begin{defn}\label{def: Phi x tild}
        Let $\phi$ be a $\Phi$-structure on $X$. Let $\widetilde X$ be the homotopy pullback of the maps $X, \Theta(N) \to \Phi(N)$; this maps to $\Theta(N)$ and hence to $BO(N)$. Let $T \to \widetilde X$ be the pullback of the tautological bundle over $BO(N)$. A \emph{$E_\Psi$-orientation relative to $\phi$} manifold/left module/bilinear map/bordism lying over $X$ is a $E_\Psi$-orientation relative to $T$ and lying over $\widetilde X$ (lifting the maps to $X$), in the sense of Section \ref{sec:bordism coefficients for obstruction theory}.
    \end{defn}
    
    \begin{thm}\label{thm:ind comm}
        Under Assumption \ref{asmp:tang}: 

        The flow categories, bilinear maps, units and associators considered in Theorem \ref{thmdef:unor fuk} all compatibly admit $E_\Psi$-orientations. Furthermore, these $E_\Psi$-orientations are canonical up to concordance in the sense of Example \ref{ex:concordance of Flow orientations}.

        Similarly, the left modules, bilinear maps and bordisms considered in Theorem \ref{thm: unor OC mod} are compatibly $E_\Psi$-oriented relative to the $\Phi$-orientation $\phi$ on $X$, and (in the case of (2)) the natural $E_\Psi$-orientation on $L$ relative to $\phi$.
    \end{thm}
    We defer the proof of Theorem \ref{thm:ind comm} to Section \ref{sec: proo ind comm}.

    \begin{defn}\label{def: E fuk}
        We define the category $\cF(X; \Psi)$ to have:
        \begin{itemize}
            \item Objects are elements $L \in \cL$.

            Recall this consists of some chosen collection of closed exact Lagrangians $L \subseteq X$, equipped with $\Theta$-orientations, compatible with $\phi$, the chosen $\Phi$-orientation on $X$.
            \item Morphisms from $L$ to $K$ are the right bordism groups $\cF_*(L,K;\Psi) := \Omega_*^E(\cM^{LK})$.
            \item Composition is induced by the $E$-oriented bilinear maps $\cM^{LKP}$.
            \item The units are the $E$-oriented right modules $\cE^L \in \cF_0(L,L;\Psi)$.
        \end{itemize}
    \end{defn}
   
    \begin{defn}\label{def: trunc fuk}
        For $i \geq 0$, we define the $\tau_{\leq i} \scrF(X; \Psi)$ to be the category with objects the same as $\scrF(X; \Psi)$, and morphisms $L$ to $K$ given by $\Omega^{E_\Psi, \leq i}_*(\tau_{\leq i}\cM^{LK})$, the $i$-truncated (bordism classes of) left module over the truncated Floer flow category $\tau_{\leq i}\cM^{LK}$.

        Composition is defined using the truncation of the usual Floer bilinear maps.
    \end{defn}
\subsection{Open-closed map with tangential pairs\label{sec:OC definitions}}

    Let $\Psi = (\Theta \to \Phi)$ be a commutative tangential pair, $F = \hofib(\Theta \to \Phi)$ the homotopy fibre and $E_\Psi$ the commutative Thom $\cI$-monoid constructed in Section \ref{sec:abst disc fix case}. Let $X$ be a Liouville domain equipped with a $\Phi$-orientation $\phi: X \to \Phi(N)$.

    Let $T \to \widetilde X$ be as in Definition \ref{def: Phi x tild}. We write $\Omega^{E_\Psi,\cO\cC}_i(X, \phi)$ for the groups $\Omega^{E_\Psi}_*(\widetilde X, T)$. Note that for $N \gg 0$ this is independent of $N$: more precisely, for fixed $i$ and sufficiently large $N \gg 0$, letting $\phi'$ be the composition of $\phi$ with $\Phi(N) \to \Phi(N+1)$, $\Omega^{E_\Psi,\cO\cC}_i(X, \phi) \cong \Omega^{E_\Psi,\cO\cC}_i(X, \phi')$. 
    
    For a $\Psi$-oriented Lagrangian $L \subseteq X$, we define $\Omega^{E_\Psi,\cO\cC}_*(L,\phi|_L)$ similarly using the restriction of $\phi$ to $L$.

    \begin{rmk}\label{rmk:explicit OC bordism class data}
        The groups $\Omega^{E_\Psi; \cO\cC}_i(X, \phi)$ are naturally isomorphic to the set of cobordism classes of tuples $(M, f, k, \rho, h)$ where:
        \begin{itemize}
            \item $f: M \to X$ is a closed $i$-manifold over $X$.
            \item $k \in \cI$.
            \item $\rho: (M, TM)(k+N) \to (\Theta(N), \Gamma) \times E(i+k)$ is a map of Thom spaces, where $\Gamma \to \Theta(N)$ is the pullback of the universal bundle over $BO(N)$.
            \item $h$ is a homotopy between the two maps $M \to \Phi(N)$ which factor through $f$ and $\Base(\rho)$ respectively.
        \end{itemize}
    \end{rmk}

The definition of the groups $\Omega_*^{E_{\Psi};\cO\cC}(\bullet)$ makes sense for any pair $(Y,\phi)$ where $Y$ is a cell complex with a complex vector bundle $\phi: Y \to \Phi \to BU$ equipped with a  $\Phi$-lift $\phi$. Recall that 
the ring spectrum $R_{\Psi}$ has an associated generalised cohomology theory $R^*(\cdot)$.
    \begin{lem} \label{lem:module structure for OC bordism}
        The groups $\Omega_*^{E_{\Psi};\cO\cC}(Y,\phi)$ are modules over $R_{\Psi}^{-*}(Y)$.
    \end{lem}
         \begin{proof}[Proof sketch]
        Let $f: M \to Y$ be as above. Let $\alpha \in R^{-i}(Y)$; represent this my a map $Y \to \Omega^{\infty+i}R$. This lifts to a map $Y \to \Omega^{j+i}\Thom(E(j))$ for some $j \gg 0$.

        Considering the composition $Y \to \Omega^{j+i}\Thom(E(j))$ and taking adjoints, we obtain a map $M_+ \wedge S^{j+i} \to \Thom(E(j))$. Generically perturbing and taking the preimage of the zero-section defines a new $(j+i)$-dimensional manifold $M'$; it has a natural map to $Y$ along with the relevant tangential structures to determine a class in $\Omega^{E_\Psi; \cO\cC}_*(Y,\phi)$. We leave to the reader that this gives a well-defined and  linear action.
    \end{proof}

    \begin{lem}\label{lem: def oc or}
        Let $L \in \cL$ be a $\Theta$-oriented Lagrangian in $X$. Then $L$ admits a natural $E_\Psi$-orientation relative to $\phi|_L$, and hence defines a class $[L] \in \Omega^{E_\Psi,\cO\cC}_d(L; \phi|_L)$.
    \end{lem}
    Pushing forwards along the inclusion defines $\iota^L: L \to X$ then defines a class $\iota^L_*[L] \in \Omega^{E_\Psi; \cO\cC}_d(X; \phi)$.
    \begin{proof}
        This is essentially by the second part of Definition \ref{def: phi thet or}: the $E_\Psi$-orientation can be taken to be trivial. 
    \end{proof}

    \begin{lem}\label{lem: tw cap}
        Let $g: M \to L$ be a map from a closed manifold, which is $E_\Psi$-oriented relative to $\phi|_L$. Let $f: M \to \Omega^n\Thom(\bV^\Psi(n) \to \bU^\Psi(n)) = \Omega^\infty R_\Psi(n)$ for some $n$, and $\hat f: M \times S^n \to \Thom(\bV^\Psi(n) \to \bU^\Psi(n))$ a generic perturbation of its adjoint. Let $\hat M$ be the preimage of 0; note this is also naturally $E_\Psi$-oriented relative to $\phi|_L$. 
        
        Let $[f] \in R_\Psi^0(M)$ be the class represented by the composition $M \to \Omega^\infty R_\Psi(n) \to \Omega^\infty R_\Psi$. Then:
        \begin{equation}
            g_*[\hat M] = g_*\left([M] \cap [f]\right) \in \Omega^{E_\Psi,\cO\cC}_d(L, \phi|_L)
        \end{equation}
    \end{lem}

    \begin{proof}
        $[f] \in R_\Psi^0(M)$ is Atiyah dual to a class in $\Omega^{E_\Psi}_d(M; -TM)$; chasing through the Pontrjagin-Thom construction, this class is represented by $\hat M$. In terms of the cap product, the Atiyah dual class is $[M] \cap [f]$. Pushing forwards along $g_*$ gives the desired formula.
    \end{proof}

    \begin{prop}\label{prop: quas same fund no loc}
        Assume Assumption \ref{asmp:tang}. Let $L, K \in \cL$ be objects in $\scrF(X; \Psi)$.

        If $L$ and $K$ are isomorphic, then $\iota^L_*[L] = \iota^K_*[K] \in \Omega^{E_\Psi; \cO\cC}_d(X; \phi)$.

        In words: $L$ and $K$ are $E_\Psi$-oriented bordant over $X$, relative to $\phi$.
    \end{prop}
    \begin{proof}
        Follows from the same argument as Lemma \ref{lem: OC same fund unor}, now using the given $E_\Psi$-orientations from Theorem \ref{thm:ind comm}.
    \end{proof}
    \begin{rmk}[Relation to constructions in \cite{PS2}]
    
        In \cite{PS2}, associated to any graded Liouville domain $X$ with $\Phi$-structure and any set of Lagrangians $\cL$ with compatible $\Theta$-structures, we also constructed a Fukaya category with objects $\cL$. Here our construction differs: our notion of $E$-oriented flow category considered here is stronger than the definition considered in \cite[Section 5]{PS2}. 
        In particular, we are able to bypass the use of ``capsule bordism'' as the target of the open-closed map considered in \cite[Section 5.10]{PS}-- morally this is an analogue of the splitting of ring spectra $THH(R) \to R$ for a commutative Thom ring spectrum $R$ \cite{Blumberg-THH}, which fails to hold in general if $R$ is not commutative.
        
        The key difference is that in \cite{PS2}, we worked in a more general setting, where we did not assume that $\Theta$ and $\Phi$ were commutative. In this paper, the results of Section \ref{sec: OC} require strong commutativity assumptions.
    \end{rmk}

\begin{ex}\label{ex:OC bordism maps here}
    In specific cases, the bordism theory $\Omega_*^{E_{\Psi};\cO\cC}$ maps to a more familiar theory, even if it does not obviously co-incide with anything classical:
    \begin{enumerate}
    \item If $\Phi = \{pt\}$ (so $X$ is stably framed) then there is a map $\Omega_*^{E_{\Psi};\cO\cC}(X,\phi) \to \Omega_*^{E_{\Psi}}(X)$. Taking $\Theta$ the pullback of $\widetilde{U/O} \to U/O$ under $U \to U/O$, there is a map $\Omega^{E_{\Psi}}_*(X) \to MU_*(X)$.
     \item If $\Psi = (pt,BS_{\pm}U)$ (so the Lagrangians are stably framed) then again there is a map $\Omega^{E_{\Psi};\cO\cC}(X,\phi) \to MU_*(X)$.
        \item If $(\Theta \to \Phi) = (BO \times F \stackrel{\pi_1}{\longrightarrow} BO)$, then $X$ is polarised, i.e. $TX \sim V\otimes \bC$ is stably a complexification. There is a map  $\Omega^{E_{\Psi};\cO\cC}(X,\phi) \to \Omega^{E_{\Psi}}_*(X,V) = \Omega^{E_{\Psi}}_*(\Thom(-V))$ to the bordism theory of manifolds $f: M \to X$ with an $E_{\Psi}$-orientation of $TM-f^*V$. In the special case $F = B(G/O)$ where $G = GL_1(\bS)$, then $\Omega^{E_{\Psi}}_* = \Omega^{\bS}_*$ is bordism of manifolds with a trivialisation of the stable tangent spherical fibration.
       
        \item The Unoriented case $\Psi = BO \to BU$. In this case the data of the homotopy $h$ in Remark \ref{rmk:explicit OC bordism class data} means that $\Omega^{E_{\Psi};\cO\cC}_*(X,\phi)$ is the bordism of pairs of a manifold $f: M \to X$ equipped with a stable polarisation of $f^*TX$. This admits a forgetful map to $MO_*(X)$.
    \end{enumerate}
\end{ex}

\subsection{Proof of Theorem \ref{thm:ind comm}}\label{sec: proo ind comm}
    \begin{proof}
        
        We prove Theorem \ref{thm:ind comm} in stages.

        \subsubsection*{Single moduli space, 1 output}

        We begin by constructing a $E_\Psi$-orientation on a single Floer moduli space $\cM$ with 1 output (so one of the moduli spaces considered in Theorem \ref{thmdef:unor fuk}).
        Let $i$ be the number of inputs of $\cM$. We assume we have chosen index data $d_\cM \in \cI$, as in Definition \ref{def: ind data}. We wish to construct a map of Thom spaces 
        \begin{equation} \label{eq: goal ind comm}
            (\cM, I^\cM)(\nu_\cM) \to (\bU^\Psi[l](\nu_\cM + d_\cM), \bV)
        \end{equation}
        for some $\nu_\cM \in \cI$. Here $l=0$ for Floer moduli spaces $\cM$ arising in Theorem \ref{thmdef:unor fuk}(1), $l=1$ if in Theorem \ref{thmdef:unor fuk}(2) or (3), and $l=2$ if in Theorem \ref{thmdef:unor fuk}(4).

    \subsubsection*{Set-up from \cite{PS2}}

        There is a choice of framed puncture data $\bE_x \in \cX(\cL)$ with Maslov index equal to the grading of $|x|$, cf. \cite[Section 6]{PS2}. We write $\bE_\cM$ for the tuple of puncture data given by those of the inputs and outputs of $\cM$.
        
        The main result of \cite[Section 6]{PS2} is that for $N, \nu_\cM \gg 0$ sufficiently large, there is a map of Thom spaces:
        \begin{equation}\label{eq: PS2 reca}
            \rho^{orig}_\cM: (\cM, I^\cM)(\nu_\cM) \to (\bU^\Psi_{\bE_\cM}[l](\nu_\cM, N), \bV)
        \end{equation}
    \subsubsection*{Outline of strategy}
        We begin by separating the given map $\rho^{orig}_\cM$ into two pieces: roughly speaking, this is a lift $\rho'_\cM$ of $\rho^{orig}_\cM$ along a map of the form $\oplus: \bU^{fr}_{\ldots}(\ldots) \times_{\cR\cS_{i1}} \bU^{\Psi}_{\ldots}(\ldots) \to \bU^\Psi_{\bE_\cM}(\ldots)$.

        Next, we trivialise the index bundle over the image in the first factor $\rho'_\cM$ (which we may do since it is a space of framed discs). In the second factor, we choose a ``capping map'' $\cM \to \bU^{fr}_{01}(0, \ldots)$ (note the index bundle over the target is 0) and glue this to the inputs of the second factor $\rho''_\cM$. Together, this defines a map of Thom spaces $(\cM, I^\cM) \to (\bU^\Psi_{01}(\ldots), \bV) \oplus \bR^{d_\cM}$. Lifting along the (highly-connected) inclusion $\bU^\Psi(\ldots) \to \bU^\Psi_{01}(\ldots) \oplus \bR^{d_\cM}$ (cf. Remark \ref{rmk: fix dom case}) provides the desired $\bE_\Psi$-orientation on $\cM$. 
    \subsubsection*{Separating into two pieces}
        
        We choose $N', \nu'_\cM, \nu''_\cM \gg 0$ large, and injections 
        \begin{equation}\label{eq: eruhgeroghedsough}
            \xymatrix{
                [N'] \sqcup [\nu''_\cM]
                \ar[r]
                &
                [N]
                &
                \textrm{ and }
                &
                [\nu'_\cM] \sqcup[\nu''_\cM]
                \ar[r]
                &
                [\nu_\cM]
            }
        \end{equation}
        We may assume that $\bE_\cM = \bE^{std}_{i1} \oplus \bE'_\cM$, for some other framed puncture data $\bE'_\cM$.

        Consider the following homotopy lifting problem:
        \begin{equation}\label{eq: htpy lift i}
            \xymatrix{
                &
                \bU^{fr}_{\bE'_\cM}(\nu'_\cM, N') \times_{\cR\cS_{i1}} \bU^\Psi_{i1}(\nu''_\cM, \nu''_\cM)
                \ar[d]_\oplus
                \\
                \cM
                \ar@{-->}[ur]
                \ar[r]_\rho
                &
                \bU^\Psi_{\bE_\cM}(\nu_\cM, N)
            }
        \end{equation}  
        The right vertical map is induced by the maps (\ref{eq: eruhgeroghedsough}), and is highly connected (more precisely: for sufficiently large choices of $N',\nu'_\cM,\nu''_\cM$, its connectivity is larger than the dimension of $\cM$), so we may choose such a homotopy lift. We write $\rho'_\cM:\cM \to \bE^{fr}_{\bE'_\cM}(\nu'_\cM, N')$ and $\rho''_\cM: \cM \to \bU^\Psi_{i1}(\nu''_\cM, \nu''_\cM)$ for the two factors. The homotopy lift $(\rho'_\cM, \rho''_\cM)$ extends naturally to one of Thom spaces; explicitly, after including the fibre product into the product, we obtain a map of Thom spaces:
        \begin{equation}
            (\cM, I^\cM)(\nu_\cM) \to \left(\bU^{fr}_{\bE'_\cM}(\nu'_\cM, N') \times \bU^\Psi_{i1}(\nu''_\cM, \nu''_\cM), \bV \oplus \bV \oplus \bR^{\nu_\cM-\nu'_\cM-\nu''_\cM} \oplus \bR^{N-N'-\nu''_\cM} \oplus \bR^l\right) 
        \end{equation}
        Note that we do not treat the summands $(\bU^{fr}_{\bE'_\cM}(\ldots), \bV)$ and $(\bU^\Psi_{i1}(\ldots), \bV)$ on the RHS separately: in particular, we do not separate $I^\cM$ into two summands, and then lift them separately. We adopt the same convention in future lifting arguments without comment. 
    \subsubsection*{The pieces}
        $\bU^{fr}_{\bE'_\cM}(\nu'_\cM, N')$ is highly connected, so we may trivialise the index bundle $\bV$ over the image of $\cM$; note its rank is $\nu'_\cM + d_\cM$.

        $\rho''_\cM$ lands in a space of abstract discs with $i$ inputs; our goal (\ref{eq: goal ind comm}) is to land in a space with none. To deal with this, we glue ``caps'' to these inputs. Explicitly, choose a map $\gamma_\cM: \cM \to (\bU^{fr}_{01}(0, \nu''_\cM))^i$. 
        The target is not in general contractible (even as $\nu''_\cM \to \infty$), but by Lemma \ref{lem: cr sect} it contains a contractible subspace; enforcing that $\gamma_\cM$ lands in this subspace thus specifies $\gamma_\cM$ up to contractible choice. Also note the index bundle of the codomain of $\gamma_\cM$ is 0. Now compose $\rho''_\cM$ with the map which glues these discs on, and we obtain a map $\cM \xrightarrow{\gamma_\cM \# \rho''_\cM} \bU^\Psi_{01}(\nu''_\cM, \nu''_\cM)$; since the index bundle over the space of caps used here is 0, this naturally extends to a map of Thom spaces.

        The inclusion $\bU^\Psi(\nu''_\cM) \to \bU^\Psi_{01}(\nu''_\cM, \nu''_\cM)$ is highly-connected, so we may lift $\gamma_\cM \# \rho''_\cM$ along it. Combining with the trivialisation of the framed part above, we obtain a map of Thom spaces $(\cM, I^\cM)(\nu''_\cM) \to \bU^\Psi[l](\nu''_\cM) \oplus \bR^{d_\cM}$, which we then may compose with the map $E_\Psi(\nu''_\cM)\oplus \bR^{d_\cM} \to E_\Psi(\nu''_\cM + d_\cM)$ to obtain (\ref{eq: goal ind comm}).

        We next explain how to modify the above argument so that the resulting $E_\Psi$ orientations are compatible with Floer gluing. 
    \subsubsection*{Coherent combinatorics}
    
        We choose $\nu'_\cM$ and $\nu''_\cM$ each as part of some choice of stabilisation data in the sense of Section \ref{sec: or flow cat I}. We may additionally choose the maps (\ref{eq: eruhgeroghedsough}) to be compatible with the injections which are part of the stabilisation data. 
        
        We may assume the choices of puncture data are compatible with this data, namely that for any inclusion of a boundary face $\prod_l \cA_l \to \cM$ (here $\cA_l$, $\cM$ are moduli spaces appearing in the statement of the theorem), $\oplus_l \bE'_{\cA_l}$ stabilises to $\bE'_\cM$. See Section \ref{def: Floer mod spac} for further discussion. Recall that each such moduli space is a manifold with faces, with a system of faces given by products of other such moduli spaces of lower dimensions.

    \subsubsection*{Coherent lifts}
        It is proved in \cite[Section 6]{PS2} that the maps $\rho^{orig}_\cM$ are compatible with inclusions of boundary faces on the LHS of (\ref{eq: PS2 reca}) and with gluing of abstract discs on the RHS of (\ref{eq: PS2 reca}).

        The right hand column of (\ref{eq: htpy lift i}) is also compatible with inclusions of boundary faces (as a diagram of Thom spaces). Note that the gluing operation on the top right entry of (\ref{eq: htpy lift i}) comes from concatenation in the two factors separately, along with appropriate stabilisation. By choosing the homotopy lift inductively over $\dim(\cM)$, we may therefore assume the homotopy lifts $(\rho'_\cM, \rho''_\cM)$ are also compatible with inclusions of boundary faces.

        Since all $\bU^{fr}_{\bE'_\cM}(\nu''_\cM, \nu''_\cM)$ are highly connected, by choosing them inductively over $\dim(\cM)$, we may trivialise their index bundles over the images of each $\cM$ coherently with respect to inclusions of boundary faces.

        Now let $\cM \times \cN \to \cP$ be an inclusion of a boundary faces. $\gamma_\cP \# \rho''_\cP$ is not directly related to $\gamma_\cM \# \rho''_\cM$ and $\gamma_\cN \# \rho''_\cN$; 
        
        Note that $\gamma_\cM \# \rho''_\cM$ and $\gamma_\cN \# \rho''_\cN$ can't be $\oplus$ed together over $\cM \times \cN$, since they don't land in a space of abstract discs with the same domain. The inclusion of the subspace with the same domains $\bU^\Psi_{01}(\ldots) \times_{\cR\cS_{01}} \bU^\Psi_{01}(\ldots ) \to \bU^\Psi_{01}(\ldots) \times \bU^\Psi_{01}(\ldots)$ is highly connected, and so a homotopy lift of $(\gamma_\cM \# \rho''_\cM, \gamma_\cN \# \rho''_\cN)$ exists, and in fact $\gamma_\cP \# \rho''_\cP$ is obtained from $(\gamma_\cM \# \rho''_\cM, \gamma_\cN \# \rho''_\cN)$ by choosing some homotopy lift (of maps of Thom spaces), and then direct summing. The space of choices in how one does this is highly connected (note that the contractible space of choices of the capping maps is incorporated into this). A more restrictive (but still highly-connected in terms of the space of choices) way to do this is to choose homotopy lifts of the $E_\Psi$-orientations of $\cM$ and $\cN$ along the inclusion $\bU^\Psi(\ldots) \to \bU^\Psi_{01}(\ldots)$. Hence the stabilisation of $E_\Psi$-orientation on $\cM \times \cN$ is homotopic to the restriction of that of $\cP$.
        
        Using this, over a collar neighbourhood of the face $\cM \times \cN$ in $\cP$, we may homotope the $E_\psi$-orientation on $\cP$ so that these agree on the nose.

        The same argument applies to more general inclusions of boundary faces $\prod_l \cA_l \to \cP$. By first inducting over $\dim(\cP)$ and then inducting over the dimension of the faces of $\cP$, we may homotope the $E_\Psi$-orientation on $\cP$ so that it agrees with the direct sums of the $E_\Psi$-orientations on its faces induced by the direct sums of those of the lower-dimensional moduli spaces; this is exactly the required coherence.
    \subsubsection*{0 output case: set-up}
        Now assume $\cM$ is a moduli space of pseudohomolorphic curves with no outputs, i.e. one arising in the statement of Theorem \ref{thm: unor OC mod}. Then instead of (\ref{eq: PS2 reca}), \cite[Section 6]{PS2} produces a map of Thom spaces:
        \begin{equation}\label{eq: PS2 reca 0 out}
            \rho^{orig}_\cM: (\cM, I^\cM)(\nu_\cM+N) \to (\bU^\Psi_{\bE_\cM}[l](\nu_\cM, N), \bV)
        \end{equation}
        Furthermore, there are maps of spaces $\eps: \cM \to X$, such that the following diagram commutes:
        \begin{equation}
            \xymatrix{
                \cM
                \ar[r]_\eps
                \ar[d]
                &
                X
                \ar[d]_\phi
                \\
                \bU^\Psi_{\bE_\cM}[l](\nu_\cM,N)
                \ar[r]
                &
                \Phi(N)
            }
        \end{equation}  
        Both horizontal maps are given by evaluating at the interior marked point.
        
        $\bE_\cM$ is defined similarly to before, and $l=1$ if we are considering a moduli space from Theorem \ref{thm: unor OC mod}(1) or (3), and $l=2$ otherwise. Recall from Definition \ref{def: RSij} that when $\cM$ has no outputs (as is the case currently), elements of $\bU_{\bE_\cM}^\Psi(\ldots)$ are equipped with a marked point $p$ in the interior of the domain Riemann surface, along with a tangent direction at $p$, pointing towards some point on (the smoothing of) the boundary of the domain Riemann surface.
    \subsubsection*{0 outputs: lifting}

        We make choices (\ref{eq: eruhgeroghedsough}) as before, and ``separate into two pieces'' by choosing a homotopy lift analogously to (\ref{eq: htpy lift i}), but with $\bU^\Psi_{i0}(\ldots)$ in the top right of the diagram instead of $\bU^\Psi_{i1}(\ldots)$. We now have a map $(\rho', \rho''): \cM \to \bU^{fr}_{\bE'_\cM}(\nu'_\cM, N') \times_{\cR\cS_{i0}} \bU^\Psi_{i0}(\nu''_\cM, \nu''_\cM)$, along with an extension to a map of Thom spaces.

        We may trivialise the index bundle over the image of $\rho'$ similarly to before; note this has rank $d_\cM+\nu'_\cM+N'$. We may choose a capping map $\gamma_\cM$ as before to obtain $\gamma_\cM \# \rho''_\cM: \cM \to \bU^\Psi_{00}(\nu''_\cM, \nu''_\cM)$ (and an extension to a Thom map). At this stage, we must deviate slightly from the earlier construction.

        We define two new types of space of abstract discs. These are variants of Definition \ref{def: big abst disc defn}:

        \begin{itemize}
            \item $\bU^{\Psi,fr}_{0*}(v, n) \subseteq \bU^\Psi_{00}(v, n)$ is defined to be the space of abstract discs with 0 inputs and 0 outputs, such that if $p$ is the point in the domain Riemann surface which the asymptotic marker points to, the boundary conditions are trivialised over $p$ (i.e. the homotopy lift to $\Theta$ over $p$ is lifted further to $fr$).    
            \item $\bU^{const,\Theta}_{00}(v, n) \subseteq \bU^{(\Theta,\Theta)}_{00}(v, n)$ is the subspace where the map from the domain Riemann surface to $\Theta$ is constant.

            Here $(\Theta, \Theta)$ is the commutative tangential pair where both spaces are given by $\Theta$.
        \end{itemize}
        There are maps $\bU^{\Psi, fr}_{0*}(v,n) \to \Omega F(n)$ and $\bU^{const,\Theta}_{00}(v,n) \to \Theta(n)$ whose connectivity go to infinity as $v \to \infty$. Both of these spaces naturally map also to $\bU^\Psi_{00}(v, n)$, and are equipped with the pullback of the index bundle from it. Note that by the constancy condition, the index bundle over $\bU^{const,\Theta}_{00}(v, n)$ is isomorphic to the bundle pulled back along the composition to $\Theta(n) \to BO(n)$.
        
        Recall that $\bU^\Psi_{00}(\ldots)$ approximates the homotopy limit of the diagram given by the two maps $\cL \Theta_{h\cI}, \Phi_{h\cI} \to \cL \Phi_{h\cI}$; since these are infinite loop spaces, this is equivalent to $\Omega F_{h\cI} \times \Theta_{h\cI}$, the product of the subspaces defined above.

        We choose $\mu'_\cM, \mu''_\cM \gg 0 \in \cI$ and an injection $[\mu'_\cM] \sqcup [\mu''_\cM] \to [\nu''_\cM]$, and consider the homotopy lifting problem:
        \begin{equation}
            \xymatrix{
                &
                \bU^{\Psi,fr}_{0*}(\mu'_\cM, \mu'_\cM) \times_{\cR\cS_{00}} \bU^{const,\Theta}_{00}(\mu''_\cM, \mu''_\cM)
                \ar[d]_\oplus 
                \\
                \cM 
                \ar[r]_{\gamma_\cM \# \rho''_\cM} 
                \ar@{-->}[ur]
                &
                \bU^\Psi_{00}(\nu''_\cM, \nu''_\cM)
            }
        \end{equation}
        By the arguments above, the vertical map is highly connected, so we may choose a homotopy lift; write $(\sigma'_\cM, \sigma''_\cM)$ for it. We compose $\sigma''_\cM$ with the map $\bU^{const,\Theta}_{00}(\mu''_\cM, \mu''_\cM) \to \Theta(\mu''_\cM)$. 

        Now choose some framed ``co-cap'' $C \in \bU^{fr,fr}_{1*}(0, \mu'_\cM)$; there is a contractible space of such choices as before. We choose a homotopy lift:
        \begin{equation}
            \xymatrix{  
                &
                \bU^\Psi_{01}(\mu'_\cM, \mu'_\cM) 
                \ar[d]_{\cdot \# C}
                \\
                \cM 
                \ar[r]_{\mu'_\cM}
                \ar@{-->}[ur]
                &
                \bU^{\Psi,fr}_{0*}(\mu'_\cM,\mu'_\cM)
            }
        \end{equation}  
        We then choose a further homotopy lift along the (highly-connected) inclusion $\bU^\Psi(\mu'_\cM) \to \bU^\Psi_{01}(\mu'_\cM, \mu'_\cM)$. 

        Combining the above, we have constructed a map $\cM \to \bU^\Psi(\mu'_\cM) \times \Theta(\mu''_\cM)$, which may be extended to a map of Thom spaces; we call this $\rho_\cM$. 
    \subsubsection*{Relativity to $\phi$}
        Evaluation at the interior marked point determines maps $\bU^{\Psi,fr}_{0*}(v,n) \to \Phi(n)$ and $\bU^{const,\Theta}_{00}(v,n) \to \Theta(n) \to \Phi(n)$. The first of these is nullhomotopic, via a nullhomotopy given by evaluating along the path from the interior marked point to the boundary marked point (whose evaluation to $\Phi(n)$ is trivial by definition). 
        
        Since each of the lifts in the above steps are compatible with maps to $\Phi(\cdot)$, we find that the composition $\cM \to X \to \Phi(N)$ factors through the map $(\sigma'_\cM, \sigma''_\cM)$. By the above observation, we then have that the following diagram commutes:
        \begin{equation}
            \xymatrix{
                \cM
                \ar[r]_\eps
                \ar[d]_{\rho_\cM} 
                &
                X
                \ar[r]_\phi 
                &
                \Phi(N)
                \\
                \bU^\Psi(\mu'_\cM) \times \Theta(\mu''_\cM) 
                \ar[r]
                &
                \Theta(\mu''_\cM)
                \ar[r]
                &
                \Theta(N) 
                \ar[u]
            }
        \end{equation}  
        so $\rho_\cM$ is exactly the data of a $E_\Psi$-orientation relative to $\phi$.

        The intermediate choices can be made so that the resulting $E_\Psi$-orientations relative to $\phi$ are compatible with boundary breaking of moduli spaces.
\end{proof}

\section{Floer moduli spaces and their combinatorics}\label{sec: comb}

    In later sections, it will be convenient to keep track of the collection of all moduli spaces considered in Section \ref{sec: fuk no loc}.
\subsection{Abstract moduli spaces and gluing}
    \begin{defn}\label{def: Floer mod spac}
        Assume we have fixed a choice of Lagrangian data $\cL$. 

        The set of \emph{Floer moduli spaces (with respect to $\cL$)} is defined to be the set of all moduli spaces considered in Theorems \ref{thmdef:unor fuk} and \ref{thm: unor OC mod}, with boundary conditions given by Lagrangians in $\cL$.

        Each Floer moduli space is a manifold with faces, with a system of faces given by products of other Floer moduli spaces. Let $f: \cA_1 \times \ldots \times \cA_k \to \cM$ be an inclusion of a face, from a product of Floer moduli spaces into another Floer moduli space. We call such a map $f$ a \emph{Floer gluing}.

        We call a Floer gluing which is of the form $\cA_1 \to \cM$ (i.e. $k=1$) and codimension 1 an \emph{unbroken Floer gluing}. We call a Floer gluing which is of the form $\cA_1 \times \cA_2 \to \cM$ and codimension 1 a \emph{simple (broken) Floer gluing}.
    \end{defn}
    \begin{rmk}\label{rem: Floer glue bloc}
        Any Floer gluing $\cA_1 \times \ldots \times \cA_k \to \cM$ is a composition of products of unbroken and simple unbroken Floer gluings. In particular, it will often be sufficient to check various associativity relations on Floer gluings just of this form.
    \end{rmk}
    \begin{rmk}\label{rmk: Floer glue assoc}
        The definitions in Section \ref{sec: flow} along with Theorems \ref{thmdef:unor fuk} and \ref{thm: unor OC mod} enforce various associativity relations on Floer gluing maps; essentially, whenever one can write down a square of Floer gluing maps, it commutes. More explicitly:
        \begin{itemize}
            \item Let $\cM \times \cN \to \cQ$, $\cN \times \cP \to \cR$, $\cM \times \cR \to \cS$ and $\cQ \times \cP \to \cS$ be Floer gluings. Then the following diagram commutes:
            \begin{equation}
                \xymatrix{
                    \cM \times \cN \times \cP
                    \ar[r]
                    \ar[d]
                    &
                    \cM \times \cR 
                    \ar[d]
                    \\
                    \cQ \times \cP 
                    \ar[r]
                    &
                    \cS
                }
            \end{equation}
            \item Let $\cM \times \cN \to \cP$, $\cN \to \cQ$ and $\cM \times \cQ \to \cP$ be Floer gluings. Then the following diagram commutes:
            \begin{equation}
                \xymatrix{
                    \cM \times \cN 
                    \ar[r]
                    \ar[dr]
                    &
                    \cM \times \cQ 
                    \ar[d]
                    \\
                    &
                    \cP
                }
            \end{equation}
    
            A similar compatibility holds when $\cM$ and $\cN$ are swapped.
            \item Let $\cM \to \cN$, $\cN \to \cP$ and $\cM \to \cP$ be Floer gluings. Then the following diagram commutes:
            \begin{equation}
                \xymatrix{
                    \cM 
                    \ar[r]
                    \ar[dr]
                    &
                    \cN
                    \ar[d]
                    \\
                    &
                    \cP
                }
            \end{equation}
        \end{itemize}
        By Remark \ref{rem: Floer glue bloc}, similar associativity relations hold for all Floer gluings of the form $\cA_1 \times \ldots \times \cA_k \to \cM$, but they are all implied by the ones above.
    \end{rmk}
    
    We assumed $\cL$ was finite. Since for each type of structure considered in Theorems \ref{thmdef:unor fuk} and \ref{thm: unor OC mod} there are only finitely many Floer moduli spaces for each pair/triple/quadruple of Lagrangians, we see that:
    \begin{lem}\label{lem: fin many mod spac}
        Assume $\cL$ is finite. Then the collection of Floer moduli spaces is finite. 
        
        In particular, there exists some $n^{max} \gg 0$ such that all Floer moduli spaces have dimension $\leq n^{max}$.
    \end{lem}

\subsection{Boundary labelling data}
    \begin{defn}\label{def: 8.8}
        An \emph{$\cL$-labelled domain} $\Delta$ consists of the following data:
        \begin{itemize}
            \item Integers $0 \leq i \leq 3$ and $0 \leq j \leq 1$.
            \item A set of $i$ ``input'' marked points $x_1, \ldots, x_i$ and $j$ ``output'' marked points (called $x_0$ if $j=1$) on the boundary of the disc $\partial D^2$.
            \item A labelling of each boundary component $b$ of $D^2 \setminus \{x_l\}_l$ with an object $L \in \cL$.

            We write $\Cptd\Delta$ for the set of such boundary components.
            \item A labelling of each $x_l$ with an object (that we also call $x_l$) of the flow category $\cM^{L_{b}, L_{b'}}$, where $b, b' \in \Cptd \Delta$ are the two boundary components touching $x_l$.

            Here $b$ and $b'$ are ordered according to the following convention. If $x_l$ is an input, $b'$ is clockwise of $x_l$ (and $b$ is anticlockwise of $x_l$), and the opposite way around if $x_l$ is an output (i.e. $l=0$).

        \end{itemize}
        We require that the inputs $x_1, \ldots, x_i$ are ordered anticlockwise, so that if $j=1$, the output lies directly between $x_i$ and $x_1$.

        We consider two $\cL$-labelled domains as the same if there is an oriented homeomorphism of the disc, relating the two $\cL$-labellings.

        We write $\cD_{ij}$ for the set of $\cL$-labelled boundary data with $i$ inputs and $j$ outputs.
    \end{defn}
    \begin{defn}
        Any Floer moduli space $\cM$ determines a $\cL$-labelled domain $\Delta_\cM$ as follows. The interior of $\cM$ is defined to be the space of maps from a punctured disc $D$ with boundary punctures to $X$ satisfying a PDE (a perturbed Cauchy-Riemann equation) with Lagrangian boundary conditions, and which are required to asymptote to some fixed Hamiltonian chord at each puncture. Choose a homeomorphism between $D$ and a punctured copy of $D^2$. The boundary conditions for the PDE determine, for each component $b$ of $D$, an object $L \in \cL$. The asymptotic conditions at each puncture of $D$ determine some $x \in \cM^{L_b, L_{b'}}$, where $b$ and $b'$ are the boundary components neighbouring $x$, ordered as in Definition \ref{def: 8.8}.  
        
        For any product $\cF = \prod_l \cA_l$ of Floer moduli spaces, we set $\Cptd\Delta_\cF = \sqcup_l\Cptd\Delta_{\cA_l}$.
    \end{defn}

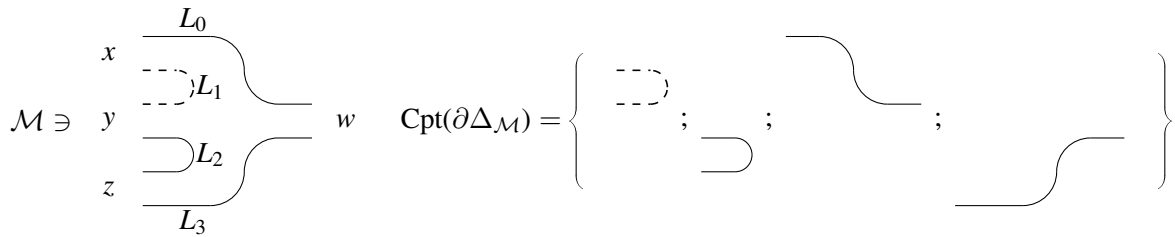
\begin{figure}[ht]
\begin{center}
\begin{tikzpicture}[scale=0.9]

\draw (0.5,0) node {$\mathcal{M} \ni $};
\draw (1.5,0) node {$y$};
\draw (1.5,1) node {$x$};
\draw (1.5,-1) node {$z$};
\draw (5,0) node {$w$};

\draw (2.75,1.5) node {$L_0$};
\draw (2.75,-1.5) node {$L_3$};
\draw (3,0.5) node {$\small{L_1}$};
\draw (3,-0.5) node {$\small{L_2}$};

\draw[semithick,dashed] (2,0.25) -- (2.5,0.25);
\draw[semithick,dashed] (2.5,0.25) arc (-90:0:0.25);
\draw[semithick,dashed] (2.75,0.5) arc (0:90:0.25);
\draw[semithick,dashed] (2,0.75) -- (2.5,0.75);

\draw (2,-0.75) -- (2.5,-0.75);
\draw (2,-0.25) -- (2.5,-0.25);
\draw (2.5,-0.25) arc (90:0:0.25);
\draw (2.75,-0.5) arc (0:-90:0.25);

\draw (2,1.25) -- (3,1.25);
\draw (3,1.25) arc (90:0:0.5);
\draw (3.5,0.75) arc (-180:-90:0.5);
\draw (2,-1.25) -- (3,-1.25);

\draw (3,-1.25) arc (-90:0:0.5);
\draw (3.5,-0.75) arc (180:90:0.5);
\draw (4,0.25) -- (4.5,0.25);
\draw (4,-0.25) -- (4.5,-0.25);

\draw (7,0) node {$\mathrm{Cpt}(\partial \Delta_{\mathcal{M}}) =$};

\draw[decoration = {calligraphic brace, amplitude=5pt}, decorate] (8.5,-1) -- (8.5,1);
\draw[decoration = {calligraphic brace, mirror, amplitude=5pt}, decorate] (17,-1) -- (17,1);

\draw[semithick,dashed] (9,0.25) -- (9.5,0.25);
\draw[semithick,dashed] (9.5,0.25) arc (-90:0:0.25);
\draw[semithick,dashed] (9.75,0.5) arc (0:90:0.25);
\draw[semithick,dashed] (9,0.75) -- (9.5,0.75);
\draw (10,0) node {;};

\draw (10.25,-0.75) -- (10.75,-0.75);
\draw (10.25,-0.25) -- (10.75,-0.25);
\draw (10.75,-0.25) arc (90:0:0.25);
\draw (11,-0.5) arc (0:-90:0.25);
\draw (11.25,0) node {;};

\draw (11.5,1.25) -- (12,1.25);
\draw (12,1.25) arc (90:0:0.5);
\draw (12.5,0.75) arc (-180:-90:0.5);
\draw (13,0.25) -- (13.5,0.25);
\draw (13.75,0) node {;};

\draw (14,-1.25) -- (15,-1.25);
\draw (15,-1.25) arc (-90:0:0.5);
\draw (15.5,-0.75) arc (180:90:0.5);
\draw (16,-0.25) -- (16.5,-0.25);

\end{tikzpicture}
\end{center}
\caption{Boundary components $\mathrm{Cpt}(\partial \Delta_{\mathcal{M}})$ for a moduli space $\mathcal{M}$\label{fig: rrr}.}
\end{figure}

    An example is shown in Figure \ref{fig: rrr}.
    \begin{rmk}
        Note that even though as a Riemann surface the domain can vary over the moduli space, the isomorphism class of $\cL$-labelled boundary data does not.
    \end{rmk}
    \begin{rmk}
        Each $\Delta_\cM$ lies in $\cD_{ij}$, where $i$ is the number of inputs of the domain of $\cM$ and $j$ is the number of outputs.
    \end{rmk}
    \begin{rmk}

        For a simple Floer gluing $\cM \times \cN \to \cP$, the output of (WLOG) $\cM$ is glued to the $l^{th}$ input of $\cN$ (with respect to the anticlockwise ordering of the inputs) for some $l$. We say $l$ is the \emph{gluing number} of this Floer gluing.
        
    \end{rmk}
    We define the set of \emph{Floer generators} $\cX(\cL)$ to be the disjoint union $\sqcup_{L,K \in \cL} \cX(L,K)$.
    \begin{defn}
        There are maps $s_l: \cD_{ij} \to \cX(\cL)$  where $s_l$ sends $\Delta$ to its $l^{\textrm{th}}$ input. If $j=1$, there is a similar map $t: \cD_{i1} \to \cX(\cL)$ sending $\Delta$ to its output.

        For $1 \leq l \leq i$ and $i+i'-1 \leq 3$, there are maps of sets:
        \begin{equation}
            \lambda_l: \cD_{i1} \prescript{}{t}\times_{s_l} \cD_{i'j} \to \cD_{i+i'-1,j}
        \end{equation}
        where we take the fibre product over $\cX(\cL)$. This is defined by gluing two labelled discs $\Delta,\Delta'$ together, by gluing the output of $\Delta$ to the $l^{\mathrm{th}}$ input of $\Delta'$. These are suitably associative, in the sense that the following diagram commutes:
        \begin{equation}\label{eq: cD assoc}
            \xymatrix{
                \cD_{i1} \prescript{}t\times_{s_l} \cD_{i'1} \prescript{}t\times_{s_l'} \cD_{i'j}
                \ar[r]
                \ar[d]
                &
                \cD_{i1} \prescript{}t\times_{s_l} \cD_{i'+i''-1,j}
                \ar[d]
                \\
                \cD_{i+i'-1,1} \prescript{}t\times_{s_{l'}} \cD_{i''j}
                \ar[r]
                &
                \cD_{i+i'+i''-2,j}
            }
        \end{equation}
        
        We call pairs $(\Delta, \Delta')$ which we can apply $\lambda_l$ to (i.e. they lie in $\cD_{i1} \prescript{}t\times_{s_l} \cD_{i'j}$, and $i+i'-1 \leq 3$) \emph{$l$-composable}.
    \end{defn}
    An example is shown in Figure \ref{fig: sss}.

\begin{figure}[ht]
\begin{center}
\begin{tikzpicture}

\draw (1.75,0) node {$y$};
\draw (1.75,1) node {$x$};
\draw (1.75,-1) node {$w$};
\draw (4.75,0) node {$r$};
\draw (2.75,1.5) node {$L_0$};
\draw (2.75,-1.5) node {$L_3$};
\draw (3,0.5) node {$\small{L_1}$};
\draw (3,-0.5) node {$\small{L_2}$};

\draw (2,0.25) -- (2.5,0.25);
\draw (2,-0.25) -- (2.5,-0.25);
\draw (2.5,0.25) arc (-90:0:0.25);
\draw (2.75,0.5) arc (0:90:0.25);
\draw (2,0.75) -- (2.5,0.75);
\draw (2,-0.75) -- (2.5,-0.75);

\draw (2.5,-0.25) arc (90:0:0.25);
\draw (2.75,-0.5) arc (0:-90:0.25);
\draw (2,1.25) -- (3,1.25);
\draw (3,1.25) arc (90:0:0.5);
\draw (3.5,0.75) arc (-180:-90:0.5);
\draw (2,-1.25) -- (3,-1.25);
\draw (3,-1.25) arc (-90:0:0.5);
\draw (3.5,-0.75) arc (180:90:0.5);
\draw (4,0.25) -- (4.5,0.25);
\draw (4,-0.25) -- (4.5,-0.25);

\draw (1,0) node {$\stackrel{\lambda_1}{\longmapsto}$};

\draw [semithick] (-6,-1.25) to [round left paren ] (-6,1.25);
    \draw [semithick] (-0.5,-1.25) to [round right paren] (-0.5,1.25);

\draw (-4.25,0.75) node {$L_0$};
\draw (-4.25,-0.25) node {$L_2$};
\draw (-5,0.5) node {$L_1$};

\draw (-5.75,0.75) node {$x$};
\draw (-5.75,-0.25) node {$y$};
\draw (-3.75,0.25) node {$z$};

\draw (-4,0) -- (-4.5,0);
\draw (-4,0.5) -- (-4.5,0.5);
\draw (-4.5,0) arc (90:180:0.25);
\draw (-4.75,-0.25) arc (0:-90:0.25);
\draw (-5,-0.5) -- (-5.5,-0.5);
\draw (-4.5,0.5) arc (-90:-180:0.25);
\draw (-4.75,0.75) arc (0:90:0.25);
\draw (-5,1) -- (-5.5,1);
\draw (-5.5,0) arc (-90:90:0.25);

\draw (-3.5,0) node {;};

\draw (-1,0) -- (-1.5,0);
\draw (-1,0.5) -- (-1.5,0.5);
\draw (-1.5,0) arc (90:180:0.25);
\draw (-1.75,-0.25) arc (0:-90:0.25);
\draw (-2,-0.5) -- (-2.5,-0.5);
\draw (-1.5,0.5) arc (-90:-180:0.25);
\draw (-1.75,0.75) arc (0:90:0.25);
\draw (-2,1) -- (-2.5,1);
\draw (-2.5,0) arc (-90:90:0.25);

\draw (-2.75,0.75) node {$z$};
\draw (-2.75,-0.25) node {$w$};
\draw (-0.75,0.25) node {$r$};
\draw (-1.25,0.75) node {$L_0$};
\draw (-1.25,-0.25) node {$L_3$};
\draw (-2,0.5) node {$L_2$};

\end{tikzpicture}
\end{center}
\caption{An example of $\lambda_1: \mathcal{D}_{2,1} \times \mathcal{D}_{2,1} \to \mathcal{D}_{3,1}$. By convention, outputs are on the right.\label{fig: sss}}
\end{figure}
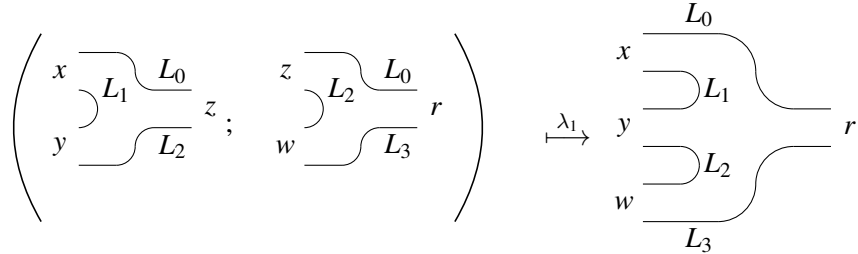

    \begin{rmk}
        Let $\cM \times \cN \to \cP$ be a simple Floer gluing of gluing number $l$. Then $\Delta_\cM$ and $\Delta_\cN$ are $l$-composable, and $\Delta_\cP = \lambda_l(\Delta_\cM, \Delta_\cN)$.

        Let $\cP \to \cQ$ be an unbroken Floer gluing. Then $\Delta_\cP = \Delta_\cQ$.
    \end{rmk}
    \begin{defn}\label{def: boundary}
        Let $(\Delta, \Delta')$ be $l$-composable. We define $\pi=\pi_{\Delta\Delta';l}: \Cptd\Delta \sqcup \Cptd\Delta' \to \Cptd \lambda_l(\Delta, \Delta')$ to send a boundary component to the component it is sent to in the glued disc.

        From the definition, we see that $\pi$ is surjective, and the fibres of $b$ under $\pi$ are a single point unless $b$ is in the image of the boundary component(s) adjacent to the output of $\Delta$. 
        
        When $\cM \times \cN \to \cP$ is a simple Floer concatenation of gluing number $l$, we sometimes write $\pi_{\cM\cN\cP}$ for $\pi_{\Delta_\cM \Delta_\cN;l}$. 
        
        Similarly when $\cP \to \cQ$ is an unbroken Floer face, there is a natural bijection $\Cptd\Delta_\cP \to \Cptd\Delta_\cM$, since $\cP$ and $\cQ$ are defined using the same Lagrangian boundary conditions and asymptotic conditions at the punctures (but not necessarily the same PDE).

        Let $\cF = \prod_l \cA_l \to \cM$ be any Floer gluing. Since it can be factored as a sequence of unbroken and simple Floer gluings (cf. Remark \ref{rem: Floer glue bloc}), we by composing the maps $\pi$, we obtain maps $\pi=\pi_{\cF\cM}:\Cptd\Delta_\cF \to \Cptd\Delta_\cM$; this is well-defined.
    \end{defn}
    The $\pi$ are suitably associative, in the sense that when $\cF_l = \prod_k \cA_{kl} \to \cG_l$, $\prod_l\cG_l \to \cM$ and $\prod_l \cF_l \to \cM$ are Floer gluings, the following diagram commutes:
    \begin{equation}
        \xymatrix{
            \bigsqcup\limits_l \Cptd\Delta_{\cF_l} 
            \ar[r] 
            \ar[dr]
            &
            \bigsqcup\limits_l \Cptd\Delta_{\cG_l} 
            \ar[d]
            \\
            &
            \Cptd\Delta_\cM
        }
    \end{equation}

    For each $L \in \cL$, choose a spider $Y_L \subseteq L$ as in Section \ref{sec:tangerine} through basepoints $\ell_L \in L$. We may assume $Y_L$ passes through (the image in $L$ of) all Floer generators $\cX(L, K)$, $\cX(K, L)$ for all $K$ (note this is a finite subset of $L$). Choose homotopy inverses $\theta_L: L/{Y_L} \to L$ to the collapse maps.

    Assuming all $L \in \cL$ are equipped with Riemannian metrics, there are then maps $ev=ev^\cM$ from each Floer moduli space $\cM$ to the (Moore) based loop spaces of each $L_b/Y_{L_b}$, where $L_b$ are the Lagrangians appearing in the boundary conditions defining $\cM$: these send $u: D \to X$ to the restriction of $u$ to the boundary componend $b$, parametrised anticlockwise. Composing with $\theta_{L_b}$ and reparametrising so each path has speed 1, we obtain maps:
    \begin{equation}\label{eq: ev}
        ev=ev^\cM: \cM \to \prod\limits_{b \in \Cptd\Delta_\cM} \Omega L_b
    \end{equation}
    which are compatible with Floer gluing and concatenation of loops in the natural way.
\subsection{Boundary stabilisation data}
    \begin{defn}\label{def: stab data}
        Let $\cL$ be a choice of Lagrangian data. \emph{Boundary stabilisation data} $S$ for $\cL$ consists of:
        \begin{itemize}
            \item For each Floer moduli space $\cM$ and each $b \in \partial \Delta_\cM$, an element $n_b=n_{b,\cM}^S \in \cI$.

            \item For each Floer gluing $\cF=\prod_l\cA_l \to \cM$, injections of finite sets:
            \begin{equation}\label{eq: n stab}
                \bigsqcup\limits_{b' \in \pi_{\cF\cM}^{-1}\{b\}} [n_{b'}] \to [n_b]
            \end{equation}
        \end{itemize}
        We require that the maps (\ref{eq: n stab}) are suitably associative, meaning whenever $\cF_l =\prod_k \cA_{kl} \to \cG_l$, $\prod_l \cG_l \to \cM$ and $\prod_l \cF_l \to \cM$ are Floer gluings and $b \in \Cptd\Delta_\cM$, the following diagram commutes:
        \begin{equation}\label{eq: n stab ass}
            \xymatrix{
                \bigsqcup\limits_{b'' \in \pi_{\prod_l \cF_l; \cM}^{-1}\{b\}} [n_{b''}]
                \ar[r]_\pi 
                \ar[dr]_-\pi 
                &
                \bigsqcup\limits_{b' \in \pi_{\prod_l \cG_l; \cM}^{-1}\{b\}} [n_{b'}]
                \ar[d]_\pi 
                \\
                &
                [n_b]
            }
        \end{equation}
        A \emph{morphism} of boundary stabilisation data $S \to S'$ consists of injections of finite sets $[n^S_b] \to [n^{S'}_b]$ for all Floer moduli spaces $\cM$ and all $b \in \Cptd\Delta_\cM$, which are suitably compatible with the maps (\ref{eq: n stab})
    \end{defn}
    \begin{lem}\label{lem: stab data uniq conc}
        Let $S,S'$ be two choices of boundary stabilisation data. 
        
        Then there exists a choice $T$ of boundary stabilisation data, such that there are morphisms $S, S' \to T$.
    \end{lem}
    \begin{proof}
        Set $n^T_b = n^S_b + n^{S'}_b$ for all $b$. This induces suitable maps (\ref{eq: n stab}) in a natural way.
    \end{proof}
    
    \begin{lem}\label{lem: ex big S}
        Let $m$ be an integer. Then there exists a choice of boundary stabilisation data $S$ with $n^S_b \geq m$ for all $b$.
    \end{lem}
    \begin{proof}
        We choose all the $n_b$ for $b \in \Cptd\Delta_\cM$, inductively on the dimension of the Floer moduli space $\cM$. Consider the following collection of data:
        \begin{itemize}
            \item Objects $n_b \in \cI$ for all $b \in \Cptd\Delta_\cM$, whenever $\cM$ is a Floer moduli space with $\dim(\cM) \leq k$.
            \item Injections (\ref{eq: n stab}) whenever $\dim(\cP) \leq k$, such that (\ref{eq: n stab ass}) commutes whenever $\dim(\cP) \leq k$.
        \end{itemize}
        For $k=0$, it is clear this exists, by setting $n_b=m$ for all relevant $b$; note there are not yet any conditions of the form (\ref{eq: n stab ass}).

        Now assume we have this data for some $k$. We choose similar data extending this to $k+1$. We define finite sets for all $b \in \Cptd\Delta_\cP$ for Floer moduli spaces with $\dim(\cP) = k+1$:
        \begin{equation}
            T_b := \left(\bigsqcup\limits_{\substack{\cF=\prod_l\cA_l \to \cM \,\textrm{Floer gluing}
            \\
            b' \in \pi_{\cF\cM}^{-1}\{b\}}} [n_b] \right)/\sim
        \end{equation}
        where $\sim$ is the minimal equivalence relation defined such that (\ref{eq: n stab ass}) always commutes for this $\cP$, where each $[n_{b'}] \to T_b$ is the natural inclusion for all $b' \in \pi^{-1}\{b\}$.
        
        Set $n_b = \max\{\#T_b,m\}$. Choose injections $T_b \to [n_b]$; composing with the natural maps $[n_{b'}] \to T_b$ defines the maps required. (\ref{eq: n stab ass}) then commutes whenever it is defined by construction. The fact that each map $[n_{b'}] \to T_b$ is injective follows from the imposed commutativity of (\ref{eq: n stab ass}).

        We therefore obtain boundary stabilisation data $S$ for $\cL$ by induction.

        An alternative more geometric argument is to choose a large finite subset $A \subseteq X$. Then for $b \in \Cptd\Delta_\cM$, let $A_b$ be the set of $x \in A$ such that there is some $u \in \cM$ such that the restriction of $u$ to $b$ hits $x$ under the evaluation map. Let $n_b = \#A_b$ and choose bijections $A_b \cong [n_b]$ arbitrarily. The maps $\pi$, under this bijection, send $x \in A$ to $x$. For $A$ chosen generically, this is always injective.
    \end{proof}

\section{The Fukaya category with monodromy local systems}\label{sec: mon loco sys}

    For the entirety of this section, we fix a commutative $\cI$-monoidal tangential pair (cf. Definition \ref{def: comm tang pair}) $\Psi = (\Theta, \Phi)$, and let $F = \hofib(\Theta \to \Phi)$. Fix a finite CW complex $L$.

    Let $E_\Psi$ be the commutative Thom $\cI$-monoid constructed from $\Psi$ (cf. Definition \ref{def: E Psi}), and $R_\Psi$ the commutative ring spectrum constructed from $E_\Psi$ (cf. Remark \ref{rmk: comm thom ring spec}). Let $GL_1(R_\Psi)$ be its commutative $\cI$-space of units (cf. Definition \ref{def: GL one}); for brevity, we will abbreviate this to $GL_1^\Psi$ throughout much of this section. We write $B(GL_1)_{h\cI}=BGL_1(R_\Psi)_{h\cI}$ for space given by the bar construction applied to the monoid $GL_1(R_\Psi)_{h\cI}$.

    Note that since $E_\Psi$ is connective, so are both $R_\Psi$ and $B(GL_1)_{h\cI}$.

    \begin{lem}\label{lem: pi zero R}
        $R_\Psi$ is a connective ring spectrum. Furthermore:
        \begin{equation}
            \pi_0 R_\Psi \cong \begin{cases}
                \bZ & \textrm{ if $\Psi$ is oriented} \\
                \bZ/2 & \textrm{ otherwise.}
            \end{cases}
        \end{equation}
    \end{lem}
    \begin{proof}
        For the first statement, any Thom spectrum is connective. $F$ is simply connected by assumption so $\Omega F \simeq \Base(E_\Psi)$ is connected. By the stable Hurewicz theorem, $\pi_0 R \cong H_0(R)$. By the Thom isomorphism theorem, $H_0(R) \cong \bZ$ if the index bundle over $\Base(E_\Psi)$ is orientable, and $H_0(R) \cong \bZ/2$ otherwise. 

        By Hurewicz and the fact that $\pi_2 U/O \cong \bZ/2$, the class $[w_2 \circ Re] \in H^2(\widetilde{U/O}; \bZ/2)$ detects $\pi_2$. Therefore it follows from \cite[Theorem 1.1]{Georgieva} (cf. also \cite[Section 3]{deSilva}) that the index bundle is orientable if and only if $F \to \widetilde{U/O}$ is 0 on $\pi_2$.
    \end{proof}

\subsection{Monodromy local systems}

    \begin{defn}
        A \emph{($\Psi$-)(monodromy) local system on $L$ (MLS)} consists of a map of spaces $\xi: L \to B(GL_1)_{h\cI}$.

        A \emph{homotopy} between local systems $\xi$ and $\zeta$ consists of a homotopy between the corresponding maps.
    \end{defn}
    \begin{ex}
        The \emph{trivial local system} $\xi: L \to B(GL_1)_{h\cI}$ is the constant map to the basepoint.
    \end{ex}

    \begin{lem} \label{lem:MLS to classical local system} 
        If $\Psi$ is oriented, there is a function from the set of $\Psi$-monodromy local systems on $L$ to the set of $\bZ$-local systems on $L$.
    \end{lem}

    \begin{proof}
        By Lemma \ref{lem: pi zero R}, there is a map of ring spectra $R \to H\bZ$, hence we obtain a map $BGL_1(R_\Psi) \to BGL_1(H\bZ) \simeq K(\bZ/2, 1)$; postcomposition with this gives the required map.
    \end{proof}
\subsection{Twisting data}
    We fix a set $\cL$ of Lagrangian data as in Section \ref{sec:unoriented category}.
    \begin{defn}\label{def: 95}
        Let $S$ be boundary stabilisation data for $\cL$. \emph{Twisting data} $\Xi$ for $\cL$ (with respect to $S$) consists of maps of spaces:
        \begin{equation}\label{eq: 92}
            \Xi_\cM: \cM \to \prod\limits_{b \in \Cptd\Delta_\cM} GL_1^\Psi(n_b^S)
        \end{equation}
        for each Floer moduli space $\cM$. We require that the following diagram commutes for each Floer gluing map $\prod_l \cA_l \to \cM$:
        \begin{equation}\label{eq: pre tw comp}
            \xymatrix{
                \prod_l \cA_l
                \ar[r]
                \ar[d]
                &
                \prod\limits_l \prod\limits_{b \in \Cptd\Delta_{\cA_l}} GL_1^\Psi(n^S_b)
                \ar[d]
                \\
                \cM
                \ar[r]
                &
                \prod\limits_{b'' \in \Cptd\Delta_\cM} GL_1^\Psi(n_{b''}^S)
            }
        \end{equation}
        Here the right-hand vertical arrow is given by the product induced by (\ref{eq: n stab}).

        A \emph{homotopy} between twisting data $\Xi$ and $\Xi'$ (both with respect to the same boundary stabilisation data $S$) consists of a 1-parameter family of twisting data $\{\Xi_t\}_{t \in [0,1]}$ interpolating between $\Xi$ and $\Xi'$. This can be encoded equivalently as maps $\cM \times [0,1] \to \prod\limits_{b \in \Cptd\Delta_\cM} GL_1^\Psi(n_b^S)$ satisfying a similar associativity condition.

    \end{defn}
    \begin{defn}
        Let $S \to T$ be a morphism of boundary stabilisation data, and $\Xi$ twisting data for $\cL$ with respect to $S$. The \emph{induced} twisting data for $\cL$ with respect to $T$ is given by composing each $\Xi_\cM$ with (products of) the maps $GL_1^\Psi(n_b^S) \to GL_1^\Psi(n_b^T)$ induced by the maps $[n_b^S] \to [n_b^T]$ in $\cI$.

        A \emph{concordance} between twisting data $\Xi$ and $\Xi'$ (with respect to boundary stabilisation data $S$ and $S'$ respectively) consists of morphisms of boundary stabilisation data $S,S' \to T$ for some boundary stabilisation data $T$, along with a homotopy between the two induced sets of twisting data with respect to $T$. 

    \end{defn}
    We now incorporate $\Psi$-monodromy local systems. 
    \begin{defn}
        Let $\cL$ be a choice of Lagrangian data. A choice of \emph{local system data} $\xi_\cL$ on $\cL$ consists of a $\Psi$-monodromy local system $\xi_L: L \to BGL_1(R_\Psi)_{h\cI}$ for each $L \in \cL$.
    \end{defn}
    \begin{prop}
        Let $\xi_\cL$ be a choice of local system data on a set of Lagrangian data $\cL$, and $S$ sufficiently large boundary stabilisation data. Then associated to $\xi_\cL$ there is a preferred homotopy class of twisting data for $\cL$.
    \end{prop}
    
    \begin{proof}
        Assume the boundary stabilisation data $S$ is large enough such that for each Floer moduli space $\cM$, the map induced by inclusions of 0-simplices
        \begin{equation}
            \prod\limits_{b \in \Cptd\Delta_\cM} GL_1^\Psi(n_b) \to \prod\limits_{b \in \Cptd\Delta_\cM} (GL_1^\Psi)_{h\cI} = \left(GL_1^\Psi\right)_{h\cI}^{\Cptd\Delta_\cM}
        \end{equation}
        is $(\dim(\cM)+2)$-connected.

        Consider the following lifting problem for each Floer moduli space $\cM$:
        \begin{equation}\label{eq: Xi lift prob}
            \xymatrix{
                &&
                \prod\limits_{b \in \Cptd\Delta_\cM} GL_1^\Psi(n^S_b)
                \ar[d]_{\iota_0}
                \\
                &&
                \left(GL_1^\Psi\right)_{h\cI}^{\Cptd\Delta_\cM} 
                \ar@{~>}[d]_{\widetilde\iota_1}
                \\
                \cM
                \ar[r]_{ev}
                \ar@{-->}[uurr]_{\exists ?}
                &
                \prod\limits_{b \in \Cptd\Delta_\cM} \Omega L
                \ar[r]_{\Omega \xi}
                &
                \left(\Omega B(GL_1^\Psi)_{h\cI}\right)^{\Cptd\Delta_\cM} 
            }
        \end{equation}
        
        The upper vertical arrow is highly-connected by the choice of $S$. $\widetilde\iota_1$ is the zig-zag from Lemma \ref{lem: zig zag}; recall each map in the zig-zag is a monoid map which is a homotopy equivalence of spaces. The map $ev$ is from (\ref{eq: ev}).

        Therefore for each Floer moduli space $\cM$ there exists a homotopy lift: this is a diagonal map as in the diagram, along with a homotopy between the two ways around the diagram. Similarly, if such a homotopy lift is already chosen over the subspace given by the boundary $\partial \cM$, we may choose a homotopy lift over all of $\cM$ to extend this one.

        Using this, by inducting over  $\dim(\cM)$ (which we can do, since $\sum_l \dim(\cA_l) < \dim(\cM)$ whenever $\prod_l \cA_l \to \cM$ is a Floer gluing), we may choose homotopy lifts as above so that the resulting maps $\cM \to \Pi_{b \in \Cptd\Delta_\cM} GL_1^\Psi(n^S_b)$ are suitably associative, and so assemble into a set of twisting data with respect to $S$.

        The same argument shows that for any two such sets of twisting data $\Xi, \Xi'$ constructed in this way, there is a homotopy between them.
    \end{proof}

\subsection{Transverse sections}
    
    Let $(X, E)$ be a Thom space, and $M$ a manifold. Assume the rank of $E$ is at least 5. Let $f: M \to \Thom(E \to X)$ be a continuous map. 

    Let $M' = f^{-1}\Tot(E \to X)$ and $M_0 = f^{-1} \{0\} \subseteq M'$, both subspaces of $M$. The map $f: M' \to \Tot(E \to X)$ defines a section $s_f$ of the pullback bundle $f^*E \to M'$ (really this is $(\pi \circ f)^* E$, where $\pi: E \to X$ is the projection). Its zero set is exactly $M_0$.

    $\Tot(f^*E \to M')$ does not automatically have a smooth structure. However, its tangent microbundle has a canonical vector bundle lift, given by $f^*E \oplus TM'$, so by smoothing theory (using the rank assumption above), it  can be equipped with a smooth structure canonically up to contractible choice; we choose such a structure implicitly.

    \begin{rmk}
        The smoothing of a topological manifold with a vector bundle lift of its tangent microbundle admits a relative version.  Given a manifold with faces,  flow-category,  etc, one can therefore construct such smoothings compatibly over the strata, cf. the analogous inductive argument given in \cite{Bai-Xu}.
    \end{rmk}

    We say $f$ is \emph{smooth along 0} if $s_f$ is smooth along $M_0$, and if this holds, that $f$ is \emph{transverse along 0} if the vertical derivative $TM|_{M_0} \to f^*E$ is fibrewise surjective. If $f$ is smooth and transverse along 0, $M_0$ is naturally a smooth manifold. The following is standard:
    
    \begin{lem}\label{lem: gen tran}
        Let $f: M \to \Thom(E \to X)$ be as above. Then there is a $C^0$-small perturbation $f'$ of $f$ that is smooth and transverse along 0.

        Furthermore, if $f$ is smooth and transverse over some open subset $U \subseteq M$, we can take $f'$ to agree with $f$ over $U$.

        If $f$ is smooth and transverse along 0, there is a canonical (up to contractible choice) isomorphism of vector bundles over $M_0$:
        \begin{equation}\label{eqn:tangent bundle of transverse zero-set}
            TM_0 \oplus f^*E \cong TM|_{M_0}
        \end{equation}
    \end{lem}   
    If $M$ has boundary (or faces), we additionally require that $f$ is smooth and transverse over each boundary face. In that case, $M_0$ is also a manifold with boundary (or faces, with a system of faces induced by that of $M$). (\ref{eqn:tangent bundle of transverse zero-set}) is compatible with passing to boundary faces in the natural way.

\subsection{Twisted moduli spaces}\label{sec:twisted moduli}
    Assume we are given some twisting data $\Xi$. Recall that by definition, $GL_1(n)$ is a union of components of $\Omega^n \Thom(\bV^\Psi(n) \to \bU^\Psi(n))$. Hence the maps $\Xi_\cM$ give maps, for each Floer moduli space $\cM$:
    \begin{equation}\label{eqn:map defining twisted moduli}
        \hat\Xi_\cM: \cM \times S^{(\sum_b n_b)} \to  \Thom\left(\bigoplus_b\bV^\Psi(n_b) \to \prod_b\bU^\Psi(n_b)\right)
    \end{equation}
    where $b$ runs over $b \in \Cptd\Delta_\cM$.
    By Lemma \ref{lem: gen tran}, we may assume these maps are all smooth and transverse along 0. We define the \emph{twisted Floer moduli spaces} to be the preimage of 0 of each $\hat \Xi_\cM$. We use superscripts $\xi$ to denote the twisted Floer moduli spaces: for an arbitrary Floer moduli space, we write its twisted version $\hat\Xi_\cM^{-1}\{0\}$ as $\cM^\xi$, and for a specific Floer moduli space e.g. $\cM^{LL'}_{xx'}$ associated to Lagrangians $L, L'$ and $x, x' \in \cM^{LL'}$, we write $\cM^{\xi_L, \xi_{L'}}_{xx'}$ for its twisted version. Note that since the rank of each $\bV^\Psi(n_b)$ is $n_b$, this manifold has the same dimension as $\cM$.

    Since the maps $\Xi_\cM$ are compatible with Floer gluing maps, the $\hat \cM$ inherit similar maps $\prod_l \hat \cA_l \to \hat\cM$. In particular, we obtain all the same algebraic structure: for example, given $L, L' \in \cL$ with local systems $\xi_L, \xi_{L'}$, the $\cM^{\xi_L,\xi_{L'}}_{xx'}$ assemble to form a flow category.

\subsection{$E_\Psi$-orientations on twisted moduli spaces}\label{sec: E Psi or tw modu}

We now return to Equation \ref{eqn:map defining twisted moduli} and the twisted moduli space $\cM^{\xi} = \hat\Xi_{\cM}^{-1}(0)$.  The $E_{\Psi}$-orientation on the moduli space $\cM$ is given by maps
\begin{equation} \label{eqn:input twisting}
(\cM,I^{\cM})(\nu_{\cM}) \to E(\nu_{\cM}+d_{\cM}) = (\bU^{\Psi},\bV^\Psi)(\nu_{\cM} + d_{\cM})
\end{equation}
(perhaps with a further shift $[i]$ on the RHS if the moduli space $\cM$ arises from a morphism $(i=1)$, bordism of morphisms or associator $(i=2)$, etc), which induce isomorphisms of pullback bundles.  

Consider Equation \ref{eqn:tangent bundle of transverse zero-set}, with $M_0 = \cM^{\xi} \subset \cM \times S^n = M$, with $n = \sum_{b \in \Cpt(\partial_{\cM})} n_b$, and with $E = \oplus_b \bV^{\Psi}(n_b)$. We see that
\[
T\cM^{\xi} \oplus \oplus_b \bV^{\Psi}(n_b) = T\cM \oplus \bR^n.
\]
Together with \eqref{eqn:input twisting}, we have
\begin{equation}\label{eqn:twisted index before further stabilisation}
I^{\cM^{\xi}}(\nu_{\cM}) \oplus \oplus_b \bV^{\Psi}(n_b) = \bV^{\Psi}(\nu_{\cM} + d_{\cM}) \oplus \bR^n
\end{equation}
Recall the notion of an inverse for a map to an $\cI$-monoid, as in Definition \ref{defn:inverse in  I-monoid}.

\begin{lem}\label{lem:coherent inverses}
    Fix a finite collection of Floer moduli spaces $\{\cM\}$ with stabilisation, boundary stabilisation and index data. Construct twisted moduli spaces $\{\cM^{\xi}\}$ as above. We may choose:
    \begin{enumerate}
        \item $n_{\cM^\xi}$ and $l_{\cM^\xi}$ and injections $[\sum_{b\subset \partial\cM} n_b] \sqcup [n_{\cM^\xi}] \hookrightarrow [l_{\cM^\xi}]$;
        \item maps $\cM^{\xi} \to \bU(n_{\cM^\xi})$ inverse to the maps $\cM^{\xi} \to \bU^{\Psi}(\sum_b n_b)$
    \end{enumerate}
    coherently under breaking.
\end{lem}

\begin{proof}[Sketch]
 For a single moduli space, existence is given by Lemma \ref{lem:inverse maps to I-monoids exist}.   We pick $n_{\cM^\xi}, l_{\cM^\xi}$, and the inverse map, together with associated nullhomotopy of the map $\cM^{\xi} \to \bU^{\Psi}(\sum_n n_b) \times \bU^\Psi(n_{\cM^\xi}) \to \bU^\Psi(l_{\cM^\xi})$, inductively in the dimension of $\cM^{\xi}$ (equivalently the dimension of $\cM$).  Since, by  Lemma \ref{lem:inverse maps to I-monoids exist}, the space of choices is arbitrarily highly connected for large $n_{\cM^\xi}, l_{\cM^\xi}$, these choices can always be extended from boundary strata to higher dimensional moduli spaces in the given finite collection.
\end{proof}

Adding $\bV^\Psi(n_{\cM^{\xi}}) \oplus \bR^{l_{\cM^{\xi}} - \sum_b n_b - n_{\cM^{\xi}}}$ to both sides of \eqref{eqn:twisted index before further stabilisation}, where the first term is pulled back via the inverse map constructed in Lemma \ref{lem:coherent inverses}, we obtain that
\[
I^{\cM^{\xi}}(\nu_{\cM}) \oplus \bR^{l_{\cM^{\xi}}} = \bV^{\Psi}(\nu_{\cM} + d_{\cM}) \oplus \bV^{\Psi}(n_{\cM^{\xi}}) \oplus \bR^{l_{\cM^{\xi}} - n_{\cM^{\xi}}}
\]
where the two $\bV$-terms on the RHS arise respectively from the maps $\cM^{\xi} \to \cM \to \bU^{\Psi}(\nu_{\cM}+d_{\cM})$ and the `inverse map' $\cM^{\xi} \to \bU^\Psi(n_{\cM^{\xi}})$.

Combining this with the maps 
\[
\bV^{\Psi}(\bullet) \oplus \bR^k \to \bV^{\Psi}(\bullet+k)
\]
from the Thom $\cI$-space structure, we find that if we define
\begin{equation}
    \nu_{\cM^{\xi}} = \nu_{\cM} + l_{\cM^{\xi}} \qquad \mathrm{and} \quad d_{\cM^{\xi}} = d_{\cM} + l_{\cM^{\xi}}
\end{equation}
then we obtain valid choices of index and stabilisation data for the twisted moduli space $\cM^{\xi}$. 

The coherence of the classifying maps to $E_\Psi(\nu_{\cM^{\xi}} + d_{\cM^{\xi}})$ under breaking is inherited from (i) the corresponding coherence for the $E_{\Psi}$-orientation on $\cM$, (ii) the compatibility of the maps $\Xi_{\cM}$ with Floer gluings, and (iii) the coherence in the construction of the inverse maps in Lemma \ref{lem:coherent inverses}. This serves to define coherent $E_{\Psi}$-orientations on the twisted moduli spaces $\{\cM^{\xi}\}$.

\subsection{The Fukaya category with local systems}\label{sec: fuk loc sys}

We pick a finite set of Lagrangians $\cL$, and for each one pick some finite collection of $\Psi$-monodromy local systems $\{\xi^{(j)}_L \, | \,  1 \leq j \leq J(L), \, L \in \cL\}$ ($J(L)$ is some number).  We define a spectral (Donaldson-)Fukaya category $\scrF^{loc}(X; \Psi)$ with objects $(L,\xi_L)$, where $\xi_L \to L$ is one of this finite set of distinguished monodromy local systems. We will typically abbreviate the pair ($L,\xi_L)$ to just $\xi_L$, and we call such pairs \emph{$\Psi$-branes}.

We fix a commutative $\cI$-monoidal tangential structure $\Psi$ and $E_\Psi$-orientations on all the twisted moduli spaces, as above. We thus have twists
\[
\cM^{\xi_L,\xi_K}
\]
of $\cM^{LK}$ as in Section \ref{sec:twisted moduli}, which are $E_\Psi$-oriented flow categories.

\begin{defn}
    The morphism group $\scrF^{loc}_i((L,\xi_L), (K,\xi_K); \Psi)$ is the group of bordism classes of right flow modules $\Omega_i^{E_\Psi}(\cM^{\xi_L,\xi_K})$.
\end{defn}

It follows from the discussion in Section \ref{sec:twisted moduli} that all the algebraic structures in Theorems \ref{thmdef:unor fuk}, \ref{thm: unor OC mod} and \ref{thm:ind comm} exist for tuples of pairs $(L, \xi_L)$ rather than just Lagrangians $L$.

Therefore in the same way as in Definition \ref{def: E fuk}, the objects $\xi_L$ and morphism groups above define a category $\scrF^{\loc}(X;\Psi)$.

There are truncations $\tau_{\leq i} \scrF^{loc}(X; \Psi)$, similarly to Definition \ref{def: trunc fuk}.

\begin{defn}\label{def: oc loco}
    Let $(L, \xi_L)$ be an object in $\scrF^{loc}(X; \Psi)$. We define the class $[L, \xi] \in \Omega^{E_\Psi,\cO\cC}_d(L, \phi|_L)$ to be the bordism class of the closed manifold $\Upsilon_{\xi_L} \circ \cE^{\xi_L}$, with its natural $E_\Psi$-orientation relative to $\phi|_L$. 
\end{defn}
\begin{rmk}
    Note that in contrast to Lemma \ref{lem: def oc or}, the definition of $[L,\xi_L]$ is not obviously topological: it uses the Floer moduli spaces. We give a topological description of this class in Section \ref{sec: OC}. 
\end{rmk}
\begin{prop}\label{prop: priojgoierghorhgv}
    Let $(L, \xi_L)$ and $(K, \xi_{K})$ be objects in $\scrF^{loc}(X; \Psi)$.

    If $(L, \xi_L)$ and $(K, \xi_K)$ are isomorphic, then $\iota^L_*[L,\xi_L] = \iota^{K}_*[K,\xi_K]$.
\end{prop}
\begin{proof}
    Similarly to Proposition \ref{prop: quas same fund no loc}, this follows from Theorem \ref{thm: unor OC mod}(4) and (5), twisting the appropriate moduli spaces by the local systems as usual.
\end{proof}

\subsection{Invariance}

We fix Lagrangian and local system data $\cL$, meaning a finite collection of pairs $(L,\xi_L)$ of Lagrangians (with their tangential lifts) equipped with spectral local systems $\xi_L : L \to B(GL_1)_{h\cI}$, as in Section \ref{sec: fuk loc sys}. We always suppose that $\cL \supset \cL^{\triv}$ includes the $\Psi$-branes associated to the trivial local system on each Lagrangian $L$ appearing in some pair in $\cL$.

    \begin{lem}\label{lem:concordance gives isomorphism}
        Let $(L, \xi_L)$ and $(K, \xi_K)$ be two pairs in $\cL$.
        
        Then flow category $\cM^{\xi_L, \xi_K}$ is well-defined up to isomorphism in $\Flow^{\Psi}$: it does not depend on the choice of Floer data, boundar stabilisation data, nor the choice of representative $\Xi$ of the equivalence class of twistings defined from $\cL$. 

        Its isomorphism class is also invariant under replacing the $\Psi$-local systems $\xi_L$ with homotopic ones.
    \end{lem}

    \begin{proof}
        The local systems, along with a choice of boundary stabilisation data $S$, determines a unique homotopy class of transverse twisting data.  A homotopy  of transverse twisting data defines a morphism of the twisted flow categories associated to the transverse end-points, composition with which is an equivalence in $\Flow$ (one can build an inverse in a similar manner). Coherently $\Psi$-orienting the homotopy lifts this equivalence to $\Flow^{\Psi}$.  Changing the boundary stabilisation data up to concordance (so changing the ranks of the twisting data) yields a stabilisation of the associated twisted flow categories in the sense of \cite[Remark 5.8]{PS2}. 
    \end{proof}

\begin{prop}
    The subcategory $\scrF^{\loc;\cL}(X;\Psi)$ associated to the objects $\cL$ is well-defined up to quasi-isomorphism.
\end{prop}

\begin{proof}
    Independence of boundary stabilisation data follows from Lemma \ref{lem:concordance gives isomorphism} (a similar argument applies to bilinear maps). Independence of the generic almost complex structures and perturbations used in the construction follows as in the case without local systems.

    Independence with respect to the choice of spider follows from the same argument as in \cite[Section 6.2]{BDHO}.
\end{proof}

If $\zeta_L \to L$ is a MLS and $\{L_t\}$ is a Hamiltonian isotopy from $L = L_0$ to $L' = L_1$, there is an induced MLS $\zeta_{L'} \to L'$. By expanding if necessary, suppose both $(L,\zeta_L)$ and $(L',\zeta_{L'})$ belong to the fixed Lagrangian data $\cL$. 

\begin{lem}
    The pairs $(L,\zeta_L)$ and $(L',\zeta_{L'})$ are quasi-isomorphic in $\scrF^{\loc;\cL}(X,\Psi)$.
\end{lem}

\begin{proof}
    Follows from a standard continuation map argument, twisted as in the proof of Lemma \ref{lem:concordance gives isomorphism}.
\end{proof}

    \begin{prop}
        The full subcategory
        \[
        \scrF^{\loc,\cL^{triv}}(X;\Psi) \subset \scrF^{\loc,\cL}(X;\Psi)
        \]
        is isomorphic to the category with the same objects as defined in Section \ref{sec: tang fuk}.
    \end{prop}
    \begin{proof}
        All the Floer moduli spaces relevant to $\cL^{\triv}$ carry trivial local systems. We take boundary stabilisation data $S \equiv \{0\}$ on all moduli spaces, so $GL_1(n_b^S) = \{pt\}$ reduces to the base-point. Note that the inductive argument which constructs pretwisting data goes through for this $S$, since $\Omega\xi$ is a constant map. The associated twisting data can then be taken to give maps $\hat\Xi_\cM: \cM \to GL_1(0) = \{pt\}$, so the twisted moduli spaces agree identically with the untwisted ones. The result then follows from independence of the construction on the choice of boundary stabilisation data.
    \end{proof}

\subsection{Relation to ordinary Floer theory}\label{sec: rel ord floer}

We begin with a brief discussion of the appropriate tangential pairs for Liouville domains which may fail to be spin or even orientable.

\begin{lem}\label{lemon 3.8}
    Assume $\Psi$ is an oriented commutative tangential pair. Write $\theta: \Theta \to BO$, $\phi: \Phi \to BS_\pm U$, $f: F \to \widetilde{U/O}$ and $g: \Theta \to \Phi$ be the corresponding maps, and $w_2 \in H^2(BO_{h\cI}; \bZ/2)$ for the universal second Stiefel-Whitney class.

    Then $\theta^*w_2$ lies in the image of $g^*:  H^2(\Phi_{h\cI}; \bZ/2) \to  H^2(\Theta_{h\cI}; \bZ/2)$. Additionally, $g^*$ is injective (so the preimage of $\theta^* w_2$ is unique).
\end{lem}

\begin{defn}\label{defn: univ back clas}
    Let $\Psi = (g: \Theta \to \Phi)$ be an oriented commutative tangential pair. 
    
    We define the \emph{universal background class for $\Psi$} to be the (unique) preimage $\theta^*w_2$ under $g^*$, where $\theta$ is the map $\Theta \to BO$. We write $\omega_2 \in H^2(\Phi_{h\cI}; \bZ/2)$ for this class.
\end{defn}
\begin{rmk}
    Let $\Psi=(\Theta, \Phi)$ be an oriented tangential pair, and $\phi$ a $\Phi$-orientation on $X$. Then Lemma \ref{lemon 3.8} implies any $\Psi$-oriented Lagrangian $L \subseteq X$ is relatively Spin, relative to the background class $\phi^*\omega_2 \in H^2(X; \bZ/2)$.
\end{rmk}
\begin{proof}[{Proof of Lemma \ref{lemon 3.8}}]
    Throughout this proof, we write $H^*(\cdot)$ for $H^*(\cdot; \bZ/2)$, and for an $\cI$-space $X$ (by abuse of notation) we write $X$ for $X_{h\cI}$.

    For any fibration sequence $G \to E \to B$ with simply-connected fibre $G$, an application of the Serre spectral sequence shows that there is a short exact sequence:
    \begin{equation}
        0 \rightarrow H^2(B) \to H^2(E) \to \ker\left(d_3: H^0(B; H^2(G)) \to H^3(B)\right) \to 0
    \end{equation}  
    where $d_3$ is the differential on the $E_3$-page. In particular this implies $g^*$ is injective. Note this exact sequence is functorial.

    Applying this to the homotopy fibration sequences $\widetilde{U/O} \to BO \to BS_\pm U$ and $F \to \Theta \to \Phi$, we obtain a commutative diagram with exact rows:
    \begin{equation}
        \xymatrix{
            0
            \ar[r]
            &
            H^2(BS_\pm U)
            \ar[r]
            \ar[d]_{\phi^*}
            &
            H^2(BO)
            \ar[r]
            \ar[d]_{\theta^*}
            &
            \ker d_3 
            \ar[r]
            \ar[d]_p
            &
            0
            \\
            0
            \ar[r]
            &
            H^2(\Phi) 
            \ar[r]
            &
            H^2(\Theta)
            \ar[r]
            &
            \ker d'_3
            \ar[r]
            &
            0
        }
    \end{equation}
    where $d_3$ and $d'_3$ are the differential pages on the two Serre spectral sequences respectively. We next argue that the map $p: \ker d_3 \to \ker d'_3$ vanishes; the lemma then follows by a diagram chase.

    Since $\widetilde{U/O}$ is simply connected, the Hurewicz map $\bZ/2 \cong \pi_2 \widetilde{U/O} \to H_2(\widetilde{U/O})$ is an isomorphism. $\bZ/2$ has no non-trivial automorphisms so $\pi_1 BS_\pm U$ must act trivially on it, and so $H^0(BS_\pm U; H^2(\widetilde{U/O})) \cong \bZ/2$. It is standard that $H^2(BO) \cong (\bZ/2)^{\oplus 2}$ and $H^2(BS_\pm U) \cong \bZ/2$, so by a dimension count we find that $d_3: \bZ/2 \to \ldots$ must also vanish.

    Since $\Psi$ is oriented, $f_*: \pi_2 F \to \pi_2 \widetilde{U/O}$ vanishes; since $F$ and $\widetilde{U/O}$ are simply connected, $f_*$ must vanish on $H_2(\cdot; \bZ)$ and hence $f^*$ vanishes on $H^2$. Since $p = H^2(\phi^*; f^*): H^0(BS_\pm U; H^2(\widetilde{U/O})) \to H^0(\Phi; H^2(F))$, $p$ must vanish.
\end{proof}

\begin{ex}
     The map $w_2: BSO \to K(\bZ/2,2)$ associated to the second Stiefel-Whitney class is realised by an infinite loop map. (This is because $\pi_2(BSO) = \bZ/2$ is the lowest non-trivial homotopy group, the map is associated to the Postnikov truncation $bso \to \tau_{\leq 2} bso$ in spectra, which is always infinite loop.) There is therefore a commutative tangential pair $(\Theta \to \Phi) = (hofib(BSO \to K(\bZ/2,2)) \to BS_{\pm U})$. In this case the universal class $w_2$ vanishes, and $\Theta$-oriented Lagrangians are graded and $Spin$, recovering a `classical' setting. 
\end{ex}

\begin{rmk}
    Oscar Randal-Williams pointed out to us that, by contrast, $w_2: BO \to K(\bZ/2,2)$ does not deloop. Therefore $w_1^2+w_2$ also does not deloop on $BO$ (since the infinite loop automorphism $-1$ of $BO$ exchanges these two classes).
\end{rmk}

Recall that the $\Phi$-orientation on $X$ means it is also equipped with a grading. Let $\scrF^{loc}(X;\bZ)$ denote the $\bZ$-graded compact Fukaya category whose objects are (exact, graded so Maslov index zero) Lagrangian branes $(L,\xi_\bZ)$ equipped with rank one $\bZ$-local systems $\xi_{\bZ}: \pi_1(L)\to H_1(L;\bZ) \to \bZ/2$.

For $w \in H^2(X;\bZ/2)$ there is a $\bZ$-coefficient $\bZ$-graded Fukaya category $\scrF(X,w;\bZ)$ whose objects are graded exact Lagrangian submanifolds which are relatively spin for the background class $w$; choosing $w$ amounts to fixing a particular coherent orientation scheme for the Fukaya category, see \cite[Section 8]{Seidel:book}. Similarly, we write $\scrF^{loc,\cL}(X,w;\bZ)$ for the $w$-sign-twisted Fukaya category with objects Lagrangians drawn from $\cL$, perhaps equipped with $\bZ$-local systems.

\begin{prop}\label{prop:spectral to ordinary}
    Let $\Psi$ be an oriented (graded) tangential pair. Fix a set of $\Psi$-oriented Lagrangian data $\cL$. Then there  is a fully faithful functor
    \[
    \tau_{\leq 0} \scrF^{loc,\cL}(X;\Psi) \to \scrF^{loc,\cL}(X,\phi^*w_2;\bZ).
    \]
\end{prop}

\begin{proof}
By definition, the morphism group $\scrF_i^{loc}(\xi_L,\xi_K;\Psi) = [\ast[i],\cM^{\xi_L,\xi_K}]^{\Flow_{E_{\Psi}}}$. As observed above, any $\Theta$-oriented Lagrangian is $\phi^*w_2$-relatively spin, so defines an object in $\scrF^{loc,\cL}(X,\phi^*w_2;\bZ)$; Lemma \ref{lem:MLS to classical local system} allows one to extend this association to Lagrangians equipped with local systems. Lemma \ref{lem:truncation to morse complexes} then says that it suffices to compare the signs associated to $\Psi$-oriented zero-manifolds and those arising in classical Floer theory. In both cases, these are determined by the fact that moduli spaces of Floer strips are canonically oriented relative to the orientation lines of the $\overline{\partial}$-operators associated to the puncture data at the ends.

Recall that $H^1(\mathcal{L} U/O;\bZ/2) = \bZ/2 \oplus \bZ/2$ with the generators being the mod $2$ reduction of the Maslov class and the pullback of $w_2$ under the map $S^1 \times \mathcal{L}(U/O) \to U/O \to BO$. In particular, on loops of Maslov index zero (or any fixed Maslov index) the isomorphism class of the determinant line of the $\overline{\partial}$-operator associated to a loop of Lagrangian subspaces is determined by $w_2$, cf \cite[Lemmas 11.7 \& 11.17]{Seidel:book} (note that the group $\mathrm{Pin}_n$ in \emph{op.cit.} is the one, sometimes called $\mathrm{Pin}^+$, which is associated to $w_2 \in H^2(BO_n;\bZ/2)$).
\end{proof}

   The image of the functor is the set of those $\bZ$-local systems on Lagrangians in $\cL$ lying in the image of the map from Lemma \ref{lem:MLS to classical local system}. 

\section{Open-closed map and local systems}\label{sec: OC}

    In this section, we give a topological description of the classes $[L, \xi] \in \Omega^{E_\Psi,\cO\cC}_d(L, \phi|_L)$ from Section \ref{sec: fuk loc sys}.

    Fix some commutative tangential pair $\Psi$, and let $E_\Psi$, $R = \Thom(E_\Psi)$ and $GL_1^\Psi = GL_1(R)$ be as in Section \ref{sec: mon loco sys}.
    \begin{thm}\label{thm: tech OC main}
        Let $(L, \xi)$ be an object of $\scrF(X; \Psi)$: so $L$ is a closed exact $\Psi$-oriented Lagrangian in $X$, and $\xi$ a $\Psi$-local system on $L$. Let $[L, \xi] \in \Omega^{E_\Psi,\cO\cC}_d(L, \phi|_L)$ be the class defined in Definition \ref{def: oc loco}, and $[L] \in \Omega^{E_\Psi,\cO\cC}_d(L, \phi|_L)$ the fundamental class of $L$. Then:
        \begin{equation}
            [L,\xi] = [L] \cap [\eta\xi] \in \Omega_d^{E_\Psi, \cO\cC}(L, \phi_L)
        \end{equation}
        where $\cap$ is the module action from Lemma \ref{lem:module structure for OC bordism}, and $[\eta\xi]$ is the degree 0 $R$-cohomology class given by the following composition:
        \begin{equation}\label{eq: owsetrthdhdrerjgej}
            L \xrightarrow{\xi} B(GL_1^\Psi)_{h\cI} \xrightarrow{\eta_*} (GL_1^\Psi)_{h\cI} \subseteq \Omega^\infty R 
        \end{equation}
        Here $\eta_*$ is the action of the Hopf map $S^3 \to S^2$ on the infinite loop space $GL_1^\Psi$ as in Section \ref{sec: de loop}.
    \end{thm}
    Theorem \ref{thm:main2} is obtained by pushing forwards along the inclusion $L \to X$.
\subsection{Overview}
    We prove Theorem \ref{thm: tech OC main} in several steps.
    \begin{enumerate}
        \item We interpret the class $[L, \xi]$ as the twist of a single manifold in Section \ref{sec: mode twis}, by a map which naturally lands in a cyclic bar-type construction applied to $GL_1^\Psi$ (which we describe in detail in Section \ref{sec: cyc tw}).
        \item In Section \ref{sec: incl cons loop}, we factor this ``cyclic twist map'' through the constant loops, roughly using the fact that for appropriate Floer data, the relevant holomorphic curves approximate Morse trajectories. The free loops we see then look like loops which follow a gradient flow followed by its reverse; we may then apply Lemma \ref{lem: eoghrtoughor}.
        \item In Section \ref{sec: oc proj}, we relate this construction to Schlichtkrull's description of $\eta$: he shows that for an infinite loop space $E$, the composition:
        \begin{equation}
            BE \xrightarrow{const} \cL BE \simeq \Omega BE \times BE \simeq E \times BE \xrightarrow{project} E
        \end{equation}
        is homotopic to $\eta_*$.
    \end{enumerate}
    At most stages of the argument, we must keep track of ``finite-level'' approximations for cyclic bar constructions and similar: our maps land more naturally in e.g. $GL_1^\Psi(n)$ for finite $n$ than $(GL_1^\Psi)_{h\cI}$.

    Throughout this section, we fix $(L, \xi)$ as in Theorem \ref{thm: tech OC main}.
\subsection{Set-up}
    We begin by introducing some notation and intermediary spaces. 
\subsubsection{Labelling}\label{sec: OC lab}
    We write $\cM$ for the flow category $\cM^{LL}$, $\cE$ for the right flow module $\cE^L$ and $\Upsilon$ for the left flow module $\Upsilon_L$.

    By setting $\cM_{xx} := \operatorname{pt}$, $\cM$ becomes a unital topological category, and $\cE$ and $\Upsilon$ are right and left modules over it respectively. 
    \begin{defn}
        We define the \emph{$\cO\cC$-moduli spaces} to be the Floer moduli spaces $\cE_x$, $\Upsilon_x$ and $\cM_{xx'}$, for $x, x' \in \cM$.

        \emph{$\cO\cC$-boundary stabilisation data} $S$ consists of the same data as in Definition \ref{def: stab data}, but only for the $\cO\cC$-moduli spaces; along with one additional piece of data, namely some $n_0=n^S_0 \in \cI$ and (suitably associative) injections $\sqcup_{b \in \Cptd\Delta_\cM}[n_b] \to [n_0]$ for all $\cO\cC$-Floer moduli spaces $\cM$.

        We say $S$ is \emph{saturated} if all the injections (\ref{eq: n stab}), as well as all injections of the form:
        \begin{equation}
            \left(\sqcup_{b \in \Cptd\Delta_{\cE_x}} [n_b]\right) \sqcup \left(\sqcup_{b' \in \Cptd\Delta_{\Upsilon_x}} [n_{b'}]\right) \to [n_0]
        \end{equation}
        are in fact bijections.
    \end{defn}
    \begin{lem}
        Saturated $\cO\cC$-boundary stabilisation data $S$ exists. Furthermore given some $m$, we may choose $S$ so that each $n^S_b \geq m$.
    \end{lem}
    \begin{proof}
        We construct such an $S$ explicitly (though note that we could alternatively prove it more abstractly; it is essentially the same as the proof of Lemma \ref{lem: ex big S}). We choose large integers $A,B,C \geq m$, and set $n_b = A-2C|x|$ for $b \in \Cptd\Delta_{\cE_x}$, $n_b = C(|x|-|x'|)$ for $b \in \Cptd\Delta_{\cM^{LL}_{xx'}}$, and $n_b = B+2C|x|$ for $b \in \Cptd\Delta_{\Upsilon_x}$. We set $n_0 = A+B$. The data of the maps can be chosen similarly. 
    \end{proof}

    Fix some choice $S$ of saturated $\cO\cC$-boundary stabilisation data for $L$.

    \begin{rmk}
        We could have instead not made $\cM$ unital, but then various simplicial sets we will define in what follows would only be semisimplicial sets.
    \end{rmk}

    For convenience, we also set $\Delta_{\cM_{xx}} \in \cD_{11}$ to be the labelled disc with 2 punctures, both labelled by $x$, and both boundary components labelled by $L$. For both $b \in \Cptd\Delta_{\cM_{xx}}$, we set $n_b = 0$. For a sequence $x_0, \ldots, x_k \in \cM$, we define the set $\Cptd\Delta_{x_0 \ldots x_k}$ to be the disjoint union:
    \begin{equation}
        \Cptd\Delta_{x_0 \ldots x_k} := \Cptd\Delta_{\cA_{x_0}} \sqcup \Cptd\Delta_{\cM_{x_0 x_1}} \sqcup \ldots \sqcup \Cptd\Delta_{\cM_{x_{k-1} x_k}} \sqcup \Cptd\Delta_{\cD_{x_k}}
    \end{equation}
    
    For each $n \in \cI$, the finite set $[n]$ has a natural ordering. Pulling back this ordering on $[n_0]$ along the injection 
    \begin{equation}
        \bigsqcup\limits_{b \in \Cptd\Delta_{x_0 \ldots x_k}} [n_b] \to [n_0]
    \end{equation}
    induces an ordering on the LHS (and hence on all its subsets), which we call the \emph{unnatural ordering}. $[n^+_b]$ admits a similar ordering. By the associativity from \ref{eq: n stab ass}, the maps (\ref{eq: n stab}) always respect the unnatural ordering.

    For any subset $U \subseteq \Cptd\Delta_{x_0 \ldots x_k}$, we define 
    \begin{equation}
        \omega=\omega^U_{x_0 \ldots x_k}: \bigsqcup\limits_{b \in U} [n_b] \xrightarrow{\cong} \left[ \sum\limits_{b \in U} n_b\right]
    \end{equation}
    to be the unique order-preserving bijection, with respect to the unnatural ordering on the LHS and the natural ordering on the RHS. 

    For $x,x' \in \cM$, we let $b^+ \in \Cptd\Delta_{\cM_{xx'}}$ be the upper boundary component (going counterclockwise from the output to the input), and $b^-$ the lower boundary component.

    Gluing at $x_i$ defines maps $d_i: \Cptd\Delta_{x_0 \ldots x_k} \to \Cptd\Delta_{x_0 \ldots \hat{x}_i \ldots x_k}$, and inserting a repeat of $x_i$ defines maps $s_i: \Cptd\Delta_{x_0 \ldots x_k} \to \Cptd\Delta_{x_0 \ldots x_i x_i \ldots x_k}$. These satisfy the simplicial identities, and furthermore in this case the $\cO\cC$-boundary stabilisation data determines bijections:
    \begin{equation}\label{eq: stab dat bD simp}
        \xymatrix{
            \bigsqcup\limits_{b \in d_i^{-1}\{b'\}} [n_b] 
            \ar[r] 
            &
            [n_{b'}]
            &
            \bigsqcup\limits_{b \in s_i^{-1}\{b'\}} [n_b] 
            \ar[r] 
            &
            [n_{b'}]
        }
    \end{equation}
    for all $b'$, which satisfy similar identities.

    We partition $\Cptd\Delta_{x_0,\ldots, x_k}$ into two pieces $\Cptd\Delta_{x_0,\ldots, x_k}^\pm$. These are defined by:
    \begin{equation} 
        \Cptd\Delta_{x_0\ldots x_k}^+ := \Cptd\Delta_{\cE_{x_0}} \sqcup \Cptd\Delta_{\cM_{x_0 x_1}}^+ \sqcup \ldots \sqcup \Cptd\Delta_{\cM_{x_{k-1} x_k}}^+ \sqcup \Cptd\Delta_{\Upsilon_{x_k}}
    \end{equation}
    and
    \begin{equation} 
        \Cptd\Delta_{x_0\ldots x_k}^- :=\Cptd\Delta_{\cM_{x_0 x_1}}^- \sqcup \ldots \sqcup \Cptd\Delta_{\cM_{x_{k-1} x_k}}^-
    \end{equation}
    
    We also define $\Cptd\Delta^\sim_{x_0 \ldots x_k}$ to be the quotient set $\Cptd\Delta_{x_0 \ldots x_k} / \Cptd\Delta^+_{x_0 \ldots x_k}$, and write $b^+=b^+_{x_0 \ldots x_k}$ as the image of $\Cptd\Delta^+_{x_0 \ldots x_k}$.  (Compare to Figure \ref{fig:from 2-pointed to 1-pointed}.) Since the $d_i$ and $s_i$ send $\Cptd\Delta^+$ to $\Cptd\Delta^+$, they descend to maps $d_i^\sim$ and $s_i^\sim$ between the $\Cptd\Delta^\sim$, still satisfying the simplicial identities. By construction, the $d_i$, $d^\sim_i$, $s_i$, $s_i^\sim$ commute with the quotient maps. We set:
    \begin{equation}
        n_{b^+_{x_0 \ldots x_k}} := \sum\limits_{b \in \Cptd\Delta^+_{x_0 \ldots x_k}} n_b \in \cI
    \end{equation}
    We also set $\omega^+ := \omega^{\Cptd\Delta^+_{x_0 \ldots x_k}}$; this equips $[n_{b^+_{x_0 \ldots x_k}}]$ with an unnatural ordering. By identifying $n_{b^+_{\ldots}}$ with $\sqcup_{b \in \Cptd\Delta^+_{\ldots}} [n_b]$ via $\omega^+$, the maps (\ref{eq: stab dat bD simp}) induce bijections:
    \begin{equation}\label{eq: stab dat bD simp sim}
        \xymatrix{
            \bigsqcup\limits_{b \in (d_i^\sim)^{-1}\{b'\}} [n_b] 
            \ar[r] 
            &
            [n_{b'}]
            &
            \bigsqcup\limits_{b \in (s_i^\sim)^{-1}\{b'\}} [n_b] 
            \ar[r] 
            &
            [n_{b'}]
        }
    \end{equation}
    which are compatible with the maps (\ref{eq: stab dat bD simp}) in the appropriate way.

    Note that the maps in (\ref{eq: stab dat bD simp}) and (\ref{eq: stab dat bD simp sim}) all preserve the unnatural order on both sides.

    For each $b \in \Cptd\Delta^\sim_{x_0 \ldots x_k}$, we define a simplex $\tau_b \in (S^1)_k$ by setting $\tau_{b^+_{x_0 \ldots x_k}}$ to be $\tau_0$, and proceding inductively by specifying that for $b$ directly counterclockwise of $b'$, $\tau_b=\tau_{l+1}$ when $\tau_{b'} = \tau_l$.

\subsubsection{Models for twist maps}\label{sec: mode twis}

    We define $\overline L$ to be the space $\overline L = B(\cE,\cM,\Upsilon)$; concretely, this is the geometric realisation of the simplicial space:

    \begin{equation}
        [k] \mapsto \bigsqcup\limits_{x_0, \ldots, x_k \in \cM} \cE_{x_0} \times \cM_{x_0 x_1} \times \ldots \times \cM_{x_{k-1} x_k} \times \Upsilon_{x_k}
    \end{equation}

    There is a natural map $p': \overline L \to L$, given by the evaluation map $\Upsilon_x \to L$ for each $x$, cf. Remark \ref{rmk: ups fact}.

    Recall from Definition \ref{def: 95} that we are given maps $\Xi_\cN: \cN \to \prod_{b \in \Cptd\Delta_\cN} GL_1^\Psi(n_b)$. We write $\Xi_b: \cN \to GL_1^\Psi(n_b)$ for the projection to the $b^{th}$ factor.
    \begin{defn}
        We construct a map $\overline \Xi_0: \overline L \to GL_1^\Psi(n_0)$ as follows. 
    
        From the $\cO\cC$-boundary stabilisation data $S$, we obtain an injection $\sqcup_{b \in \Cptd\Delta_{x_0, \ldots, x_k}} [n_b] \to [n_0]$
        which induces a map (as in Remark \ref{rmk: inj prob}):
        \begin{equation}\label{eq: cyc prod tw}
            \prod\limits_{b \in \Cptd\Delta_{x_0 \ldots x_k}} GL_1^\Psi(n_b) \to GL_1^\Psi(n_0)
        \end{equation}
        Applying the maps $\Xi_b$ defines a map from $\cE_{x_0} \times \cM_{x_0x_1} \times \ldots \times \cM_{x_{k-1} x_k} \times \Upsilon_{x_k}$ to the domain of (\ref{eq: cyc prod tw}); composing with (\ref{eq: cyc prod tw}) then gives maps from the $k$-simplices $B(\cE,\cM,\Upsilon)_k \to GL_1^\Psi(n_0)$. The associativity of the pre-twisting data implies these assemble to form a map $\overline \Xi_0: \overline L \to GL_1^\Psi(n_0)$.
    \end{defn}
    Most of the work in the rest of this section will be in proving the following proposition:
    
    \begin{prop}\label{prop: OC tech main}
        The following diagram commutes up to homotopy:
        \begin{equation}\label{eq: OC tech main}
            \xymatrix{
                \overline L
                \ar[r]_{p'}
                \ar[d]_{\overline \Xi_0} 
                &
                L 
                \ar[r]_-{\xi} 
                &
                B(GL_1^\Psi)_{h\cI} 
                \ar[d]_{\eta_*}
                \\
                GL_1^\Psi(n_0) 
                \ar[rr]_{\iota_0} 
                &
                &
                (GL_1^\Psi)_{h\cI}
            }
        \end{equation}
    \end{prop}
\subsection{$\cO\cC$ from twist maps}
    Before proving Proposition \ref{prop: OC tech main}, we first deduce Theorem \ref{thm: tech OC main}.
    
    Let $\hat B(\cE,\cM,\Upsilon) = \colim_{[k] \in \Delta^{op}} B(\cE,\cM,\Upsilon)_k$; here we take the colimit over the simplex category, in contrast to the usual bar construction, which is the homotopy colimit. The natural transformation $\hocolim \to \colim$ provides a natural map $\overline L =B(\cE,\cM,\Upsilon) \to \hat B(\cE,\cM,\Upsilon)$; this is a homotopy equivalence since the inclusion of a boundary face into a manifold with faces is a cofibration. $(\Upsilon \circ \cE)$ is constructed as a subspace of a coequaliser diagram in the form of \cite[Equation (33)]{PS2}, which deformation retracts to $\hat B(\cE,\cM,\Upsilon)$ (via the projection $[0,\eps) \to \{0\}$), so we obtain a map $(\Upsilon \circ \cE) \to \overline L$.

    Let $\widehat{(\Upsilon \circ \cE)}$ be the ``twist'' of $(\Upsilon \circ \cE)$ by the composition $(\Upsilon \circ \cE) \to \overline L \xrightarrow{\overline \Xi_0} GL_1^\Psi(n_0)$; explicitly, this is obtained by taking the zero-section of a generic perturbation to the adjoint map $(\Upsilon \circ \cE) \times S^{n_0} \to \Thom(E_\Psi(n_0))$ as in Section \ref{sec:twisted moduli}; this is naturally $E_\Psi$-oriented relative to $\phi$ in the same way as in Section \ref{sec: E Psi or tw modu}.
    
    \begin{lem}\label{lem: 11 point 8}
        $\widehat{(\Upsilon \circ \cE)}$ is $E_\Psi$-oriented (relative to $\phi$) bordant to $(\Upsilon_\xi \circ \cE^\xi)$ over $L$.
    \end{lem}
    \begin{proof}
        This is essentially by opening up the constructions on both sides: $(\Upsilon_\xi \circ \cE^\xi)$ is obtained from $(\Upsilon \circ \cE)$ by twisting by a twist map $\Xi$ on each $\Upsilon_x \times \cE_x$ (recall $(\Upsilon \circ \cE)$ is constructed essentially by gluing these pieces together, cf. \cite[Section 4]{PS}). But one can alternatively twist after gluing these pieces together, and this is exactly $\widehat{(\Upsilon \circ \cE)}$.
    \end{proof}

    \begin{proof}[Proof of Theorem \ref{thm: tech OC main} from Proposition \ref{prop: OC tech main}]
        Let $\cM_0$ be the bordism from $(\Upsilon \circ \cE)$ to $L$ over $L$ (which is $E_\Psi$-oriented relative to $\phi|_L$) from Theorem \ref{thm: unor OC mod}(2). Let $\hat\cM_0$ be its twist (as in Section \ref{sec:twisted moduli}) by the composition of the map $\cM_0 \to L$ with the map from $L$ to $(GL_1^\Psi)_{h\cI}$ in (\ref{eq: OC tech main}); note that we must choose some homotopy lift of this composition along the inclusion $GL_1^\Psi(n_0) \to (GL_1^\Psi)_{h\cI}$ for this to make sense. This provides a cobordism between the twist $\hat L$ by (a lift of) $\eta_* \circ \xi$ and some twist of $(\Upsilon \circ \cE)$.

        By Proposition \ref{prop: OC tech main}, the restriction of the twist map for $\cM_0$ to $(\Upsilon \circ \cE)$ is homotopic to $\overline \Xi_0$. It follows that both twists of $(\Upsilon \circ \cE)$ are cobordant (over $L$ and with appropriate orientation data), and hence also to $\hat L$. It then follows from Lemma \ref{lem: 11 point 8} that these are also cobordant to $(\Upsilon_\xi \circ \cE^\xi)$, which equals $[L,\xi]$ (by definition). Theorem \ref{thm: tech OC main} then follows from Lemma \ref{lem: tw cap}.
    \end{proof}

    The maps $ev$ from each moduli space to $\Omega L$ (cf. (\ref{eq: ev}) together assemble to form a map:
    \begin{equation}
        ev_\partial: \overline L \to \BB^{cyc}\Omega L
    \end{equation}
    
    \begin{lem}
        For an appropriate choice of Floer data, there is a map $\overline{ev}_\partial:\overline L \to \oBB^{cyc} \Omega L$ making the following diagram commute up to homotopy:
        \begin{equation}\label{eq: cyc tw 4}
            \xymatrix{
                \overline L 
                \ar[d]_{\overline{ev}_\partial}
                \ar[dr]^{ev_\partial} 
                &
                \\
                \oBB^{cyc} \Omega L 
                \ar[r]
                &
                \BB^{cyc} \Omega L
            }
        \end{equation}
    \end{lem}
    \begin{proof}
        For Floer data induced by a $C^2$-small Morse function on $L$, the relevant moduli spaces are homeomorphic to Morse moduli spaces\footnote{A stronger statement, namely that the Floer flow category and Morse flow category can be identified as smooth unoriented flow categories, is proved and used in \cite{PS4}.}, compatibly with evaluation to $L$ up to a $C^0$-small discrepancy. Hence the condition in Definition \ref{def: ev0 bar} holds up to a $C^0$-small discrepancy; homotoping the evaluation maps (inductively in the dimension of the moduli space, so as to do so compatibly with respect to gluing) gives a map homotopic to $ev_\delta$ that lands in $\oBB^{cyc}\Omega L$.
    \end{proof}
    By construction, we find:
    \begin{lem}
        The following diagram commutes up to homotopy:
        \begin{equation}\label{eq: cyc tw 1}
            \xymatrix{
                \overline L
                \ar[r]^{p'}
                \ar[dr]_{\overline{ev}_\partial}
                &
                L
                \\
                &
                \oBB^{cyc}\Omega L
                \ar[u]_{ev_0}
            }
        \end{equation}
    \end{lem}

\subsection{Cyclic twists}\label{sec: cyc tw}
    Recall $S$ is our choice of $\cO\cC$-boundary stabilisation data. Using $S$, we define a `finite-level approximation' of the homotopy colimit $\BB^{cyc}(GL_1^\Psi)_{h\cI}$:
    \begin{defn}
        We define $\BB^{cyc}_S GL_1^\Psi$ to be the realisation of the simplicial space $\left(\BB_S^{cyc}GL_1^\Psi\right)_\bullet$ defined by:
        \begin{equation}\label{eq: def BBcycS}
            \left(\BB_S^{cyc}GL_1^\Psi\right)_\bullet: [k] \mapsto \bigsqcup\limits_{x_0, \ldots, x_k \in \cM} \left(\prod\limits_{b \in \Cptd\Delta_{x_0 \ldots x_k}} GL_1^\Psi(n_b)\right)
        \end{equation}
        and simplicial maps induced by the maps $d_i, s_i$ from Section \ref{sec: OC lab} along with multiplication in $GL_1^\Psi$. 

        Inclusion of 0-simplices defines a map of spaces
        \begin{equation}
            \iota_0: \BB^{cyc}_S GL_1^\Psi \to \BB^{cyc}(GL_1^\Psi)_{h\cI}
        \end{equation}
        The maps $\Xi_b$ together assemble a map of simplicial spaces $\left(\XX^{cyc}_S\right)_\bullet: B(\cE,\cM,\Upsilon)_\bullet \to \left(\BB^{cyc}_SGL_1^\Psi\right)_\bullet$, and hence a map of spaces $\XX^{cyc}_S: \overline L \to \BB_S^{cyc} GL_1^\Psi$.
        
        We define $\XX^{cyc}$ to be the composition of $\XX_S^{cyc}$ with $\iota_0$, so that the following diagram commutes: 
        \begin{equation}\label{eq: cyc tw 2}
            \xymatrix{
                \overline L
                \ar[d]_{\XX^{cyc}_S}
                \ar[dr]^{\XX^{cyc}}
                &
                \\
                \BB^{cyc}_S GL_1^\Psi
                \ar[r]_{\iota_0}
                &
                \BB^{cyc}(GL_1^\Psi)_{h\cI}
            }
        \end{equation} 
    \end{defn}
    \begin{defn}
        We work with the following notational convention. For a map $\prescript{}{2}{F}$ into $\BB^{cyc}(\ldots)$, 
        we write $F$ (i.e. drop the $\prescript{}{2}{\cdot}{}$) for the postcomposition with $\BB_1$ (cf. (\ref{eq:2b1})). Similarly, we add a $\prescript{}{2}{\cdot}{}$ for the precomposition with $\BB_1$ with a map out of $\BB^{cyc}(\ldots)$.

        For example, $\Xi^{cyc}$ is defined to be $\BB_1 \circ \XX^{cyc}$.
    \end{defn}

    \begin{defn}
    
        We define $B^{cyc}_S GL_1^\Psi$ to be the realisation of the simplicial space $\left(B_S^{cyc}GL_1^\Psi\right)_\bullet$ defined by:
        \begin{equation}\label{eq: BcycSbT}
            \left(B_S^{cyc}GL_1^\Psi\right)_\bullet: [k] \mapsto \bigsqcup\limits_{x_0, \ldots, x_k \in \cM} \left( \prod\limits_{b \in \Cptd\Delta^\sim_{x_0 \ldots x_k}} GL_1^\Psi(n_b)\right)
        \end{equation}
        
        Multiplication in $GL_1^\Psi$ and the maps $d_i^\sim$ and $s_i^\sim$ introduced in Section \ref{sec: OC lab} together define the face and degeneracy maps.

        The maps induced by $\omega^+_{x_0 \ldots x_k}$ and multiplication in $GL_1^\Psi$ define maps of simplicial spaces $\BB_1: \left(\BB^{cyc}_SGL_1^\Psi\right)_\bullet \to \left(B_S^{cyc}GL_1^\Psi\right)_\bullet$ and hence maps of spaces $\BB_1: \BB^{cyc}_SGL_1^\Psi \to B^{cyc}_SGL_1^\Psi$.
    \end{defn}
    \begin{lem}
        The following diagram commutes up to homotopy:
        \begin{equation}\label{eq: cyc tw 3}
            \xymatrix{
                \overline L 
                \ar[r]_{ev_\partial}
                \ar[d]_{\XX^{cyc}} 
                &
                \BB^{cyc}\Omega L 
                \ar[d]_{\BB^{cyc}\Omega \xi_L} 
                \\
                \BB^{cyc}(GL_1^\Psi)_{h\cI} 
                \ar@{~>}[r]_{\widetilde\iota_1} 
                &
                \BB^{cyc} \Omega B(GL_1^\Psi)_{h\cI}
            }
        \end{equation}
    \end{lem}
    \begin{proof}
        From how we construct the pre-twisting data $\Xi$ in (\ref{eq: Xi lift prob}), we see that each $\left(\XX^{cyc}_S\right)_k$ is defined to be a solution to the following homotopy lifting problem
        \begin{equation}\label{eq: hom lift prob 2}
            \xymatrix{
                &
                &
                \prod\limits_{b \in \Cptd\Delta_{x_0 \ldots x_k}} GL_1^\Psi(n_b)
                \ar[d]_{\iota_0}
                \\
                &
                &
                (GL_1^\Psi)_{h\cI}^{\Cptd\Delta_{x_0 \ldots x_k}} 
                \ar@{~>}[d]_{\widetilde\iota_1}
                \\
                \cE_{x_0} \times \cM_{x_0 x_1} \times \ldots \times \cM_{x_{k-1} x_k} \times \Upsilon_{x_k} 
                \ar[r]_-{ev_\partial} 
                \ar@{-->}[uurr]_{\left(\XX^{cyc}_S\right)_k}
                &
                \left(\Omega L\right)^{\Cptd\Delta_{x_0 \ldots x_k}} 
                \ar[r]_{\Omega \xi_L} 
                &
                \left(\Omega B (GL_1^\Psi)_{h\cI}\right)^{\Cptd\Delta_{x_0 \ldots x_k}}
            }
        \end{equation}

        (this diagram is just a product of various copies of (\ref{eq: Xi lift prob})).

        Furthermore, these homotopy lifts are compatible with Floer gluings (in the bottom left corner of (\ref{eq: hom lift prob 2})) and with multiplication in the other factors.

        It follows that $\XX^{cyc}_S: \overline L \to \BB^{cyc}_S GL_1^\Psi$ solves a similar homotopy lifting problem; this is exactly (\ref{eq: cyc tw 3}).
    \end{proof}
\subsection{Inclusion of constant loops}\label{sec: incl cons loop}
    \begin{prop}
        The following diagram commutes up to homotopy:
        \begin{equation}\label{eq: step 1}
            \xymatrix{
                \overline L
                \ar[r]_{p'}
                \ar[d]_{\XX^{cyc}}
                &
                L
                \ar[r]_{\xi_L}
                &
                B(GL_1^\Psi)_{h\cI} 
                \ar[d]_{\operatorname{const}}
                \\
                \BB^{cyc}(GL_1^\Psi)_{h\cI} 
                \ar@{~>}[r]_{ \iota_1}
                &
                \BB^{cyc}\Omega B (GL_1^\Psi)_{h\cI} 
                \ar[r]_{\ee}
                &
                \cL B(GL_1^\Psi)_{h\cI} 
            }
        \end{equation}
        Here ${\iota}_1$ is the zig-zag of spaces obtained from Lemma \ref{lem: zig zag} (and applying $\BB^{cyc}$), and $\ee$ is as in Definition \ref{def: cyc bar free loop} (recall we implicitly compose with $\BB_1$).
    \end{prop}
    \begin{proof}
        Consider the following diagram.
        \begin{equation}
            \xymatrix{
                \overline L 
                \ar[r]^{p'} 
                \ar[d]^-{\overline{ev}_\partial}
                \ar@/^-2.0pc/[dd]_{\XX^{cyc}}
                \ar@{}[dr]^(.4){(\ref{eq: cyc tw 1})}
                &
                L
                \ar[dr]^{\operatorname{const}}
                \ar[r]^{\xi_L} 
                \ar@{}[d] | {(\ref{eq: bar pre 1})}
                &
                B(GL_1^\Psi)_{h\cI} 
                \ar@/^2.0pc/[dd]^{\operatorname{const}}
                \\
                \oBB^{cyc}\Omega L 
                \ar[r]
                \ar[ur]_{ev_0}
                \ar@{}[dr] | {(\ref{eq: cyc tw 4})+(\ref{eq: cyc tw 3})}
                &
                \BB^{cyc}\Omega L 
                \ar[r]_{\ee} 
                \ar[d]^{\BB^{cyc}\Omega\xi_L}
                &
                \cL L 
                \ar[d]_{\cL \xi_L} 
                \\
                \BB^{cyc}(GL_1^\Psi)_{h\cI} 
                \ar@{~>}[r]_{\widetilde \iota_1}
                &
                \BB^{cyc}\Omega B (GL_1^\Psi)_{h\cI} 
                \ar[r]_{\ee}
                &
                \cL B(GL_1^\Psi)_{h\cI}
            }
        \end{equation}
        All but two of the subdiagrams have already been show to commute (up to homotopy); we have indicated in the interior of each subdiagram where this has been shown. It is clear that the other two subdiagrams commute. The outer square is exactly (\ref{eq: step 1}), so we are done.
    \end{proof}
    Heuristically Proposition \ref{sec: incl cons loop} says that some model for the twist map (going down and then across) factors through the inclusion of constant loops.
\subsection{Projection}\label{sec: oc proj}

    We will use Proposition \ref{prop: 331} to identify the Hopf map. To do so, we must pass from $\cI$-spaces to $\Gamma$-spaces. Unfortunately, this involves passing through the equivalence $GL_1^{\Psi,\Gamma}(S^1_+) \to V^{cyc}GL_1^\Psi$, which does not (to our knowledge) admit a concrete homotopy inverse. We construct an explicit homotopy lift of the map $B^{cyc}_S GL_1^\Psi \to B^{cyc}(GL_1^\Psi)_{h\cI} \to V^{cyc}GL_1^\Psi$ along this equivalence; roughly, this uses the fact that $B^{cyc}_S GL_1^\Psi$ is built out of finitely many pieces, whereas $B^{cyc}(GL_1^\Psi)_{h\cI}$ is a homotopy colimit over infinitely many terms. Slightly more precisely, we exploit the unnatural order to build various functors between categories we index homotopy colimits over.
    \begin{defn}
        For $x_0, \ldots, x_k \in \cM$, we define an object $\theta_{x_0 \ldots x_k} \in \cS((S^1_+)_k)$ as follows (recall objects of this category are functors $\cP({(S^1_+)_k}) \to \cI$ satisfying a condition, cf. Definition \ref{def: sS cat}).

        For $1 \leq l \leq k$, we set $n_{b_l}$ to be $n_b$ for $b \in \Cptd\Delta^-_{\cM_{x_{l-1}x_l}}$ the unique element, and we set $n_{b_0}$ to be $n^+_{x_0 \ldots x_k}$.
        
        Let $U \subseteq \overline{\left(S^1_+\right)_k}$. 
        We set:
        \begin{equation}
            \theta_{x_0 \ldots x_k}(U) := \sum\limits_{\tau_j \in U} n_{b_j} \in \cI
        \end{equation}

        For an inclusion $i: U \to U'$, the corresponding map $\theta_{x_0 \ldots x_k}(U) \to \theta_{x_0 \ldots x_k}(U')$ is defined so that the following diagram commutes:
        \begin{equation}
            \xymatrix{
                \left[\sum\limits_{\tau_j \in U} {n_{b_j}}\right] 
                \ar[rr]_{\theta_{x_0 \ldots x_k}(i)} 
                &&
                \left[\sum\limits_{\tau_j \in U'} {n_{b_j}}\right]
                \\
                \bigsqcup\limits_{\tau_j \in U} \left[{n_{b_j}}\right] 
                \ar[u]_{\omega^U}^\cong
                \ar@{^{(}->}[rr] 
                &&
                \bigsqcup\limits_{\tau_j \in U'} \left[{n_{b_j}}\right]
                \ar[u]_{\omega^{U'}}^\cong
            }
        \end{equation}  
        where the bottom right arrow is the identity $\left[{n_{b_j}}\right] \to \left[{n_{b_j}}\right]$ for all $\tau_j \in U$.
    \end{defn}

    \begin{lem}\label{lem: simp theta}
        The following equalities hold:
        \begin{enumerate}
            \item $(d_i)_*\theta_{x_0 \ldots x_k} = \theta_{x_0 \ldots \hat x_i \ldots x_k}$
            \item $(s_i)_*\theta_{x_0 \ldots x_k} = \theta_{x_0 \ldots x_i x_i \ldots x_k}$
        \end{enumerate}
        where $d_i: (S^1_+)_k \to (S^1_+)_{k-1}$ and $s_i: (S^1_+)_k \to (S^1_+)_{k+1}$ are the face and degeneracy maps respectively.
    \end{lem}
    \begin{proof}
        This is essentially by opening up definitions. On objects, this follows from the saturatedness condition. On morphisms, this follows from existence of the unnatural ordering and the fact that for $U \subseteq U'$, morphisms $\theta_{\ldots}(U) \to \theta_{\ldots}(U')$ are constructed to be unnatural order-preserving.
    \end{proof}
    \begin{rmk} 
        The reason we imposed saturatedness on $S$ is to ensure \emph{equalities} of functors in Lemma \ref{lem: simp theta}. Without imposing the saturatedness condition, we would instead have that $(d_i)_*\theta_{x_0 \ldots x_k}$ and $\theta_{x_0 \ldots \hat x_i \ldots x_k}$ are related by a natural transformation, which would necessitate incorporating more coherence technology into the definition of $j$.
    \end{rmk}

    \begin{defn}
        We define a map of simplicial spaces $j_\bullet: \left(B_S^{cyc}GL_1^\Psi\right)_\bullet \to \left(GL_1^{\Psi,\Gamma}(S^1_+)\right)_\bullet$, inducing a map $j$ on their realisations, as follows.

        On $k$-simplices, $j_k$ is defined on the component corresponding to $x_0, \ldots, x_k$ (cf. (\ref{eq: BcycSbT})) to be the composition:

        \begin{equation}
            \prod\limits_{b \in \Cptd\Delta^\sim_{x_0 \ldots x_k}} GL_1^\Psi(n_b)
            \xrightarrow{=}
            \prod\limits_{\tau_l \in \overline{(S^1_+)_k}} GL_1^\Psi\left( \theta_{x_0 \ldots x_k}(\{\tau_l\})\right) 
            \xrightarrow{\iota_0} 
            GL_1^{\Psi,\Gamma} \left(\left(S^1_+\right)_k\right)
        \end{equation}
        where the first map comes from identifying $n_b = \theta_{x_0 \ldots x_k}(\{\tau_b\})$, and the second map $\iota_0$ is the inclusion of 0-simplices corresponding to the object $\theta_{x_0 \ldots x_k} \in \cS({(S^1_+)_k})$ in the category which $GL_1^{\Psi,\Gamma}$ is a homotopy colimit over (cf. Definition \ref{def: Gam spc from I mon}).
    \end{defn}
    
    It follows from Lemma \ref{lem: simp theta} that $j_\bullet$ is compatible with the face and degeneracy maps and so does define a map of simplicial spaces. 
    \begin{lem}
        The following diagram commutes up to homotopy:
        \begin{equation}\label{eq: 10.18}
            \xymatrix{
                B^{cyc}_S GL_1^\Psi
                \ar[r]_{\iota_0}
                \ar[d]_j
                &
                B^{cyc}(GL_1^\Psi)_{h\cI} 
                \ar[d] 
                \\
                GL_1^{\Psi,\Gamma}(S^1_+) 
                \ar[r]
                &
                V^{cyc} GL_1^\Psi
            }
        \end{equation}
    \end{lem}
    \begin{proof}
    
        We construct an explicit 3-step homotopy. 
        
        As a first step, we define an intermediate homotopy  $\{F_t\}_{t \in [0,1]}: B^{cyc}_S GL_1^\Psi \to V^{cyc}GL_1^\Psi$, for $j=0,1$, as follows. 

        Here we use the observation in Remark \ref{rmk: I V} that any isometric embedding $\bR^k \to \bR^l$ induces a map $GL_1^\Psi(k) \to GL_1^\Psi(l)$. For any $n \in \cI$ and $t \in [0,1]$, let $f^n_t: \bR^n \to \bR^{2n}$ be the isometric embedding given by the block matrix $\begin{pmatrix}
            \cos \frac \pi 2 t \cdot \operatorname{Id}_{\bR^n} \\
            \sin \frac \pi 2 t \cdot \operatorname{Id}_{\bR^n}
        \end{pmatrix}$ with respect to the decomposition $\bR^{2n}=\bR^n\oplus\bR^n$. Note $f^n_0$ and $f^n_1$ are the inclusions of the first $n$ and the last $n$ co-ordinates respectively, so they're exactly the linear maps induced by the two corresponding inclusions $[n] \to [2n]$.
        
        On the component of the $k$-simplices corresponding to $x_0, \ldots, x_k \in \cM$, we define $F_t$ to be the map which is induced factorwise by $f_t^{n_b}$:

        \begin{equation}
            \prod\limits_{b \in \Cptd\Delta^\sim_{x_0 \ldots x_k}} GL_1^\Psi(n_b) \to \prod\limits_{b \in \Cptd\Delta^\sim_{x_0 \ldots x_k}} GL_1^\Psi(2n_b)
        \end{equation}
        followed by the inclusion of the RHS into the homotopy colimit (cf. (\ref{eq: Vcyc})) corresponding to the object $(n_{b_0}, \ldots, n_{b_k}, \theta_{x_0, \ldots, x_k}) \in \cI^{k+1} \times \cS((S^1_+)_k)$, where the $b_i$ are labelled going counterclockwise, beginning with $b_0 := b^+$. This specifies $F_t$ over the $k$-simplices, and is compatible with the face and degeneracy maps and so induces a map on realisations.

        In particular, note that the endpoints of the homotopy $F_0$ and $F_1$ use only functoriality of $GL_1^\Psi(\cdot)$ under maps in $\cI$, rather than arbitrary isometric embeddings.

        By existence of the family $\{F_t\}$, $F_0$ and $F_1$ are homotopic.
        \begin{claim}\label{claim: 1019}
            $F_0$ is homotopic to the composition of $\iota_0$ with the right vertical arrow in (\ref{eq: 10.18}).
        \end{claim}
        \begin{proof}[Proof of claim]

            Choose a component of the $k$-simplices of the domain corresponding to $x_0,\ldots,x_k \in \cM$. Let $\cC=\{\pt\}$ and $\cD=\cI^{k+1}\times \cS((S^1_+)_k)$. The two maps we wish to compare are of the form
            \begin{equation}
                \prod_{b \in \partial^\sim\bD_{x_0 \ldots x_k}} GL_1^\Psi(n_b) \to \underset\cD\hocolim GL_1^\Psi(\ldots) \times \ldots \times GL_1^\Psi(\ldots)
            \end{equation}
            We think of the left hand side as a homotopy colimit over the appropriate constant functor $\cC \to Spc_*$. The two maps are induced (as in Definition \ref{def: hocolim func}, and using the same notation) by functors $A,A':\cC \to \cD$ along with appropriate natural transformations. Here $A$ sends $\pt$ to $(n_{b_0},\ldots,n_{b_k},0)$ and $A'$ sends $\pt$ to $(n_0,\ldots,n_k,\theta_{x_0\ldots x_k})$. A natural transformation $U: A \to A'$ is given by the map $f_0$ factorwise on the RHS in (\ref{eq: Vcyc}) corresponding to $A(\pt)$ and $A'(\pt)$. 

            Lemma \ref{lem: hocolim hom} then constructs the required homotopy. A direct check shows this homotopy is compatible with the face and degeneracy maps, and so defines a homotopy of simplicial spaces and therefore a homotopy on realisations.
        \end{proof}
        An identical argument to the proof of Claim \ref{claim: 1019} shows that $F_1$ is homotopic to the composition of $j$ with the lower horizontal arrow in (\ref{eq: 10.18}). These three homotopies together provide the required homotopy.
    \end{proof}
    \begin{defn}
        We define a map $r: B^{cyc}_S GL_1^\Psi \to GL_1^\Psi(n_0)$ as follows. On each component of the $k$-simplices $(B^{cyc}_S GL_1^\Psi)_k$, we define maps
        \begin{equation}
            \prod\limits_{b \in \Cptd\Delta^\sim_{x_0 \ldots x_k}} GL_1^\Psi(n_b) \to GL_1^\Psi(n_0)
        \end{equation}
        to be induced by the maps induced by $S$:
        \begin{equation}
            \bigsqcup\limits_{b \in \Cptd\Delta^\sim_{x_0 \ldots x_k}} [n_b] \to [n_0]
        \end{equation}
        These are compatible with the face and degeneracy maps and hence induce the map $r$ desired.
    \end{defn}
    \begin{lem}
        The following diagram commutes:
        \begin{equation}\label{eq: 1034}
            \xymatrix{
                \overline L
                \ar[r]^{\Xi^{cyc}}
                \ar[dr]_{\overline \Xi_0}
                &
                B_S^{cyc}GL_1^\Psi
                \ar[d]^r
                \\
                &
                GL_1^\Psi(n_0)
            }
        \end{equation}
    \end{lem}
    \begin{proof}
        This follows from opening up the definitions, and using the associativity of the boundary stabilisation data.
    \end{proof}
    \begin{lem}
        The following diagram commutes: 
        \begin{equation}\label{eq: 1035}
            \xymatrix{
                &
                B_S^{cyc}GL_1^\Psi 
                \ar[r]^j
                \ar[dl]_r 
                &
                GL_1^{\Psi, \Gamma}(S^1_+)
                \ar[d]^{S^1_+ \to S^0} 
                \\
                GL_1^\Psi(n_0)
                \ar[r]_{\iota_0}
                &
                (GL_1^\Psi)_{h\cI}
                \ar[r]_=
                &
                GL_1^{\Psi,\Gamma}(S^0)
            }
        \end{equation}
    \end{lem}
    \begin{proof}
        This follows from opening up the definitions, and using the fact that the map $S^1_+ \to S^0$ sends $\theta_{x_0 \ldots x_k}$ to the functor $\cP(S^0) \to \cI$ sending the unique nontrivial object to $n_0$.
    \end{proof}
    \begin{proof}[Proof of Proposition \ref{prop: OC tech main}]
        Consider the following diagram:
        \begin{equation}
            \xymatrix{
                \overline L
                \ar[r]^{p'}
                \ar[d]^{\overline\Xi_S^{cyc}}
                \ar[dr]^{\Xi^{cyc}}
                \ar@/^-2.0pc/[dd]_{\overline{\Xi}_0}
                &
                L
                \ar[r]^{\xi_L}
                \ar@{}[dr] | {(\ref{eq: step 1})+(\ref{eq: bar pre 1 revisited})}
                &
                B(GL_1^\Psi)_{h\cI} 
                \ar[d]^{\operatorname{const}}
                \\
                B^{cyc}_SGL_1^\Psi 
                \ar[r]_{\iota_0} 
                \ar[d]_r
                \ar[dr]_j
                &
                B^{cyc}(GL_1^\Psi)_{h\cI} 
                \ar[r]_q^\simeq
                \ar[dr]
                \ar@{}[d] | {(\ref{eq: 10.18})}
                &
                \cL B(GL_1^\Psi)_{h\cI} 
                \ar@/^2.0pc/[dd]
                \\
                GL_1^\Psi(n_0) 
                \ar[d]_{\iota_0} 
                \ar@{}[dr] | {(\ref{eq: 1035})}
                &
                GL_1^{\Psi,\Gamma}(S^1_+) 
                \ar[r]^\simeq
                \ar[d]^{S^1_+\to S^0} 
                \ar[dr]
                &
                V^{cyc}GL_1^\Psi 
                \\
                (GL_1^\Psi)_{h\cI} 
                \ar[r]_= 
                &
                GL_1^{\Psi,\Gamma}(S^0)
                &
                \cL  GL_1^{\Psi,\Gamma}(S^1)
            }
        \end{equation}
        All subdiagrams have already been shown to commute up to homotopy. Most are indicated. The remaining ones are the triangles near the top left, which are (\ref{eq: 1034}) and (\ref{eq: cyc tw 2}); and the irregular pentagon on the right, which is (\ref{eq: 330}).

        Proposition \ref{prop: 331} then shows that the composition along the right hand side (inverting equivalences as appropriate) is exactly $\eta_*$; it follows that the outside of the above diagram is exactly (\ref{eq: OC tech main}).
    \end{proof}
\subsection{Proof of Corollary \ref{2}}\label{proof 2}
\begin{proof}[Proof of Corollary \ref{2}(1)]
    Corollary \ref{2}(1) follows from the fact that $BGL_1(R_\Psi)$ is connected, so the composition (\ref{eq: owsetrthdhdrerjgej}) lands in the path-component of the unit in $\Omega^\infty R_\Psi$.
\end{proof}
\begin{proof}[Proof of Corollary \ref{2}(3)]
    Corollary \ref{2}(3) follows from the fact that $\eta_*$ is 2-torsion on homotopy groups.
\end{proof}

    Recall from Remark \ref{rmk:A} that since $BGL_1(R_\Psi)$ and $GL_1(R_\Psi)$ correspond to infinite loop spaces, they give rise to connective spectra $bgl_1(R_\Psi)$ and $gl_1(R_\Psi)$ respectively, and that $bgl_1(R_\Psi) \simeq \Sigma gl_1(R_\Psi)$. 

    Applying the $\Sigma^\infty$-$\Omega^\infty$ adjunction to the composition of maps of spaces:
    \begin{equation}
        L \xrightarrow{\xi} BGL_1(R_\Psi) \xrightarrow{\eta_*} GL_1(R_\Psi)
    \end{equation}
    gives a composition of maps of spectra:
    \begin{equation}
        \Sigma^\infty_+ L \xrightarrow{\xi} bgl_1(R_\Psi) \xrightarrow{\eta} gl_1(R_\Psi)
    \end{equation}

\subsubsection*{$p$-localisation}
    Fix $p$ an odd prime. We use two forms of $p$-localisation.
    \begin{enumerate}
        \item A spectrum $Y$ is \emph{$p$-local} if the groups $\pi_*Y$ are uniquely $p$-divisible.\footnote{See \cite[Proposition 2.4]{Bousfield:localisation} for a comparison to a potentially more familiar definition of $p$-locality.}

        \item A space $Y$ is \emph{$p$-local} if, whenever $U \to V$ is a map of spaces which induces an isomorphism on homology with $\bZ_{(p)}$-coefficients, $[V, Y] \to [U, Y]$ is a bijection.
    \end{enumerate}
    There are `$p$-localisation functors' $Y \mapsto Y_p$, from (a).~spectra to $p$-local spectra and (b).~nilpotent spaces to $p$-local spaces, along with (in both cases) a natural transformation $Y \to Y_p$. In both cases, the homotopy groups of $Y_p$ are the $p$-localisations of those of $Y$ (cf. \cite[Section 2]{Bousfield:localisation}, \cite[Theorem 6.1.2]{May-Ponto}). 
    
    Since $\Omega^\infty$ sends spectra to spaces which have nilpotent components and $\Omega^\infty$ preserves homotopy groups for connective spectra, we have:
    \begin{lem}\label{lem: p p}
        $p$-localisation commutes with $\Omega^\infty$. More precisely, if $Y$ is a spectrum, then $\Omega^\infty(Y_p) \simeq (\Omega^\infty Y)_p$.
    \end{lem}
\begin{proof}[Proof of Corollary \ref{2}(2)]
    Consider the homotopy commutative diagram of spectra:
    \begin{equation}\label{eq: irejgpirug0tghsegbsebrg}
        \xymatrix{
            bgl_1(R_\Psi)
            \ar[r]
            \ar[d]
            &
            gl_1(R_\Psi)
            \ar[d]
            \\
            bgl_1(R_\Psi)_p
            \ar[r]
            &
            gl_1(R_\Psi)_p
        }
    \end{equation}
    The bottom horizontal map may be identified with $\eta_*: \Sigma gl_1(R_\Psi)_p \to gl_1(R_\Psi)_p$. Since both these spectra are $p$-local, the action of $\pi_* \bS$ factors through an action by the $p$-local sphere $\pi_* \bS_p$. However, since $\eta$ is 2-torsion, its image in $\pi_* \bS_p$ vanishes, and so this bottom horizontal map is trivial.

    By Lemma \ref{lem: p p} and the above discussion, delooping (\ref{eq: irejgpirug0tghsegbsebrg}) we obtain a homotopy commutative diagram of spaces:
    \begin{equation}
        \xymatrix{
            BGL_1(R_\Psi)
            \ar[r]_{\eta_*}
            \ar[d]
            &
            GL_1(R_\Psi)
            \ar[d]
            \\
            BGL_1(R_\Psi)_p
            \ar[r]_{\mathrm{const}}
            &
            GL_1(R_\Psi)_p
        }
    \end{equation}
    Combining with the following commutative diagram of spaces:
    \begin{equation}
        \xymatrix{
            GL_1(R_\Psi)
            \ar[r]
            \ar[d]
            &
            \Omega^\infty R_\Psi
            \ar[d]
            \ar[dr]
            &
            \\
            GL_1(R_\Psi)_p
            \ar[r]
            &
            (\Omega^\infty R_\Psi)_p
            \ar[r]_-{=}
            &
            \Omega^\infty ((R_\Psi)_p)
        }
    \end{equation}
    (where the arrow labelled $=$ comes from Lemma \ref{lem: p p}) 
    we may conclude that the image of $[\eta\xi]$ in the $p$-localisation of $R^0_\Psi(L)$ agrees with the unit $1$ (since it factors through a constant map) for all odd $p$, and therefore that $1-[\eta\xi] \in R^0_\Psi(L)$ is $2^l$ torsion for some $l$.
\end{proof}

\bibliographystyle{myamsalpha}
\bibliography{Refs.bib}{} 

\end{document}